\documentclass[reqno,a4paper,11pt,oneside]{amsart}
\usepackage{xcolor}
\usepackage[final]{graphicx}
\usepackage{fixltx2e}
\usepackage{float}
\usepackage{amssymb,amsmath,amsthm}
\usepackage{mathtools}
\usepackage{fixmath}
\usepackage{caption}
\usepackage{caption}
\usepackage{subcaption}
\usepackage{enumitem}
\usepackage{units}

\usepackage{booktabs}
\usepackage[T1]{fontenc}
\usepackage{lmodern}
\usepackage{algorithm}
\usepackage{algorithmic}
\usepackage{amssymb,amsmath,exscale}
\usepackage[final,stretch=10]{microtype}

\DeclareMathOperator{\supp}{supp}

\DeclareMathOperator*{\argmin}{arg\,min}

\newcommand{\Cbochnorm}[1]{\|#1\|_{\mathcal{C}}}

\newcommand{\F}{\mathcal{F}}
\newcommand{\N}{\mathbb{N}}
\newcommand{\R}{\mathbb{R}}
\newcommand{\de}{\mathrm{d}}

\newcommand{\wnorm}[1]{\|#1\|_{W_T}}

\newcommand{\tr}[1]{tr(C(1)^{-1})}

\newcommand{\eps}{\varepsilon}

\newcommand{\vertiii}[1]{{\left\vert\kern-0.25ex\left\vert\kern-0.25ex\left\vert #1
    \right\vert\kern-0.25ex\right\vert\kern-0.25ex\right\vert}}

\makeatletter
\newcommand*\rel@kern[1]{\kern#1\dimexpr\macc@kerna}
\newcommand*\widebar[1]{%
  \begingroup
  \def\mathaccent##1##2{%
    \rel@kern{0.8}%
    \overline{\rel@kern{-0.8}\macc@nucleus\rel@kern{0.2}}%
    \rel@kern{-0.2}%
  }%
  \macc@depth\@ne
  \let\math@bgroup\@empty \let\math@egroup\macc@set@skewchar
  \mathsurround\z@ \frozen@everymath{\mathgroup\macc@group\relax}%
  \macc@set@skewchar\relax
  \let\mathaccentV\macc@nested@a
  \macc@nested@a\relax111{#1}%
  \endgroup
}
\makeatother

\numberwithin{equation}{section}

\definecolor{darkred}{rgb}{.7,0,0}

\definecolor{green}{rgb}{0,0.7,0}

\usepackage{pgfplots}
\usepackage{tikz}

\usepackage[colorlinks=true,linkcolor=blue!70!black,urlcolor=black,citecolor=green!60!black,pdftex]{hyperref}

\theoremstyle{plain}
\newtheorem{theorem}{Theorem}
\newtheorem{prop}[theorem]{Proposition}
\newtheorem{lemma}[theorem]{Lemma}
\newtheorem{coroll}[theorem]{Corollary}
\theoremstyle{definition}

\newtheorem{assumption}{Assumption}
\theoremstyle{remark}
\newtheorem{remark}{Remark}

\newcommand{\tcb}{\textcolor{black}}
\newcommand{\black}{ \color{black} }

\newcommand{\bcdot}{ {\boldsymbol{\cdot}}}
\usepackage[ulem=normalem]{changes}
\usepackage{todonotes}
\begin{document}
\title[]{Optimal feedback control of dynamical systems via value-function approximation}

\author {Karl Kunisch\textsuperscript{$*$}}
\thanks{\textsuperscript{$*$}University of Graz, Institute of Mathematics and Scientific
Computing, Heinrichstr. 36, A-8010 Graz, Austria and Johann Radon Institute for Computational and Applied Mathematics
(RICAM), Austrian Academy of Sciences, Altenberger Stra\ss{}e 69, 4040 Linz, Austria, ({\tt
karl.kunisch@uni-graz.at}).}
\author{Daniel Walter\textsuperscript{$\dagger$}}
\thanks{\textsuperscript{$\dagger$}Institut f\"ur Mathematik, Humboldt-Universit\"at~zu~Berlin, Rudower Chaussee 25, 10117 Berlin, Germany,({\tt
daniel.walter@hu-berlin.de}).}
\maketitle

\begin{abstract}
A self-learning approach for optimal feedback gains for finite-horizon  nonlinear continuous time control systems is proposed and analysed. It relies on parameter dependent approximations to the optimal value function obtained from a family of universal approximators. The cost functional for the training of an approximate optimal feedback law incorporates two main features. First, it contains the average over the objective functional values of the parametrized feedback control for an ensemble of initial values. Second, it is adapted to exploit the relationship between the maximum principle and dynamic programming.
Based on universal approximation properties,  existence, convergence and first order optimality conditions for optimal neural network feedback controllers are proved.
\end{abstract}



\section{Introduction} \label{sec:Introduction}

In this paper we focus on optimal feedback control for problems of the form
\begin{equation} \label{def:refproblem}\tag{$P$}
\left\{
\begin{aligned}
\quad &\inf_{y,u} J(y,u) \coloneqq \frac{1}{2}\int_{0}^T \left
( |Q_1(y(t)-y_d(t))|^2+ \beta |u(t)|^2 \right)~\de t  +\frac{1}{2}|Q_2(y(T)-y^T_d)|^2\\
&s.t. \quad \dot{y}= {f}(y)+g(y)u, \quad y(0)=y_0, \text{ and } u \in L^2(0,T; \R^m),
\end{aligned}
\right.
\end{equation}
with nonlinear dynamics described by $f:[0,T]\times \R^n \to \R^n$. The system can be influenced by choosing a control input~$u$ which enters through a control operator $g:\R^n\to \R^{n \times m}$. We assess the perfomance of a given control by its objective functional value which comprises the (weighted) distance between the associated state trajectory~$y$ and a given desired state~$y_d$ as well as the norm of the control for some cost parameter~$\beta>0$. The weighting matrices $Q_i$, for $i=1,2$, are assumed to be symmetric positive semi-definite. Searching for an optimal control $u^*$  in feedback form requires to find a function $F^*:[0,T]\times \R^n \to \R^m$ such that
$$
u^*(t)= F^*(t,y^*(t)), \text{ for } t\in (0,T).
$$
Here $(u^*,y^*)$ denotes an optimal control-trajectory pair associated to \eqref{def:refproblem}.
Under appropriate conditions, see e.g. \cite{FS06}, the feedback mapping can be expressed as
\begin{equation}\label{eq:intro1}
F^*(t,y) = -\frac{1}{\beta}g^\top(t,y) \partial_y V^*(t,y),
\end{equation}
where $V^*$ stands for the value function associate to \eqref{def:refproblem}, i.e. for $(T_0,y_0)\in [0,T]\times \R^n$:
$$
V^*(T_0,y_0)= \min_{y,u} J_{T_0}(y,u), \text{ subject to } \dot{y}= {f}(y)+g(y)u, \quad y(T_0)=y_0,
$$
and
$$
J_{T_0}(y,u)=\frac{1}{2}\int_{T_0}^T \left
( |Q_1(y(t)-y_d(t))|^2+ \beta |u(t)|^2 \right)~\de t  +\frac{1}{2}|Q_2(y(T)-y^T_d)|^2.
$$
The value function $V^*$ satisfies a Hamilton-Jacobi-Bellman (HJB) equation which is a time-dependent first order hyperbolic equation of spatial dimension $n$. Numerical realisations, therefore, are plagued by the curse of dimensionality. Indeed  a direct solution of the HJB equation already becomes computationally prohibitive for moderate dimensions~$n$.

 Therefore, for practical realization, the interest in alternative techniques arises. In many situations of practical relevance researches have relied on linear approximations to the nonlinear dynamical system and have treated the resulting linear-quadratic problem by Riccati techniques. Much research has concentrated on validating this approach locally around a reference trajectory. Globally such a strategy may fail, see for instance \cite{dkk19, KW2020}.

In this paper we follow an approach, possibly first proposed in \cite{KW2020}, circumventing the construction of the value function on the basis of solving the HJB equation. Rather the feedback mapping is constructed by an unsupervised self-learning technique. In practice, this requires the approximation of~$V^*$ by a family of functions~$V_\theta$ which are parametrized by a finite dimensional vector~$\theta$ and satisfy a uniform approximation property. Possible families of universal approximators include, e.g., neural networks or piecewise polynomial approximations. Subsequently, in view of \eqref{eq:intro1}, we introduce the corresponding feedback law
\begin{equation}\label{eq:intro2}
F_\theta(t,y) = -\frac{1}{\beta}g^\top(y) \partial_y V_\theta(t,y), \text{ for } (t,y)\in [0,\infty) \times \R^n,
\end{equation}
as approximation to $F^*$. An "optimal" parametrized feedback law is then determined by a variant of the following self-learning, structure preserving, variational problem:
\begin{equation}\label{eq:intro3}
\begin{array}l
\min_{\theta}\; J(y,\F_{\theta}(y))\\[1.5ex]
\qquad \; +\frac{1}{2}\int^T_0 {\gamma_1} |V_{\theta}(t,y(t))-J_t(y,F_{\theta}(\cdot,y))|^2 + {\gamma_2}|\partial_y V_{\theta}(t,y(t))-p(t)|^2~\mathrm{d}t +\frac{\gamma_\eps}{2}  {|\theta|^2}\\[1.8ex]
\text{s.t. }   \quad \dot{y}= {f}(y)+g(y) F_{\theta}(y), \quad y(0)=y_0, \quad p(T)= Q^\top_2 Q_2(y(T)-y^T_d)\\[1.5ex]
\qquad \;- \dot p= {f}(y)^\top p+\lbrack D {g}(y)^\top F_\theta (y)\rbrack p+Q_1^\top Q_1(y-y_d).
\end{array}
\end{equation}
In this problem, minimization with respect to $u$ is replaced by minimizing with respect to the parameters $\theta$ which characterize $V_\theta$ and~$F_\theta$. The cost functional of problem \eqref{eq:intro3} consists of four parts: The first term represents the objective functional of \eqref{def:refproblem} where the control~$u$ is replaced by the closed loop expression~$F_\theta(y)$. The next two terms realize the fact that $V_\theta$ is constructed as  approximation to the value function associated to $\eqref{def:refproblem}$ and exploit the  well-known property that, under certain conditions, the gradient of the value function coincides with the solution of a suitable adjoint equation, see e.g. \cite[page 21]{FS06}. The final term penalizes the norm of the structural parameters.
We point out that $V_\theta$ and $F_\theta$ are learned along the
orbit $\mathcal {O}= \{y(t;y_0): t\in (0,\infty)\}$ within the state space $\R^n$.
To  accommodate the case that one trajectory does not provide enough information, we
propose to involve  an ensemble of orbits departing from a set $Y_0$ of initial
conditions, and to reformulate problem \eqref{eq:intro3} accordingly. This will be done in Section 4 below.

In our earlier work on learning a feedback function, \cite{KW2020}, we considered  infinite horizon optimal control problems. In that case, the time-dependent HJB equation results in a stationary one. There we had not yet incorporated the structure preserving terms involving $V_\theta$ and $\partial_y V_\theta$ into the cost. Moreover we directly constructed an approximation  $F_\theta$ to the vector valued function $F^*$, rather than approximating the scalar valued function $V^*$ and subsequently using \eqref{eq:intro2}. In the present paper we provide the theoretical foundations for the learning based technique that we propose  to construct an approximation to the optimal feedback function for \eqref{def:refproblem}. Recently in \cite{onlfor2020} a variant of the  approach as in \cite{KW2020} was used for  interesting numerical investigations to construct optimal feedback functions for finite horizon  multi-agent optimal control problems.

Let us very briefly  mention some of  the vast literature
on solving the HJB equations. Semi-Lagrangian schemes and finite difference methods have been deeply investigated to directly solve HJB equations directly, see e.g. \cite{bggk13, ff16, kkr18}.
Significant progress was made in solving high
dimensional HJB equations by the of use policy iterations combined with tensor calculus techniques, \cite{dkk19,
kk18, foss2020}. The use of Hopf formulas was proposed in e.g. \cite{lr86,
cloy18}. Interpolation techniques, utilizing ensembles of open loop
solutions have been analyzed in the works of \cite{akk2020, ngk19}, for example.
Finally we mention that optimal feedback control is intimately
related to reinforcement learning, see e.g. the monograph
\cite{bert19}, and also the survey articles \cite{LV09, recht18,
vls14}.

The manuscript is structured as follows. Some pertinent notation is gathered in Section 2.
In Section 3 concepts of optimal feedback control, semi-global with respect to the initial condition $y_0$, are gathered. Section 4 is devoted to describing the learning technique that we propose to approximate the optimal feedback function. In Section 5 the required  assumptions on approximating subspaces  are checked for a class of neural networks and a class of piecewise polynomials. Existence of solutions to the approximating learning problems is proved in Section 6. Their convergence is analyzed in Section 7. The case of learning from finitely many orbits is the focus of Section 8. Section 9 provides an example illustrating the numerical feasibility of the proposed
method. We do not aim for sophistication in this respect.
The appendix details the proofs of several necessary technical results.

\section{Notation}

For ~$I:=(0,T)$, with $T>0$, we define
$
W_T= \{\,y \in L^2(I; \R^n)\;|\;\dot{y}\in L^2(I; \R^n)\,\},
$
where the temporal derivative is understood in the distributional sense. We equip~$W_T$ with the norm induced by the inner product
\begin{align*}
(y_1,y_2)_{W_T}=(\dot{y}_1,\dot{y}_2)_{L^2(I;\R^n)}+(y_1,y_2)_{L^2(I;\R^n)} \quad \text{for } y_1, y_2 \in W_T,
\end{align*}
making it a Hilbert space.
We recall that $W_T$ embeds continuously into $C(\bar I; \R^n)$.
For a compact  metric space~$X$ we denote the space of continuous functions between~$X$ and~$Y$ by~$\mathcal{C}(X;Y)$ which we endow with
$
\|\varphi\|_{\mathcal{C}(X;Y)}= \max_{x \in X} \|\varphi(x)\|_Y$  as norm.  By $Y_0$ we denote a compact set of initial conditions in $\R^n$. When arising as index, the space $\mathcal{C}(Y_0;W_T)$ will frequently  be abbreviated by $\mathcal{C}$. The space~$\mathcal{C}^1(X;Y)$ of continuously differentiable functions is defined analogously.
Open balls of radius $\eps$ in a Banach space $X$ with center $x$ will be denoted by $B_{\eps}(x)$.
The space of bounded linear operators between Banach spaces~$X$ and~$Y$, endowed with the canonical norm, is denoted by~$\mathcal{B}(X,Y)$.
We further abbreviate~$\mathcal{B}(X):=\mathcal{B}(X,X)$.

\section{Semi-global optimal feedback control}
\label{sec:o}
Consider the controlled nonlinear dynamical system of the form
\begin{align}
\label{eq:openloop}
\dot{y}= \mathbf{f}(y)+\mathbf{g}(y) u \quad \text{in}~L^2(I;\R^n), \quad y(0)=y_0,
\end{align}
described by  Nemitsky operators
\begin{equation}\label{eq:spaces}
\begin{array}l
\mathbf{f}\colon W_T \to L^2(I;\R^n), \quad \mathbf{f}(y)(t)= f(t,y(t)) \\[1.3ex]
\mathbf{g} \colon W_T  \to \mathcal{L}(L^2(I;\R^m);L^2(I;\R^n)), \quad \mathbf{g}(y)(t)= g(t,y(t))
\end{array}
\end{equation}
for~$a.e.~t\in I$,~$f \colon I \times \R^n  \to \R^m$ and~$g \colon I \times \R^n  \to \R^{n \times m}$. The smoothness requirements on~$f$ and~$g$ will be  detailed in Assumption~\ref{ass:feedbacklaw} below. Our aim  is to choose a control input~$u^*\in L^2(I;\R^m)$ which keeps the associated solution~$y^* \in W_T$  close to a known reference trajectory~$y_d$, while keeping the control effort small. This is formulated as the constrained minimization problem
\begin{equation} \label{def:openloopproblem}\tag{$P_{y_0}$}
\left\{
\begin{aligned}
\quad &\inf_{y \in W_T,\, u \in L^2(I; \R^m)} J(y,u) \\
&s.t. \quad \dot{y}= \mathbf{f}(y)+\mathbf{g}(y) u, \quad y(0)=y_0,
\end{aligned}
\right.
\end{equation}
where
\begin{align*}
J(y,u)=\frac{1}{2}\int_{I} \left ( |Q_1(y(t)-y_d(t))|^2+ \beta |u(t)|^2   \right )~\de t+\frac{1}{2}|Q_2(y(T)-y^T_d)|^2,
\end{align*}
which incorporates the weighted misfit between the trajectory ~$y$ within the time horizon $I=(0,T)$ and at the terminal time  to desired states~$y_d\in  L^2(I;\R^n)$ and $y_d^T \in \R^n$, as well as the norm of the control~$u$.  While this ~\textit{open loop} optimal control problem captures well the objective formulated above, it comes with several disadvantages. First, its solution is a function of time only, and does not include the current state $y(t)$. This makes the open loop approach susceptible to possible perturbations in the dynamical system. Second, determining the control action for a new initial condition requires to solve~\eqref{def:openloopproblem} from the start.

The aforementioned limitations of open loop optimal controls motivate the study of~\textit{semi-global optimal feedback control} approaches to \eqref{def:openloopproblem}. More precisely, given a compact set~$Y_0 \subset \R^n $, we look for a feedback function~$F^* \colon I \times \R^n \to\R^m$ which induces a Nemitsky operator
\begin{align*}
\F^* \colon W_T \to {L^2}(I; \R^m) , \quad \F^*(y)(t)=F^*(t,y(t)) \quad \text{for a.e.}~t \in I,
\end{align*}
such that for every~$y_0 \in Y_0$ the~\textit{closed loop system}
\begin{align} \label{eq:cloloop}
\dot{y}= \mathbf{f}(y)+ \mathbf{g}(y) \mathcal{F}^*(y), \quad y(0)=y_0,
\end{align}
admits a unique solution~$y^*( y_0 ) \in W_T$ and~$(y^*( y_0), \mathcal{F}^*(y^*(y_0 )))$ is a minimizing pair of~\eqref{def:openloopproblem}.

The determination of an optimal feedback function usually rests on the computation of the value function to\eqref{def:openloopproblem} which is defined as

\medskip

\begin{align} \label{def:valuefunc}
V^*(T_0,y_0):= \min_{\substack{y \in H^1(T_0,T;\R^n), \\ u \in L^2(T_0,T; \R^m)}} J_{T_0}(y,u) \quad s.t. \quad \dot{y}= \mathbf{f}(y)+ \mathbf{g}(y) u,~ \quad y(t_0)=y_0,
\end{align}

\medskip

\noindent
where~$(T_0,y_0)\in I \times \R^n$, and $J_{T_0}(y,u)$ is defined as
\begin{align*}
J_{T_0} (y,u)= \frac{1}{2}\int^T_{T_0} \left ( |Q_1(y(t)-y_d(t))|^2+ \beta |u(t)|^2   \right )~\de t+\frac{1}{2}|Q_2(y(T)-y_d(T))|^2.
\end{align*}
By construction $V^*$ satisfies the final time boundary condition
\begin{align*}
V^*(T,y_0)=\frac{1}{2} |Q_2(y_0-y_d(T))|^2 \quad \forall y_0 \in \R^n.
\end{align*}
If~$V^*$ is continuously differentiable in a neighborhood of some~$(t,y_0) \in I \times \R^n$ then it solves
the instationary~\textit{Hamilton-Jacobi-Bellman (HJB) equation}
\begin{align} \label{eq:HJB}
\partial_t V^*(t,y_0) +(f(y_0), \partial_y V^*(t,y_0))_{\R^n}- \frac{1}{2\beta}|g(t,y_0)^\top\partial_y V^*(t,y_0)|^2+ \frac{1}{2} |Q_1 (y_0-y_d(t))|^2=0
\end{align}
in the classical sense there, see e.g. \cite{ff14,FS06}. Here~$\partial_t V^*$ denotes the partial derivative of the value function with respect to~$t$ and~$\partial_y V^*$ is the gradient of~$V^*$ with respect to the~$y$-variable.
An optimal control for~\eqref{def:openloopproblem} in feedback form is then given by
$u^*=-\frac{1}{\beta} \mathbf{g}(y^*)^\top\partial_y \mathcal{V}^*(y^*)$ where~$\partial_y \mathcal{V}^*(y^*)(t)=\partial_y V^*(t,y^*(t))$ for every~$t\in I$, and~$y^*=y^*(y_0) \in W_T$ solves the closed loop system
\begin{align*}
\dot{y}=\mathbf{f}(y)- \frac{1}{\beta}\mathbf{g}(y) \mathbf{g}(y)^\top \partial_y \mathcal{V}^*(y), \quad y(0)=y_0.
\end{align*}
Thus
\begin{align*}
\left(y^*(y_0), -\frac{1}{\beta}\mathbf{g}(y^*(y_0))^\top \partial_y \mathcal{V}^*(y^*(y_0))\right) \in \argmin \eqref{def:openloopproblem}
\end{align*}
and the function
\begin{align*}
F^*(\cdot,\cdot)=-\frac{1}{\beta} g(\cdot,\cdot)^\top \partial_y V^*(\cdot,\cdot)
\end{align*}
is an optimal feedback law.

Realizing the optimal feedback  in this way requires a solution to \eqref{eq:HJB} which is a partial differential equation on~$\R^n$. This can be extremely challenging or even impossible depending on the dimension $n$ and the computational facilities at hand.  Similarly to our previous manuscript \cite{KW2020} we take a different approach by formulating  minimization problem over a suitable set of feedback functions involving the closed loop system as a constraint. This relates to a learning problem, within which the feedback functions are trained to achieve optimal stabilization. This  makes the problem computationally amenable.

The procedure just described will be formalized in the following section. Here we first  summarize the  assumptions on the nonlinear dynamical system that we refer to throughout the paper.
\begin{assumption} \label{ass:feedbacklaw}
\leavevmode
\begin{itemize}
\item[\textbf{A.1}] The functions~$f \colon I \times \R^n \to \R^n$ and~$g \colon I \times \R^{n} \to \R^{n \times m}$ are twice continuously differentiable. Their Jacobians and Hessians with respect to the second variable, denoted by~$D_{y}f, D_{yy}f$, and~$D_{y}g, D_{yy}g$, respectively, are Lipschitz continuous on compact sets, uniformly for $t\in I$.
\item [\textbf{A.2}] There exists a constant~$M_{Y_0}>0$ such that the value function~$V(\cdot ,\cdot)\colon I \times \R^n \to \R$ for~\eqref{def:openloopproblem} is twice continuously differentiable on~$I \times \bar{B}_{2\widehat{M}}(0)$ with Lipschitz continuous gradient and Hessian (w.r.t. $y$ uniformly in~$t \in I$) where
\begin{align}\label{eq:aux1}
\widehat M = M_{Y_0}  \, \|\imath\|_{\mathcal{B}(W_T,\, \mathcal{C}(I;\R^n))},
\end{align}
and $\imath$ denotes the embedding of $W_T$ into $\mathcal{C}(I;\R^n))$.
\end{itemize}
\end{assumption}

As a consequence of (\textbf{A.1}), the Nemitsky operators~$\mathbf{f},~\mathbf{g}$ are at least two times continuously differentiable with domains and ranges as defined in \eqref{eq:spaces}. Their derivatives, denoted by~$D \mathbf{f}$ and~$D \mathbf{g}$, are the Nemitsky operators induced by~$D_y f$ and~$D_y g$. We point out~$D \mathbf{g}(y)\in \mathcal{B}(W_T; L^2(I;\R^m); L^2(I;\R^n))$. Moreover~$ \mathbf{f}, D\mathbf{f}, \mathbf{g},D \mathbf{g}$ are Lipschitz continuous and bounded, on bounded subsets of $L^\infty(I;\R^n)$, and thus in particular on~$\mathcal{Y}_{ad}\subset W_T$, where
\begin{align}\label{eq:kk13}
\mathcal{Y}_{ad}:= \left \{\,y \in W_T\;|\;\wnorm{y}\leq 2 M_{Y_0}\,\right\}.
\end{align}
Finally $D\mathbf{f}^\top \in{\mathcal B}(W_T,L^2(I;\R^n))$ denotes the Nemitsky operator associated to $D_y f^\top$.

Analogously, due to (\textbf{A.2}), ~$V^*$ induces a twice Lipschitz continuously Fr\'echet differentiable Nemitsky operator~$\mathcal{V}^*: \mathcal{Y}_{ad}\subset W_T \to L^{\tcb{2}}(I)$.
Moreover~$\mathcal{V}^*$ and its first derivative~$D\mathcal{V}^* $ are weak-to-strong continuous.
Define the Nemitsky operator
\begin{align} \label{def:optfeddnemitsk}
\F^* \colon \mathcal{Y}_{ad}\to  L^2(I;\R^n), \tcb \quad \F^*(y)=-\frac{1}{\beta} \mathbf{g}(y)^\top \partial_y \mathcal{V}^*(y),
\end{align}
where~$\partial_y \mathcal{V}^*$ is the Nemitsky operator induced by the gradient~$\partial_y V^*=D_y V (\cdot,\cdot)^\top$.
Note also that $\F^* \in C^1(W_T; (L^2(I;\R^m);L^2(I;\R^n)))$.
We further assume the following:
\begin{itemize}
\item [\textbf{A.3}]For every~$y_0 \in Y_0$  there exists  a unique function~$y=\mathbf{y}^*(y_0)\in W_T$ satisfying
\begin{align*}
\dot{y}= \mathbf{f}(y)+ \mathbf{g}(y) \F^*(y), \quad y(0)=y_0, \quad \wnorm{y}\leq M_{Y_0}.
\end{align*}
Moreover we have
\begin{align*}
(y^*(y_0), \F^*(y^*(y_0))) \in \argmin \eqref{def:openloopproblem} \quad \forall y_0 \in Y_0.
\end{align*}
\end{itemize}
When referring to Assumption 1 we mean (\textbf{A.1})-(\textbf{A.3}).
We emphasize that the  constant $M$ appearing in (\textbf{A2}) and (\textbf{A3}) is assumed to be same.
Note further that as a consequence of (\textbf{A3}) problem \eqref{def:openloopproblem} admits a solution for each $y_0 \in Y_0$, with the optimal control given by $u^*=  \F^*(y^*(y_0)))$.

\begin{remark}\label{rem3}
Using $(\mathbf{A.1})$, $(\mathbf{A.3})$ as well as the implicit function theorem it can be readily be verified that the mapping~$\mathbf{y}^*\colon Y_0 \to W_T$ from~($\mathbf{A.3}$) is continuously differentiable. Given~$\delta y_0 \in \R^n$ the directional derivative~$\delta y\coloneqq \partial\mathbf{y}^*(y_0)(\delta y_0)$ of~$\mathbf{y}^*$ at~$y_0 \in Y_0$ in direction $\delta y_0$ satisfies the linearized ODE system
\begin{align*}
\dot{\delta y}= D \mathbf{f}(\mathbf{y}^*(y_0))\delta y+ \lbrack D\mathbf{g}(\mathbf{y}^*(y_0))\delta y\rbrack \F^*(\mathbf{y}^*(y_0)) +\mathbf{g}(\mathbf{y}^*(y_0))D\F^*(\mathbf{y}^*(y_0))\delta y,\, \delta y(0)= \delta y_0.
\end{align*}
 Here $D \mathbf{g}$ is induced by $D_yg$  which is given by
\begin{align*}
\left \lbrack D_y g(t,y) \delta y \right\rbrack_{ij}= \left( \sum^n_{k=1} \partial_k g_{ij}(t,y) \delta y_k \right) \quad \forall \delta y \in \R^n,
\end{align*}
where~$g(y)=(g_{ij})$ and~$"\partial_k "$ denotes the partial derivative w.r.t to the~$k$-th component of~$y$. The transposed~$D\mathbf{g}(y)^\top$, which will arise in the adjoint equation below, is induced by the tensor~$D_y g(t,\cdot)^\top=(D_y g(t,\cdot)_{kji})\in \R^{n\times n\times m}$, with $t\in I$. In particular, we readily verify that ~$D \mathbf{g}(\cdot)^\top \in \mathcal{B}(L^2(I;\R^m); \mathcal{B}(W_T; L^2(I;\R^n)))$.
\end{remark}

To end this section we collect structural  information on the relation between the adjoined state, denoted by ${p}$ below, the optima value function~$V^*$, and the induced optimal feedback law~$\F^*$.
\begin{prop} \label{prop:structure}
Let Assumption~\ref{ass:feedbacklaw} hold. Then there exists a unique continuous mapping~$\mathbf{p}^* \colon Y_0 \to W_T$ such that for each~$y_0 \in Y_0$  the tuple~$(y,p)=(\mathbf{y}^*(y_0),\mathbf{p}^*(y_0))$ satisfies
\begin{align}
\frac{d}{dt} y&=\mathbf{f}(y)+\mathbf{g}(y)\F^*(y),~ y(0)=y_0, \label{eq:stateprop} \\
-\frac{d}{dt} p&= D\mathbf{f}(y)^\top p+\lbrack D \mathbf{g}(y)^\top\F^*(y)\rbrack p+Q_1^\top Q_1(y-y_d),~ p(T)= Q^\top_2 Q_2(y(T)-y^T_d), \label{eq:adjointprop} \\
\F^*(y)&=- \frac{1}{\beta} \mathbf{g}(y)^\top p. \label{eq:gradienteqprop}
\end{align}
Moreover we have
\begin{align} \label{eq:dynamicalprog}
V^*(t,y(t))=J_t(y(t),F^*(t,y(t))),~p(t)= \partial_y V(t,y(t)) \quad \forall t\in [0,T].
\end{align}
\end{prop}

\begin{proof}[Proof of Proposition \ref{prop:structure}]
By (\textbf{A.3}) problem \eqref{def:openloopproblem} admits a solution for each $y_0 \in Y_0$. Then (\textbf{A.1})-(\textbf{A.2}) guarantee that
~\eqref{eq:adjointprop}, with $y=y(y_0)\in W_T$ the state component  of a solution to \eqref{def:openloopproblem},  admits a unique solution $p$ in~$W_T$ which continuously depends on~$y\in W_T$.
Moreover  \eqref{eq:stateprop} - \eqref{eq:gradienteqprop} represent the first order necessary optimality condition for~\eqref{def:openloopproblem} with the optimal control $u(t)= \F^*(y(t))$. Since $\mathbf{y}^*\colon Y_0 \to W_T$ is continuous as mentioned in Remark \ref{rem3} and the solution to \eqref{eq:adjointprop} depends continuously on $y\in W_T$, the claimed continuity $\mathbf{p}^* \colon Y_0 \to W_T$ follows.
Equation~\eqref{eq:dynamicalprog} is a direct consequence of the dynamic programming principle, and (\textbf{A.3}).
\end{proof}

\section{Optimal feedback control by value function approximation}\label{sec:learnfeedback}
This section is devoted to introducing a family of computationally tractable minimization problems from which we will "learn" approximations of optimal feedback laws. Our approach rests on two main pillars. First, given~$\eps>0$, we consider a family of functions $V^\eps_{\theta}  \in \mathcal{C}(I\times \R^n) $ which are finitely parametrized by~$\theta \in \mathcal{R}_\eps \simeq \R^{N_\eps}$,~$N_\eps \in \N$. These serve as "discrete" approximations of the optimal value function~$V^*$. The following a priori estimate is assumed, for some fixed $\eps_0 >0$:
\begin{assumption} \label{ass:approxsmoothness}
For every~$0 <\eps \leq \eps_0$ there holds~$V_{\cdot}^\eps \in\mathcal{C}^4(\mathcal{R}_\eps \times \R \times \R^n)$ and~$V^\eps_{\theta}(T,y_0)=\frac{1}{2} |Q_2(y_0-y_d(T))|^2$ for every~$y_0 \in \R^n$ and~$\theta \in \mathcal{R}_\eps$.
Moreover there exists~$\theta_\eps \in \mathcal{R}_{\eps}$ with
\begin{align} \label{eq:approxcapa}
\max_{\substack{ t \in I, \\ |y| \leq 2 \widehat{M} }}  |V^\eps_{\theta_\eps}(t,y)-V^*(t,y)|+|\partial_y (V^\eps_{\theta_\eps}(t,y)-V^*(t,y))| +\|\partial_{yy}(V^\eps_{\theta_\eps}(t,y)-V^*(t,y))\|  \leq c \eps
\end{align}
for some~$c>0$ independent of~$\eps\in (0,\eps_0]$.
\end{assumption}
Now recall from ~\eqref{def:optfeddnemitsk}
that the optimal feedback law~$\mathcal{F}^*$  is the superposition operator induced by~$F^*(t,y)=-(1/\beta)g(t,y)^\top \partial_y V^*(t,y) $. With the aim of preserving the dependence of the feedback law on the value function in our approximation, we define a set of parametrized feedback laws~$\mathcal{F}^\eps_\theta$ associated to~$V^\eps_\theta$,~$\theta \in \mathcal{R}_\eps$, by
\begin{align*}
\mathcal{F}^\eps_\theta(y)(t)= F^\eps_\theta(t,y(t))= -\frac{1}{\beta} g(t,y(t)) \partial_y V^\eps_\theta(t,y(t))
\end{align*}
for all~$y\in W_T$,~$t \in \bar{I}$ and~$\theta \in \mathcal{R}_\eps$. A first approach to obtain an
  optimal feedback law in the form~$\mathcal{F}^\eps_\theta$ can then be found by replacing the open loop control~$u$ in~\eqref{def:openloopproblem} by the closed loop expression~$\mathcal{F}^\eps_\theta(y)$ and minimizing for~$\theta \in \mathcal{R}_\eps$:
\begin{align}\label{eq:aux7}
\min_{y \in W_T, \theta \in \mathcal{R}_\eps  } J(y, F^\eps_\theta(y))+ \frac{\gamma_\eps}{2} \|\theta\|^2_{\mathcal{R}_\eps} \quad s.t. \quad \dot{y}= \mathbf{f}(y)+\mathbf{g}(y) \mathcal{F}^\eps_\theta(y),~y(0)=y_0,
\end{align}
where~$\|\cdot\|_{\mathcal{R}_\eps}$ denotes a Hilbert space norm on~$\mathcal{R}_\eps$,~$\gamma_\eps > 0$ and~$y_0 \in Y_0$ is fixed. This represents the goal of finding a feedback law~$\mathcal{F}^\eps_\theta$ together with a trajectory~$y\in W_T$ which satisfy~$(y, \mathcal{F}^\eps_\theta(y))\in \argmin \eqref{def:openloopproblem} $. However, this approach falls short in  several aspects. First, we cannot hope to recover a solution of the semiglobal optimal feedback control problem for all $y_0\in Y_0$, since the minimization in \eqref{eq:aux7} is associate to a single initial condition only. Secondly it misses to impose properties that would guide   ${\mathcal F}^\eps_\theta(y)$ to  be close to $\mathcal{V}^*$, and it does not exploit the relation between the adjoint state~$p$, see~\eqref{eq:adjointprop}, and the gradient of the value function~$\partial_y \mathcal{V}^*$. Incorporating this information into the problem can, potentially, lead to improved learning results and  improved parameterized feedback laws which behave similarly to~$\mathcal{F}^*$.
These considerations  lead to the second pillar of  our approach, namely a succinct choice of the cost for the learning problem.  For this purpose we use all of $Y_0$ as ~"learning set" for initial conditions.
It is endowed with the  normalized Lebesgue measure~$\mathcal{L}$.
Moreover  we define the augmented objective
\begin{equation}\label{eq:aux2}
\begin{array}l
J_\eps(y,p,\theta)= J(y,\F^\eps_{\theta}(y))\\[1.5ex]
+\int^T_0 \frac{\gamma_1}{2} |V^\eps_{\theta}(t,y(t))-J_t(y,\F^\eps_{\theta}(y))|^2 + \frac{\gamma_2}{2}|\partial_y V^\eps_{\theta}(t,y(t))-p(t)|^2~\mathrm{d}t
\end{array}
\end{equation}
for penalty parameters~$\gamma_1, \gamma_2 \geq 0$. The arguments in $J_t$ are the restriction of the solution  $y$ to the equation in \eqref{eq:aux7}
 and the feedback  $\F^\eps_{\theta}(y)$ to $[t,T]$.
 The additional terms in this new objective functional penalize the violation of the cost and its gradient by means of the approximation based on $V_\theta^\eps$, i.e. they penalize the differences between
$J_t(y,\mathcal{F}^\eps_\theta(y))$ and $V^\eps_\theta(t,y(t)$, as well as
$p(t)$   and~$\partial_y V^\eps_\theta(t,y(t)$.


 \black Given a strictly positive weight function~$\omega \in L^\infty(Y_0);~0<c \leq \omega$ a.e., we thus propose to find a feedback law~$\mathcal{F}^\eps_\theta$ by solving the ensemble control problem
\begin{align}\label{def:approxfeedprop}
 \min_{\substack{\mathbf{y}\in \mathbf{Y}_{ad},\\ \mathcal{F}^\eps_\theta(\mathbf{y}) \in \mathbf{U}_{ad},\\ \mathbf{p}\in \mathcal{C}(Y_0;W_T) \\ \theta \in \mathcal{R}_{\eps}}}\mathcal{J}_\eps(\mathbf{y},\mathbf{p},\theta)\coloneqq \int_{Y_0} \omega(y_0)\, J_\eps(\mathbf{y}(y_0),\mathbf{p}(y_0),\theta)~\de \mathcal{L}(y_0)+ \frac{\gamma_\eps}{2} \|\theta\|^2_{\mathcal{R}_\eps} \tag{$\mathcal{P}_\eps$}
\end{align}
subject to the system of closed loop state~\textit{and} adjoint equations
\begin{align}
\dot{\mathbf{y}}(y_0)=\mathbf{f}(\mathbf{y}(y_0))+\mathbf{g}(\mathbf{y}(y_0))\F^\eps_\theta(\mathbf{y}(y_0)) \label{eq:statepropapprox} \\
-\dot{\mathbf{p}}(y_0)= D\mathbf{f}(\mathbf{y}(y_0))^\top\mathbf{p}(y_0)+ \lbrack  D\mathbf{g}(\mathbf{y}(y_0))^\top\F^\eps_\theta(\mathbf{y}(y_0))\rbrack \mathbf{p}(y_0)+\mathbf{Q}_1^\top \mathbf{Q}_1(\mathbf{y}(y_0)-y_d) \label{eq:adjointpropapprox} \\
\mathbf{y}(y_0)(0)=y_0,~\mathbf{p}(y_0)(T)= Q^\top_2 Q_2(\mathbf{y}(y_0)(T)-y_d^T),~\mathbf{y}(y_0)\in \mathcal{Y}_{ad} \label{eq:constraintprop}
\end{align}
for~$\mathcal{L}$-a.e.~$y_0\in Y_0$. Above~$\mathbf{Y}_{ad} \subset \mathcal{C}(Y_0;W_T)$ and~$\mathbf{U}_{ad} \subset L^2(Y_0;L^2(I;\R^m))$ denote the admissible sets of ensemble state trajectories and admissible  controls. They will  be  specified in  section \ref{sec:existence}.

%
%
%

\section{Examples} \label{sec:examples}
In this section we discuss two particular examples for the parameterized mappings~$V^\eps$: deep residual networks and  piecewise polynomial functions of sufficiently high degree.

\subsection{Residual networks} \label{subsec:residual}
To explain the approximation of the value function  by  residual neural
networks, we  first fix some notation.
Let~$L_\eps\in \N$,~$L_\eps \geq 2$, as well as~$N^{\eps}_{i}\in
\N$,~$i=1, \dots,L_\eps -1$ be given. We set~$N^{\eps}_0=n+1$
and~$N^{\eps}_L=1$. Furthermore define
\begin{align*}
\mathcal{R}_\eps= \bigtimes^{L_\eps-1}_{i=1} \left ( \R^{N^\eps_{i}
\times N^\eps_{i-1}} \times \R^{N^\eps_{i} \times N^\eps_{i-1}} \times
\R^{N^\eps_i} \right ) \times \R^{N^\eps_{L} \times N^\eps_{L-1}}.
\end{align*}
The space~$\mathcal{R}_\eps$ is uniquely determined by
its~\textit{architecture}
\begin{align*}
\text{arch}(\mathcal{R}_\eps)=\left( N^\eps_0, N^\eps_1, \dots ,
N^\eps_L\right)\in \N^{L_\eps+1}.
\end{align*}
A set of parameters~$\theta \in \mathcal{R}_\eps$ given by
\begin{align*}
\theta=\left( W_{11}, W_{12}, b_1,\dots, W_{{L_\eps}}
\right)
\end{align*}
is called a~\textit{neural network} with~$L_\eps$~\textit{layers}.
Moreover let~$\sigma \in \mathcal{C}^4(\R)$ be given and assume that~$\sigma$ is not a polynomial.
The function
\begin{equation}\label{eq:shift}
\begin{array}l
V^\eps_{\theta}(t,y)= \frac{1}{2} |Q_2(y-y_d(T))|^2 \\[1.5ex]
\qquad \qquad +   f^{\sigma}_{L_\eps, \theta}  \circ
f^{\sigma}_{L_\eps-1, \theta} \circ \cdots \circ f^{\sigma}_{1,
\theta}((t,y))-f^{\sigma}_{L_\eps, \theta}  \circ f^{\sigma}_{L_\eps-1,
\theta} \circ \cdots \circ f^{\sigma}_{1, \theta}((T,y))
\end{array}
\end{equation}
for~$(t,y)\in \R \times \R^n$ where
\begin{align*}
f^{\sigma}_{L_\eps, \theta}(x)= W_{L_\eps} x \quad \forall x
\in \R^{N^\eps_{L-1}}
\end{align*}
as well as
\begin{align*}
f^{\sigma}_{i, \theta}(x)= \sigma(W_{i1}x+b_i)+ W_{i2}x \quad \forall x
\in \R^{N^\eps_{i-1}},~i=1, \dots, L_\eps-1
\end{align*}
is called the~\textit{realization} of~$\theta$ with~\textit{activation
function}~$\sigma$.
Here the application of~$\sigma$ is  defined to act componentwise i.e.
given an index~$i \in \{1,\dots,L_\eps-1\}$ and~$x \in \R^{N^\eps_i}$ we set
\begin{align*}
\sigma(x)=(\sigma(x_1), \dots, \sigma(x_{N^\eps_i}))^\top.
\end{align*}
By construction,~$V^\eps_\theta$ satisfies the terminal condition
\begin{align*}
V^\eps_\theta(T,y)=\frac{1}{2} |Q_2(y-y_d(T))|^2 \quad \forall y\in
\R^n.
\end{align*}
Moreover Assumption~\ref{ass:approxsmoothness} is fulfilled as confirmed by the following result.
\begin{theorem} \label{thm:approxbynetworks}
For every~$\eps>0$ there exists
architectures~$\operatorname{\mathcal{R}}_\eps$ and~$\theta_\eps\in
\mathcal{R}_\eps$ such that~$V^\eps \in \mathcal{C}^4(\mathcal{R}_\eps
\times \R \times \R^n)$ and~$V^\eps_{\theta_\eps}$
satisfies~\eqref{eq:approxcapa}.
\end{theorem}
\begin{proof}
Let us set $h(t,y)= V^*(t,y)$ for $(t,y) \in I \times \bar
B_{2\widehat M}(0)$. Then $h$ is twice continuously differentiable on $I
\times \bar B_{2\widehat M}(0)$ and $h(T,y)= \frac{1}{2} |Q_2(y-y^T_d)|^2$. A consequence of the
universal approximation theorem implies that for all $\eps >0 $
there exists $\tilde h_\eps \in \mathcal{M}_{\text{net}}$ such that
\begin{equation}\label{eq:aux10}
\|h-\tilde h_\eps \|_{C^2(I\times \bar B_{2\widehat M}(0))} \le
\frac{\varepsilon}{2},
\end{equation}
where $\mathcal{M}_{\text{net}}= \text{span} \{\sigma( \tilde w\cdot x+ \tilde b ):
\tilde w \in \R^{n+1},\, \tilde b\in\R \}$, see eg \cite[Theorem 4.1]{P99},
\cite{hornik91}. Let us observe that $\tilde h_\eps$ can be
expressed as a residual network. Indeed, since
$$
\tilde h_\eps = \sum_{i=1}^M \tilde c_i \sigma(\tilde w_i \cdot x
+\tilde b_i)
$$
for some $M\in \mathbb{N}$, $\tilde w_i \in \R^{n+1}, \tilde b_i, \tilde
c_i \in \R$, choosing $L_\epsilon=2,  W_{11}\in \R^{M\times (n+1)}$ with
rows $\{\tilde w_i\}_{i=1}^M$,
\begin{equation*}
b_1=\text{col}(\tilde b_1,\dots, \tilde b_M), \,W_2=(\tilde
c_1,\dots, \tilde c_M), \, W_{12}=0,
\end{equation*}
we have $\tilde h_\eps = f^\sigma_{2,\theta}\circ
f^\sigma_{1,\theta}$. Moreover, $\tilde h_\eps\in
C^4(I\times \bar B_{2\widehat M})$. Following \eqref{eq:shift} we define
\begin{align*}
V^\eps_{\theta_\eps}(t,y)= \frac{1}{2} |Q_2(y-y^T_d)|^2 + \tilde h_\eps(t,y) -
\tilde h_\eps(T,y)\in C^4(I\times \bar B_{2\widehat M}).
\end{align*}
and estimate
\begin{align*}
\| V^\eps_{\theta_\epsilon}(t,y) - V^*(t,y)\|_{C^2} &=
\| \tilde h_{\eps}(t,y) -  \tilde h_{\eps}(T,y) + V^*(T,y) -
V^*(t,y)\|_{C^2} \\
& \leq 2 \|h-\tilde{h}_\eps\|_{C^2} \leq \eps,
\end{align*}
where all norms are taken over $I\times \bar B_{2\widehat M}(0)$. This
ends the proof.
\end{proof}

\subsection{Piecewise polynomials} \label{subsec:polynomials}
Fix~$\eps_0>0$, and let $\eps\in (0, \eps_0]$ be arbitrarily fixed. Throughout this subsection we assume ~$(\mathbf{A.2})$ and in particular we shall make use of the global Lipschitz continuity of $D^2 V^*$ on $\bar K=  \bar I \times \bar B_{2\widehat M}(0)$. Since $\bar K$ is compact and hence totally bounded,  there exist $n_\eps \in \mathbb{N}$ and $\{(\bar{t}_i,\bar{y}^i_0)\}_{i=1}^{n_\eps}\in \R^{n+1}$ such that
\begin{align*}
\bar{K} \subset \bigcup^{n_\eps}_{i=1} K_i \quad \text{where} \quad K_i= B_{\eps}( (t_i,\bar{y}^i_0)) .
\end{align*}
Note that we do not highlight the dependence of $(t_i,\bar{y}^i_0)$ and $K_i$ on $\eps$.
For each~$i$ define the parametrized polynomial
\begin{align*}
V^\eps_i(A,b,c,t,y)= (t-\bar{t}_i, y-\bar{y}^i_0)^\top A(t-\bar{t}_i, y-\bar{y}^i_0)+ b^\top (t-\bar{t}_i, y-\bar{y}^i_0)+c
\end{align*}
with
\begin{align*}
(A,b,c,t,y) \in \operatorname{Sym}(n+1) \times \R^{n+1} \times \R \times \R^{n+1},
\end{align*}
where ${Sym}(n)$ denotes the space of real symmetric  $ n \times n$ matrices.  Note that~$V^\eps_i $ is infinitely many times differentiable in all of its arguments.

For each $\eps\in (0,\eps_0]$ we define a special partition of unity $\{\varphi_i\}_{i=1}^{n_\eps}$ subordinate to~$K_i$ with
$\varphi_i \colon \R \times \R^n \to [0,1]$, satisfying ~$\mathcal{C}^4$ and
\begin{equation}\label{eq:kk10}
\left\{
\begin{array}{c}
\supp \varphi_i = \bar K_i, \qquad ~\sum^{n_\eps}_{i=1} \varphi_i(t,y)=1, \, \forall (t,y) \in \bar K, \\[1.7ex]
\|D^{j}\varphi_i\|_{C(\bar K_i\cap \bar K)} =\bar \mu \eps^{-j}, \forall i=1,\dots, n_\eps, \; \text{ and } j\in\{1,2\},\\[1.7ex]
\text{card } \{i: \varphi_i(t,y) \neq 0 \} \le \mathfrak{m} \quad  \forall (t,y)\in \bar K, \eps\in (0,\eps_0],
\end{array}
\right.
\end{equation}
with $\bar \mu$ and $\mathfrak{m}$  positive constants independent of $i, (t,y)\in \bar K, \eps\in (0,\eps_0]$.
Finally we define
\begin{align*}
\mathcal{R}_\eps= \bigtimes^{n_\eps}_{i=1} ( \operatorname{Sym}(n+1) \times \R^{n+1} \times \R ),
\end{align*}
and introduce the family of parameterized functions on $\R^{n+1}$ by
\begin{align} \label{eq:taylorfunctionchoice}
V^\eps_\theta (t,y)= \frac{1}{2} |Q_2(y-y_d(T))|^2+ \sum^{n_\eps}_{i=1} \varphi_i(t,y)\left( V^\eps_i(A_i,b_i,c_i,t,y)-V^\eps_i(A_i,b_i,c_i,T,y) \right)
\end{align}
for~$\theta=(A_1,b_1,c_1,\dots, A_{n_\eps},b_{n_\eps},c_{n_\eps})\in\mathcal{R}\eps$. Obviously we have~$V^\eps_{\cdot}(\cdot) \in\mathcal{C}^4(\mathcal{R}_\eps \times \R \times \R^n)$ and
\begin{align*}
V^\eps_\theta (T,y)=\frac{1}{2} |Q_2(y-y^T_d)|^2 \quad \forall y \in \R^n.
\end{align*}
Thus the final time condition in the HJB equation is fulfilled. Next we show that { $V^\eps_{\theta}$  satisfies the approximation property in Assumption~\ref{ass:approxsmoothness} for the particular choice of
\begin{equation} \label{eq:taylorparachoice}
\begin{array}{ll}
\theta_\eps = &\big(\partial_{yy} V^*(\bar{t}_1,\bar{y}^1_0), \partial_y V^*(\bar{t}_1,\bar{y}^1_0),V^*(\bar{t}_1,\bar{y}^1_0), \\[1.4ex]
&\dots, \partial_{yy} V^*(\bar{t}_{n_\eps},\bar{y}^{n_\eps}_0), \partial_y V^*(\bar{t}_{n_\eps},\bar{y}^{n_\eps}_0),V^*(\bar{t}_{n_\eps},\bar{y}^{n_\eps}_0)\big),
\end{array}
\end{equation}
i.e. $V^\eps_i$ in \eqref{eq:taylorfunctionchoice} are chosen with
\begin{equation}\label{eq:kk11}
(\bar{A}_i,\bar{b}_i,\bar{c}_i)=(\partial_{yy} V^*(\bar{t}_i,\bar{y}^i_0), \partial_y V^*(\bar{t}_i,\bar{y}^i_0),V^*(\bar{t}_i,\bar{y}^i_0)) \quad i=1,\dots,n_\eps.
\end{equation}

\begin{theorem}
\label{thm:Taylor}
Let~$V^\eps$ and~$\theta_\eps$ be chosen according to~\eqref{eq:taylorfunctionchoice} and~\eqref{eq:taylorparachoice}, respectively, and suppose that ~$(\mathbf{A.2})$ and \eqref{eq:kk10} are satisfied. Then Assumption~\ref{ass:approxsmoothness} holds.
\end{theorem}
\begin{proof}
We already argued that~$V_\theta^\eps$ has the desired regularity. It remains to prove the required approximation capabilities.
For abbreviation set~$V^\eps_i(t,y)=V^\eps_i({\bar A_i},{\bar b_i},{\bar c_i},t,y)$, with $({\bar A_i},{\bar b_i},{\bar c_i})$ as in \eqref{eq:kk11}.

Since~$V^\eps_i$ is the second order Taylor expansion of~$V$ at~$(\bar{t}_i,\bar{y}^i_0)$  we conclude that
\begin{equation}\label{eq:kk12}
\|V^*-V^\eps_i\|_{C^{2-j}(\bar K_i \cap\bar K)} \leq \bar c\eps^{j+1}, \quad \text{for } j\in{0,1,2,}
\end{equation}
for some~$\bar c>0$ depending on the global Lipschitz constant of $V$ on $\bar K$, and independent of~$\eps \in (0,\eps_0]$ and~$i$. Still recall that the sets $K_i$ depend on $\eps$.
To estimate $V^*(t,y) -V^\eps_\theta(t,y)$ we recall that $V^*(T,y)=\frac{1}{2} |Q_2(y-y^T_d)|^2$, and express $V^*(t,y)$ as  $V^*(t,y)=  V^*(T,y)+ V^*(t,y)- V^*(T,y)$. This leads to
\begin{equation*}
\begin{array} l
V^*(t,y)-V^\eps_\theta(t,y) \\[1.5ex]
\;= \sum_{i\in \{1,\dots,n_\eps\}} \varphi_i(t,y) (V^*(t,y)-V^\eps_i(t,y)) +\sum_{i\in \{1,\dots,n_\eps\}}\varphi_i(T,y)\ (V^*(T,y)-V^\eps_i(T,y)),
\end{array}
\end{equation*}
for $(t,y) \in \bar K$.
From \eqref{eq:kk10} and \eqref{eq:kk10} we deduce that $\|V(t,y) -V^\eps_\theta(t,y)\|_{C(\bar K)}\le 2 \bar c \eps^3$.

For the gradient with respect to~$y$ we proceed similarly.  Fixing ~$(t,y)\in \bar K $ we estimate
\begin{align*}
|\partial_y V^*(t,y)&-\partial_y V^\eps_{\theta_\eps}(t,y)| \leq D_1+D_2
\end{align*}
where
\begin{align*}
D_1 &=\sum_{i \in \{1,\dots,n_\eps\}} \left \lbrack \varphi_i(t,y) |\partial_y V^*(t,y)-\partial_y V^\eps_i(t,y)|+ | V^\eps_i(t,y)-V^*(t,y)| |\partial_y \varphi_i(t,y)| \right \rbrack \\
D_2 &=\sum_{i \in \{1,\dots,n_\eps\}} \left \lbrack \varphi_i(T,y) |\partial_y V^*(T,y)-\partial_y V^\eps_i(T,y)|+ | V^\eps_i(T,y)-V^*(T,y)| |\partial_y \varphi_i(t,y)| \right \rbrack.
\end{align*}
By \eqref{eq:kk12} with $j=1$ the first terms in $D_1$ and $D_2$ can be estimated by $\bar c \eps ^2$.
Using  \eqref{eq:kk10} and \eqref{eq:kk12} the second terms in $D_1$ and $D_2$ can be bounded by  $\mathfrak{m} \bar \mu \eps ^2$. Combining these estimate we arrive at
\begin{equation*}
\|\partial_y  V^*(t,y) -\partial_y  V^\eps_\theta(t,y)\|_{C(\bar K)}\le 2 \eps^2( \bar c +  \mathfrak{m}\, \bar \mu).
\end{equation*}
In an analogous manner one can obtain a bound of the order $O(\eps)$ on the difference of the Hessians of $V$ and $V^\eps_\theta$. This finishes the proof.
\end{proof}

In Appendix \ref{app1} it is shown  how standard mollifiers can be used so that \eqref{eq:kk10} is satisfied. This requires some extra attention due to the required  bounds on the derivatives of $\varphi_i$.

\section{Existence of minimizers to~\eqref{def:approxfeedprop}} \label{sec:existence}
This section is devoted to proving the existence of minimizing triples to~\eqref{def:approxfeedprop}. Throughout this section $c$ will denote a generic constant independent of $\eps>0$ and $y_0\in Y_0$.

\subsection{Existence of admissible points} \label{subsec:exstofmin}
Recall from  Assumption~\ref{ass:feedbacklaw} \tcb and Remark \ref{rem3}  that the optimal ensemble state~$\mathbf{y}^* \in \mathcal{C}(Y_0;W_T)$ satisfies~$\|{\mathbf{y}^*}\|_{{\mathcal C}}\leq M_{Y_0}$. Accordingly we define the set of admissible states and admissible controls as
\begin{align*}
\mathbf{Y}_{ad} = \left\{\, \mathbf{y} \in \mathcal{C}(Y_0;W_T)\;|\;\Cbochnorm{\mathbf{y}} \leq 2 M_{Y_0}\, \right\}, \quad  \mathbf{U}_{ad} \coloneqq L^2(Y_0;L^2(I;\R^m)).
\end{align*}
We also recall the definition  $\mathcal{Y}_{ad}$ in \eqref{eq:kk13}.

To prove the existence of minimizers to~\eqref{def:approxfeedprop} we first argue that the admissible set
\begin{align} \label{def:admissibleset}
\mathcal{N}^{\eps}_{ad}= \left\{(\mathbf{y},\mathbf{p},\theta)\!\in\! \mathbf{Y}_{ad}\!\times\!\mathcal{C}(Y_0;W_T) \!\times\! \mathcal{R}_{\eps}|(\mathbf{y},\mathbf{p},\theta)~\text{satisfies}\!~\eqref{eq:statepropapprox}\!
-\!\eqref{eq:constraintprop}, \F^\eps_{\theta}(\mathbf{y})\!\in\! \mathbf{U}_{ad}\right\}
\end{align}
is nonempty for~$\eps$ small enough.
For this purpose  consider the family~$\theta_\eps \in \mathcal{R}_{\eps}$,~$0< \eps \leq \eps_0$, from Assumption~\ref{ass:approxsmoothness} as well as the associated closed loop system of state and adjoint equations
\begin{align} \label{eq:neuralnetstate}
\dot{y}_\eps&= \mathbf{f}(y_\epsilon)+ \mathbf{g}(y_\eps)
\F^\eps_{\theta_\epsilon} (y_\epsilon), \\
-\dot{p}_\eps&= D\mathbf{f}(y_\eps)^\top p_\eps + \lbrack  D\mathbf{g}(y_\eps)^\top\F^\eps_{\theta_\eps}(y_\eps)\rbrack p_\eps+\mathbf{Q}_1^\top \mathbf{Q}_1(y_\eps-y_d), \label{eq:neuraladjoint}
\end{align}
subject to the following initial and terminal conditions
\begin{align*}
y_\epsilon(0)=y_0,~p_\eps(T)= Q^\top_2 Q_2(y_\eps(T)-y_d^T),
\end{align*}
for every~$y_0 \in Y_0$.
We first prove the following approximation result.
\begin{theorem} \label{thm:existenceneuralnetwork2}
Let Assumptions~\ref{ass:feedbacklaw} and~\ref{ass:approxsmoothness} hold. There exists a constant $c$ such that for all~$\eps>0$ small enough and for all $y_0\in Y_0$
the system \eqref{eq:neuralnetstate} and~\eqref{eq:neuraladjoint}
admits  unique solutions~$y_\eps=\mathbf{y}_\eps (y_0) \in \mathcal{Y}_{ad}$ and~$p_\eps=\mathbf{p}_\eps (y_0) \in W_T$. Furthermore ~$\mathbf{y}_\eps \in \mathcal{C}^1 (Y_0;W_T)$,~$\mathbf{p}_\eps \in \mathcal{C} (Y_0;W_T)$, and $\F^*(\mathbf{y}^*) \in \mathcal{C} (Y_0;L^2(I;\R^m))$  hold and
\begin{align*}
\|\mathbf{y}_\eps-\mathbf{y}^*\|_{\mathcal{C}^1(Y_0;W_T)}+
 \|\mathbf{p}_\eps-\mathbf{p}^*\|_{\mathcal{C}(Y_0;W_T)}+ \|\F^\eps_{\theta_\eps}(\mathbf{y}_\eps)-\F^*(\mathbf{y}^*)\|_{\mathcal{C}(Y_0;L^2(I;\R^m))} \leq c \eps.
\end{align*}
In particular,~$(\mathbf{y}_\eps,\mathbf{p}_\eps,\theta_\eps)\in \mathcal{N}^\eps_{ad}$ for all~$\eps>0$ small enough.
\end{theorem}
In order to prove this we require several auxiliary results.
\begin{lemma} \label{lem:lipschitzoffeed}
There exists a constant $c$ such that for all $\eps$ small enough there holds
\begin{align*}
\|(\F^*(y_1)-\mathcal{F}^\eps_{\theta_\eps}(y_1))-(\F^*(y_2)-\F^\eps_{\theta_\eps}(y_2))\|_{L^2(I;\R^m)}\leq c\eps \wnorm{y_1-y_2}, \quad \forall y_1,y_2 \in \mathcal{Y}_{ad}.
\end{align*}
\end{lemma}
\begin{proof}
According to the definition of~$\mathcal{F}^*$ and~$\mathcal{F}^\eps_{\theta_\eps}$ we split
\begin{align*}
\|(\F^*(y_1)-\mathcal{F}^\eps_{\theta_\eps}(y_1))-(\F^*(y_2)-\F^\eps_{\theta_\eps}(y_2))\|_{L^2(I;\R^m)} \leq D_1+D_2
\end{align*}
with
\begin{align*}
D_1&= 1/\beta \, \|\mathbf{g}(y_1)^\top\|_{\mathcal{B}(L^2(I;\R^n),L^2(I;\R^m))} \|\partial_y((\mathcal{V}^*(y_1)\!-\!\mathcal{V}^\eps_{\theta_\eps}(y_1))-(\mathcal{V}^*(y_2)\!-\!\mathcal{V}^\eps_{\theta_\eps}(y_2)))\|_{L^2(I;\R^n)} \\
D_2 &= 1/\beta \, \|\mathbf{g}(y_1)^\top-\mathbf{g}(y_2)^\top\|_{\mathcal{B}(L^2(I;\R^n),L^2(I;\R^m))} \|\partial_y(\mathcal{V}^*(y_2)-\mathcal{V}^\eps_{\theta_\eps}(y_2))\|_{L^2(I;\R^n)}.
\end{align*}
Applying the integral mean value theorem yields
\begin{align*}
\notag
 \|\partial_y((\mathcal{V}^*(y_1)-\mathcal{V}^\eps_{\theta_\eps}(y_1))&-(\mathcal{V}^*(y_2)-\mathcal{V}^\eps_{\theta_\eps}(y_2)))\|_{L^2(I;\R^n)}
\\[1.5ex]
& \leq
\sup_{s\in
[0,1]}\|\partial_{yy}(\mathcal{V}^*(y_1+sh)-\mathcal{V}^\eps_{\theta_\eps}(y_1+sh))\|_{\mathcal{B}(W_T,L^2(I;\R^n))}
  \wnorm{h}
\end{align*}
with~$h=y_2-y_1\in W_T$. Note that~$y_1+sh \in \mathcal{Y}_{ad}$ for all~$s \in [0,1]$. Thus we can use Assumption ~\ref{ass:approxsmoothness} for every~$s \in [0,1]$ and~$\delta y \in W_\infty$ and estimate
\begin{align*}
\|\partial_{yy}(\mathcal{V}^*(y_1+sh)&-\mathcal{V}^\eps_{\theta_\eps}(y_1+sh))\delta y\|_{L^2(I;\R^n)}
\\&\leq \sqrt{\int_0^T~|\partial_{yy}(V^*(t,y_1(t)+sh(t))-V^\eps_{\theta_\eps}(t,y_1(t)+sh(t)))|^2_{\R^{n
\times n}}|\delta y(t)|^2\mathrm{d} t}
\\&\leq c \eps  \|\delta y\|_{L^2(I;\R^n)} \leq \eps c \wnorm{\delta y}.
\end{align*}
 Similarly we obtain
\begin{align*}
\|\partial_{y}(\mathcal{V}^*(y_2)-\mathcal{V}^\eps_{\theta_\eps}(y_2))\|_{L^2(I;\R^n)}= \sqrt{\int^T_0 |\partial_{y}({V}^*(t,y_2(t))-{V}^\eps_{\theta_\eps}(t,y_2(t)))|^2~\de t} \leq \sqrt{T}c \eps.
\end{align*}
Last recall that~$\mathbf{g}$ is Lipschitz continuous and uniformly bounded on~$\mathcal{Y}_{ad}$.
Combining these facts yields the desired statement.
\end{proof}
With the same arguments  the following a priori estimate can be obtained. For the sake of brevity
its proof is omitted.
\begin{coroll}\label{lem:calmness}
There exists a constant $c$ such that for all $\eps$ small enough there holds
\begin{align*}
\|\mathcal{F}^*(y)-\mathcal{F}^\eps_{\theta_\eps}(y)\|_{L^2(I;\R^m)}
\leq c \eps \wnorm{y}, \quad \forall y \in \mathcal{Y}_{ad}.
\end{align*}
\end{coroll}

Next we establish  existence of a unique solution to~\eqref{eq:neuralnetstate} as well as a first approximation result.
\begin{prop} \label{thm:existenceneuralnetwork}
Let Assumptions~\ref{ass:feedbacklaw} and~\ref{ass:approxsmoothness} hold. Then for all~$\eps>0$ small enough
 there is a unique~$\mathbf{y}_\eps \in \mathcal{C}^1(Y_0;W_T)$ such that~$y_\eps\coloneqq\mathbf{y}_\eps(y_0) \in \mathcal{Y}_{ad}$ satisfies~\eqref{eq:neuralnetstate} for all $y_0 \in Y_0$.
 Moreover there exists a constant $c$ independent of $\eps$  such that
\begin{align*}
\|\mathbf{y}^*-\mathbf{y}_\eps\|_{\textcolor{black}{\mathcal{C}(Y_0;W_T)}}+ \|\F^*(\mathbf{y}^*)-\F^{\eps}_{\theta_\eps}(\mathbf{y}_\eps)\|_{\textcolor{black}{\mathcal{C}(Y_0; L^2(I;\R^m))}} \leq c\eps.
\end{align*}
In particular we have $\|\mathbf{y}_\eps\|_{\mathcal{C}(Y_0;W_T)} \leq 2 M_{Y_0}$ for all sufficiently small $\eps$.
\end{prop}

\begin{proof}
The proof is based on a fixed-point argument. Let~$y_0 \in
Y_0$ be arbitrary but fixed. Define the set
\begin{align*}
\mathcal{M}= \left\{\, y \in W_T \;|\;\wnorm{y} \leq \frac{3}{2}
M_{Y_0}\,\right\} \subset \mathcal{Y}_{ad}.
\end{align*}
On~$\mathcal{M}$ we consider the mapping~$\mathcal{Z} \colon \mathcal{M}
\to W_T$, where~$z=\mathcal{Z}(y)\in \mathcal{Y}_{ad}$ is the unique
solution of
\begin{align} \label{eq:auxeqfixpoint}
\dot{z}= \mathbf{f}(z)+ \mathbf{g}(z)\F^*(z)+
\mathbf{g}(y)\F^\eps_{\theta_\epsilon}(y)-\mathbf{g}(y) \mathcal{F}^*(y),
\quad z(0)=y_0.
\end{align}
It is well-defined since the
perturbation function~$v=\mathbf{g}(y)\F^\eps_{\theta_\epsilon}(y)-
\mathbf{g}(y) \mathcal{F}^*(y)\in L^2(I;\R^n)$ satisfies
\begin{align*}
\|v\|_{L^2}\leq \|\mathbf{g}(y)\|_{\mathcal{B}(L^2(I;\R^m),L^2(I;\R^n))}
\|\mathcal{F}^*(y)-\F^\eps_{\theta_\epsilon}(y)\|_{L^2} \leq \frac{3}{2} c \eps  M_{Y_0}  \|\mathbf{g}(y)\|_{\mathcal{B}(L^2(I;\R^m),L^2(I;\R^n))}
\end{align*}
where we use Corollary~\ref{lem:calmness} and the definition of~$\mathcal{M}$. Hence $\|v\|_{L^2}\leq c \eps$.
Here and below $c$ denotes a generic constant which is independent of $y_0\in Y_0$ and all $\epsilon>0$ sufficiently small.
We may invoke Proposition~\ref{thm:existspert}
and Corollary~\ref{coroll:locallipschitzofstate} from the Appendix,  to assert the existence of a
unique solution~$z\in \mathcal{Y}_{ad}$ to~\eqref{eq:auxeqfixpoint} with
\begin{align*}
\wnorm{z}& \leq M_{Y_0}+c\|v\|_{L^2} \leq \frac{3}{2} M_{Y_0},\quad \forall y_0 \in Y_0,
\end{align*}
if~$\eps>0$ is chosen small enough.
From this we particularly conclude~$\mathcal{Z}(\mathcal{M})\subset \mathcal{M}$ for
all~$y_0 \in Y_0$ and~$\eps >0$ small. It remains to prove
that~$\mathcal{Z}$ is a contraction. To this end let~$y_1,~y_2 \in
\mathcal{M}$ be given. Applying
Corollary~\ref{coroll:locallipschitzofstate} yields the first inequality in
\begin{align*}
\wnorm{\mathcal{Z}(y_1)-\mathcal{Z}(y_2)} &\leq c
\|\mathcal{F}^*(y_1)-\F^\eps_{\theta_\epsilon}(y_1)-\mathcal{F}^*(y_2)+\F^\eps_{\theta_\epsilon}(y_2)\|_{L^2} \leq c \eps \wnorm{y_1-y_2}
\end{align*}
with a constant~$c>0$ independent of~$y_1, y_2 \in \mathcal{M}$ as well
as of~$y_0 \in Y_0$, and $\epsilon$ sufficiently small.
The last inequality follows from
Lemma~\ref{lem:lipschitzoffeed}.
Choosing~$\eps>0$ small enough we conclude that~$\mathcal{Z}$ admits a
unique fixed point~$y_\eps=\mathcal{Z}(y_\eps) \in W_T$
on~$\mathcal{M}$. Clearly, the function~$ \mathbf{y}_\eps(y_0):=y_\eps$
satisfies~\eqref{eq:neuralnetstate},~$y_\eps \in  {\mathcal{M}} \subset\mathcal{Y}_{ad}$ as well as
\begin{align*}
\wnorm{\mathbf{y}_\eps(y_0)-\mathbf{y}^*(y_0)}&=\wnorm{\mathcal{Z}(\mathbf{y}_\eps(y_0))-\mathcal{Z}(0)}\leq
c\eps \wnorm{y_\eps}
\leq c \eps \frac{3}{2} M_{Y_0},
\end{align*}
and by Corollary \ref{lem:calmness}
\begin{align*}
\|\mathcal{F}^*(\mathbf{y}^*(y_0))&-\mathcal{F}^\eps_{\theta_\eps}(\mathbf{y}_\eps(y_0))\|_{L^2} \\&\leq \|\mathcal{F}^*(\mathbf{y}^*(y_0))-\mathcal{F}^*(\mathbf{y}_\eps(y_0))\|_{L^2} + \|\mathcal{F}^*(\mathbf{y}_\eps(y_0))-\mathcal{F}^\eps_{\theta_\eps}(\mathbf{y}_\eps(y_0))\|_{L^2} \\ & \leq
c \wnorm{\mathbf{y}^*(y_0)-\mathbf{y}_\eps(y_0)}+c \eps \wnorm{\mathbf{y}_\eps(y_0)} \leq c \eps.
\end{align*}
Finally according to Proposition~\ref{thm:existspert} the solution $\mathbf{y}_\eps(y_0) $ is unique and  the mapping~$\mathbf{y}_\eps $ is at least of class~$\mathcal{C}^1$ .
\end{proof}

Next we estimate the~$W^{1,2}$ difference between~$\mathbf{y}_\eps$ and~$\mathbf{y}^*$.
\begin{prop} \label{thm:aprioriW12}
The mapping~$\mathbf{y}_\eps \in \mathcal{C}^1(Y_0;W_T)$ from Theorem~\ref{thm:existenceneuralnetwork} satisfies
\begin{align*}
\|\mathbf{y}_\eps-\mathbf{y}^*\|_{\mathcal{C}^1(Y_0;W_T)} \leq c\eps
\end{align*}
for~$c>0$ independently of~$\eps$ small enough.
\end{prop}
\begin{proof}
By the previous proposition the estimate is already known for ${\mathcal{C}^1(Y_0;W_T)}$ replaced by  ${\mathcal{C}(Y_0;W_T)}$.
Now fix~$y_0 \in Y_0$ and~$i \in \{1, \dots,n\}$. By the inverse mapping theorem the partial derivatives of~$\mathbf{y}^*$ and~$\mathbf{y}_\eps$ at~$y_0$  are given by~$\partial_i \mathbf{y}^*(y_0)= T_*(y_0)^{-1}(0,e_i)$,~$\partial_i \mathbf{y}_\eps(y_0)= T_\eps (y_0)^{-1}(0,e_i)$. Here,~$e_i$ denotes the i-th canonical basis vector in~$\R^n$ and
\begin{align*}
T_*(y_0)^{-1}, T_\eps(y_0)^{-1} \colon L^2(I;\R^n) \times \R^n \to W_T
\end{align*}
denote the linear continuous inverses of
\begin{align*}
T_*(y_0)\delta y\!=\!\left(
\begin{array}{c}
\!\!\!\dot{\delta y}-D\mathbf{f}(\mathbf{y}^*(y_0))\delta y-\lbrack D\mathbf{g}(\mathbf{y}^*(y_0))\delta y\rbrack\F^*(\mathbf{y}^*(y_0))-\mathbf{g}(\mathbf{y}^*(y_0))D\mathcal{F}^*(\mathbf{y}^*(y_0))\delta y\\
\delta y(0)\\
\end{array}
\!\!\!\right)
\end{align*}
and
\begin{align*}
T_\eps(y_0)\delta y\!=\!\left(
\begin{array}{c}
\!\!\!\dot{\delta y}-D\mathbf{f}(\mathbf{y}_\eps(y_0))\delta y-\lbrack D\mathbf{g}(\mathbf{y}_\eps(y_0))\delta y\rbrack\F^\eps_{\theta_\eps}(\mathbf{y}_\eps(y_0))-\mathbf{g}(\mathbf{y}_\eps(y_0))D\mathcal{F}^\eps_{\theta_\eps}(\mathbf{y}_\eps(y_0))\delta y\\
\delta y(0)\\
\end{array}
\!\!\!\right).
\end{align*}
Using Gronwall's inequality, we readily verify that
\begin{align} \label{eq:uniformT}
\max \left\{ \|T_\eps (y_0)^{-1}(\delta v,  \delta y_0 )\|_{W_T},\|T_*(y_0)^{-1}(\delta v,  \delta y_0 )\|_{W_T} \right\} \leq C(\|\delta v\|_{L^2(I;\R^n)}+|\delta y_0|_{\R^n})
\end{align}
for all~$ \delta v \in L^2(I,\R^n),~\delta y_0 \in \R^n$,~$y_0 \in Y_0$ and some~$C>0$ independent of~$y_0, \delta v, \delta y_0 $. Now we recall that~$\mathbf{y}_\eps(y_0) , \mathbf{y}^*(y_0)\in \mathcal{Y}_{ad}$ and that~$D \mathbf{f}, D \mathbf{g},\mathbf{g}$ are Lipschitz continuous, and thus in particular bounded, on~$\mathcal{Y}_{ad}$, see Assumption~\ref{ass:feedbacklaw} $\mathbf{A.1}$. Together with~boundedness of $\{ \|\F^*(\mathbf{y}(y_0))\|_{L^2}: y_0 \in Y_0\}$, Corollary~\ref{lem:calmness} and Theorem~\ref{thm:existenceneuralnetwork} we conclude
\begin{align} \label{eq:kk14}
\|(T_*(y_0)-T_\eps(y_0))\delta y\|_{L^2(I;\R^n)\times \R^n} \leq c \eps \wnorm{\delta y} \quad \forall \delta y \in W_T
\end{align}
for some~$c>0$ again independent of~$y_0 \in Y_0$. Recalling that $B^{-1} - A^{-1}=A^{-1}(A-B)B^{-1}$ for invertible bounded linear operators $A$ and $B$, we obtain
\begin{align*}
\wnorm{\partial_i\mathbf{y}_\eps(y_0)-\partial_i \mathbf{y}^*(y_0)}&= \wnorm{T_\eps (y_0)^{-1}(0,e_i)-T_* (y_0)^{-1}(0,e_i)} \\ & \leq  C^2 \sup_{\|\delta y\|_{W_T} \leq 1} \|(T_*(y_0)-T_\eps(y_0))\delta y\|_{L^2(I;\R^n)\times \R^n}  \leq c \eps,
\end{align*}
where~$C>0$ is the constant from~\eqref{eq:uniformT}. Since all involved constants are independent of~$y_0 \in Y_0$ we obtain the desired estimate
$
\|\partial_i \mathbf{y}_\eps-\partial_i \mathbf{y}^*\|_{\mathcal{C}} \leq c \eps.
$
\end{proof}
Next we address the solvability of the adjoint equation~\eqref{eq:adjointpropapprox}.
\begin{prop} \label{prop:solvofadjoint}
There exists a constant $c$ such that for all $\eps$ small enough
there exists~$\mathbf{p}_\eps \in \mathcal{C}(Y_0;W_T)$ such that~$p_\eps \coloneqq \mathbf{p}_\eps(y_0)\in W_T$ satisfies~\eqref{eq:neuraladjoint} for all~$y_0 \in Y_0$ and
\begin{align*}
\|\mathbf{p}_\eps-\mathbf{p}^*\|_{\mathcal{C}} \leq c \eps.
\end{align*}
\end{prop}
\begin{proof}
Given~$y\in \mathcal{Y}_{ad}$ consider the linear ordinary differential equation
\begin{align*}
-\dot{p}= D\mathbf{f}(y)p+\lbrack D \mathbf{g}(y)^\top \F^\eps_{\theta_\eps}(y)\rbrack p+\mathbf{Q}_1^\top \mathbf{Q}_1(y-y_d),~p(T)= Q^\top_2 Q_2(y(T)-y^T_d).
\end{align*}
It admits a unique solution~$p=P(y)\in W_T$ which is bounded independently of~$y\in \mathcal{Y}_{ad}$. Moreover the mapping~$P \colon W_T \to W_T$ is continuous on~$\mathcal{Y}_{ad}$ in virtue of the Gronwall lemma and Assumption~\ref{ass:feedbacklaw}. The existence of a mapping~$\mathbf{p}_\eps$ which satisfies~\eqref{eq:neuraladjoint} then follows by setting~$\mathbf{p}_\eps=P \circ \mathbf{y}_\eps$.

It remains to prove the estimate for the difference between~$\mathbf{p}_\eps$ satisfying \eqref{eq:neuraladjoint} and~$\mathbf{p}^*$ satisfying \eqref{eq:adjointprop}. For this purpose
we can use the same technique as in the proof of Proposition \ref{thm:aprioriW12} and therefore we only give the main estimates.
Recall that~$D \mathbf{f}(\cdot)^\top ,D \mathbf{g}(\cdot)^\top$ are Lipschitz continuous on~$\mathcal{Y}_{ad}$.
The the most involved term in the estimate analogous to \eqref{eq:kk14} is
\begin{align*}
\|\lbrack D \mathbf{g}(\mathbf{y}_\eps(y_0))^\top &\F^\eps_{\theta_\eps}(\mathbf{y}_\eps(y_0))-D \mathbf{g}(\mathbf{y}^*(y_0))^\top \F^*(\mathbf{y}^*(y_0))\rbrack \delta p\|_{L^2}\\& \leq c(\wnorm{\mathbf{y}_\eps(y_0)-\mathbf{y}^*(y_0)}+\|\F^\eps_{\theta_\eps}(\mathbf{y}_\eps(y_0))-\F^*(\mathbf{y}^*(y_0))\|_{L^2}) \wnorm{\delta p}
\end{align*}
with~$c>0$ independent of~$\eps>0$ and $\delta p \in W_T$. Now a perturbation argument as in the proof of Propostion \ref{thm:aprioriW12} provides us with
\begin{align*}
\wnorm{&\mathbf{p}_\eps(y_0)-\mathbf{p}^*(y_0)}\\& \leq c (\wnorm{\mathbf{y}_\eps(y_0)-\mathbf{y}^*(y_0)}+\|\F^\eps_{\theta_\eps}(\mathbf{y}_\eps(y_0))-\F^*(\mathbf{y}^*(y_0))\|_{L^2}+|\mathbf{y}_\eps(y_0)(T)-\mathbf{y}^*(y_0)(T)|) \\ & \leq c(\wnorm{\mathbf{y}_\eps(y_0)-\mathbf{y}^*(y_0)}+\|\F^\eps_{\theta_\eps}(\mathbf{y}_\eps(y_0))-\F^*(\mathbf{y}^*(y_0))\|_{L^2}) \leq c \eps
\end{align*}
where~$W_T \hookrightarrow \mathcal{C}(\bar{I};\R^n)$ is used
in the second inequality, and Proposition ~\ref{thm:aprioriW12} and Corollary \ref{lem:calmness} are utilized in the final one. Since all involved constants are again independent of~$y_0 \in Y_0$, this finishes the proof.
\end{proof}
Summarizing all previous observations we arrive at the proof of  Theorem~\ref{thm:existenceneuralnetwork2}.
\begin{proof}[Proof of Theorem~\ref{thm:existenceneuralnetwork2}]
This follows directly by combining  Proposition~\ref{thm:existenceneuralnetwork}, Proposition~\ref{thm:aprioriW12}, and Proposition~\ref{prop:solvofadjoint}.
\end{proof}
\subsection{Closedness of~$\mathcal{N}^\eps_{ad}$} \label{subsec:closedandexist}
As a last prerequisite for proving existence to~\eqref{def:approxfeedprop} we argue that the admissible set~$\mathcal{N}^\eps_{ad}$ is closed. The existence of at least one minimizing triple to~\eqref{def:approxfeedprop} then follows by variational arguments. From here on we always assume that~$\mathcal{N}^\eps_{ad}$  from \eqref{def:admissibleset} is nonempty, i.e. that $\eps$ is sufficiently small.

\begin{prop} \label{prop:closedofNad}
Let~$(\mathbf{y}_k,\mathbf{p}_k,\theta_k)_{k\in\N} \subset \mathcal{N}^\eps_{ad} $ be a sequence with weak limit~$(\mathbf{y},\mathbf{p},\theta)$ in~$L^2(Y_0;W_T)^2 \times \mathcal{R}_\eps$. Then
$(\mathbf{y},\mathbf{p},\theta) \in \mathcal{N}^\eps_{ad} $ and we have
\begin{align*}
(\mathbf{y},\mathbf{p}) \in \mathcal{C}(Y_0;W_T)^2,~\lim_{k\rightarrow \infty} \mathbf{y}_k(y_0)=\mathbf{y}(y_0) \text{ and}~\lim_{k\rightarrow \infty} \mathbf{p}_k(y_0)=\mathbf{p}(y_0) \text{ in } W_T, \quad \forall~y_0 \in Y_0.
\end{align*}
\end{prop}
The proof builds upon the following two lemmas.
\begin{lemma} \label{lem:weakclosedstate}
Let the sequence $(\mathbf{y}_k,\mathbf{p}_k,\theta_k)_{k\in\N} \subset \mathcal{N}^\eps_{ad} $ satisfy the prerequisites of Proposition~\ref{prop:closedofNad}. Then~$\mathbf{y} \in \mathbf{Y}_{ad}$, $\mathbf{y}_k(y_0)\rightarrow  \mathbf{y}(y_0)\text{ in}~W_T,$ $\F^\eps_{\theta_k}(\mathbf{y}_k(y_0))\rightarrow \F^\eps_{\theta}({\mathbf{y}}(y_0))~\text{in}~L^\infty (I;\R^m)$, and
\begin{align}\label{eq:aux8}
\dot{\mathbf{y}}(y_0)= \mathbf{f}({\mathbf{y}}(y_0))+ \mathbf{g}({\mathbf{y}}(y_0))\F^\eps_{\theta}({\mathbf{y}}(y_0)),~{\mathbf{y}}(y_0)(0)=y_0,
\end{align}
for all~$y_0 \in Y_0$.
\end{lemma}
\begin{proof}
By assumption we have~$\mathbf{y}_k \in \mathbf{Y}_{ad}$, and hence $\|\mathbf{y}_k(y_0)\|_{W_T} \le 2M_{Y_0}$ for all $k \in \N$ and $y_0 \in Y_0$,  and $\mathbf{y}_k \in \mathcal{C}^1(Y_0;W_T)$ for all~$k \in \N$, see Proposition \ref{thm:existenceneuralnetwork}.  Let us fix an arbitrary $y_0 \in Y_0$.
and set~$y_k \coloneqq \mathbf{y}_k(y_0)$ for abbreviation. Then there exists a subsequence, denoted by the same index, and~$\tilde y\in W_T$ such that~$y_k \rightharpoonup \tilde y $ in $W_T$. Since~$W_T \hookrightarrow_c \mathcal{C}(\bar{I};\R^n) \hookrightarrow L^p(I;\R^n)$,~$1\leq p \leq +\infty$, we immediately get
\begin{align*}
y_k(0) \rightarrow \tilde y(0)~\text{in}~\R^n,~\mathbf{f}(y_k) \rightarrow \mathbf{f}(\tilde y)~\text{in}~L^2(I;R^n),~\mathbf{g}(y_k) \rightarrow \mathbf{g}(\tilde y)~\text{in}~\mathcal{B}(L^2(I;\R^m),L^2(I;\R^n))
\end{align*}
as well as
$\F^\eps_{\theta}(y_k) \rightarrow \F^\eps_{\theta}(\tilde y)~\text{in}~L^\infty(I;\R^m). $
Moreover  by Assumption~\ref{ass:approxsmoothness}  for every~$\delta >0$ there exits~$K_\delta\in \N$ such that
\begin{align}\label{eq:kk15}
|\partial_y V^\eps_{\theta_k}(t,y)-\partial_y V^\eps_{\theta}(t,y)|\leq \delta \quad \forall (t,y)\in \bar{I} \times \bar{B}_{2\widehat{M}}(0)
\end{align}
for all~$k\geq K_\delta$. Here~$\widehat{M}$ denotes the constant from Assumption~\ref{ass:feedbacklaw}~$\mathbf{A.2}$. For all such~$k$ we get utilizing \eqref{eq:kk15} for a constant $c$ independent of $k$
 \begin{align*}
\|\F^\eps_{\theta_k}(y_k)-\F^\eps_{\theta}(\tilde y)\|_{L^\infty} &\leq c\|\F^\eps_{\theta_k}(y_k)-\F^\eps_{\theta}(y_k)\|_{L^\infty}+\|\F^\eps_{\theta}(y_k)-\F^\eps_{\theta}(\tilde y)\|_{L^\infty} \\ & \leq c \delta + \|\F^\eps_{\theta}(y_k)-\F^\eps_{\theta}(\tilde y)\|_{L^\infty}.
\end{align*}
This implies that ~$\lim_{k \to \infty} \F^\eps_{\theta_k}(y_k)=\F^\eps_{\theta}(\tilde y)$ in~$L^\infty(I;\R^m)$.
 These observations   imply
\begin{align*}
\dot{y}_k= \mathbf{f}(y_k)+\mathbf{g}(y_k)\F^\eps_{\theta_k}(y_k) \rightarrow \mathbf{f}(\tilde y)+\mathbf{g}(\tilde y)\F^\eps_{\theta}(\tilde y).
\end{align*}
Together with  $y_k \rightharpoonup \tilde y $ in $W_T$ this implies that
$\mathbf{y}_k(y_0)=y_k \rightarrow \tilde y$ in~$W_T$ and
\begin{align} \label{eq:closedloopaux14}
\dot{\tilde y}=\mathbf{f}(\tilde y)+\mathbf{g}(\tilde y)\F^\eps_{\theta}(\tilde y) ,~\tilde y(0)=y_0.
\end{align}
Since the solution to this equation is unique,  every weak accumulation point of~$y_k$ satisfies~\eqref{eq:closedloopaux14} and we have
$\mathbf{y}_k(y_0)\rightarrow \tilde y~\text{in}~W_T$
for the whole sequence. We repeat this construction for all $y_0\in Y_0$. This defines a function $ \tilde {\bf{y}}:Y_0 \to W_T$ such that $\mathbf{y}_k(y_0)\rightarrow \tilde{\bf{y}}(y_0)~\text{in}~W_T$  and such that \eqref{eq:closedloopaux14} is satisfied with $\tilde y= \tilde {\bf{y}}(y_0)$ for each $y_0\in Y_0$.
By Proposition \ref{thm:existenceneuralnetwork} it is the unique solution to \eqref{eq:aux8}.

Lebesgue's dominated convergence theorem for Bochner integrals \cite[pg 45]{DU77}  implies that $\mathbf{y}_k\rightarrow \tilde {\bf{y}}$ in $L^1(Y_0;W_T)$, and by boundedness of $\{\|\mathbf{y}_k\|_{\mathcal{C}}\}_{k=1}^\infty$  also in $L^2(Y_0;W_T)$.
By assumption ${\bf{y}}_k$ converges weakly in $L^2(Y_0;W_T)$ to $\bf{y}$. Thus we have
 $\bf{y}=\tilde {\bf{y}}$. Moreover $\|\mathbf{y}\|_{\mathcal{C}} \le 2M_{Y_0}$ and hence $\mathbf{y} \in \mathbf{Y}_{ad}$.
\end{proof}

Next we consider the behavior of the adjoint states~$\mathbf{p}_k$.
\begin{lemma} \label{lem:weakclosedadjoint}
Let  $(\mathbf{y}_k,\mathbf{p}_k,\theta_k)_{k\in\N} \subset \mathcal{N}^\eps_{ad} $ be a sequence with weak limit~$(\mathbf{y},\mathbf{p},\theta)$ satisfying  the prerequisites of Proposition~\ref{prop:closedofNad}. Then~$\|\mathbf{p}_k\|_{\mathcal{C}}\leq C$ for some~$C>0$ and all~$k \in\N$ large enough, and~$\mathbf{p} \in \mathcal{C}(Y_0;W_T)$. Moreover~$\mathbf{p}_k(y_0)\rightarrow \mathbf{p}(y_0)~\text{in}~W_T$, and
\begin{equation} \label{eq:adjointlim}
\begin{array}{ll}
-\dot{\mathbf{p}}(y_0) = D\mathbf{f}(\mathbf{y}(y_0))^\top\mathbf{p}(y_0)+ [D\mathbf{g}(\mathbf{y}(y_0))^\top\F^\eps_{\theta}(\mathbf{y}(y_0))]{\mathbf{p}}(y_0)+ \mathbf{Q}_1^\top \mathbf{Q}_1(\mathbf{y}(y_0)-y_d),\\[1.6ex]
\;\mathbf{p}(y_0))(T)= Q^\top_2 Q_2(\mathbf{y}(T)(y_0)-y^T_d),
\end{array}
\end{equation}
for all~$y_0 \in Y_0$.
\end{lemma}
\begin{proof}
From Lemma \ref{lem:weakclosedstate} recall that for the  sequences~$~y_k\coloneqq \mathbf{y}_k(y_0) \in \mathcal{Y}_{ad}$ and~$y \coloneqq \mathbf{y}(y_0)$ we have  for each $y_0\in Y_0$
\begin{align*}
y_k \rightarrow y~\text{in}~ W_T,~ \F^\eps_{\theta_k}(y_k) \rightarrow \F^\eps_{\theta}(y)~\text{in}~L^\infty(I;\R^m).
\end{align*}
Further for each~$k\in\N$ and~$y_0 \in Y_0 $, the element  $p_k \coloneqq \mathbf{p}_k(y_0) \in W_T$ satisfies
\begin{align} \label{eq:adjointaux}
-\dot{p}_k= D\mathbf{f}(y_k)^\top p_k+\lbrack D \mathbf{g}(y_k)^\top \F^\eps_{\theta_k}(y_k)\rbrack p_k+\mathbf{Q}_1^\top \mathbf{Q}_1(y_k-y_d),~p_k(T)= Q^\top_2 Q_2(y_k(T)-y^T_d).
\end{align}
 Recall from  Assumption~\ref{ass:approxsmoothness} that
~$\partial_y V_{\cdot}^\eps$  is uniformly continuous on compact sets. Thus for every~$\delta>0$ there is~$K_\delta \in \N$ such that
\begin{align*}
|F^\eps_{\theta_k}(t,x)| \leq |F^\eps_{\theta_k}(t,x)-F^\eps_{\theta}(t,x)|+ |F^\eps_{\theta}(t,x)| \leq \delta+ \max_{(t,x)\in I \times \bar{B}_{2 \widehat{M}}(0)} |F^\eps_{\theta}(t,x)|  < \infty
\end{align*}
for all~$(t,x)\in I \times \bar{B}_{2 \widehat{M}}(0)$ and~$k\geq K_\delta$. Consequently we obtain
\begin{align*}
 \sup_{k \ge K_\delta} \max_{(t,x)\in I \times \bar{B}_{2 \widehat{M}}(0)} \|A_k(t,x)\|_{\R^{n\times n}} < \infty, \text{ where }
 A_k(t,x)= Df(t,x)^{\top}+ D g(t,x)^\top F^\eps_{\theta_k}(t,x).
\end{align*}

Applying Proposition~\ref{prop:s} to the time-reversed equation \eqref{eq:adjointaux} implies that
\begin{align*}
\wnorm{p_k} \leq c \left( \|\mathbf{Q}_1^\top \mathbf{Q}_1(y_k-{y}_d)\|_{L^2}+ |y_k(T)-y^T_d| \right)
\end{align*}
for some~$c>0$ independent of~$y_0 \in Y_0$ and all sufficiently large $k$. Since~$\|\mathbf{y}_k\|_{\mathcal{C}}\leq 2M_{Y_0}$ we finally conclude~$\|\mathbf{p}_k\|_{\mathcal{C}} \leq C$ for some~$C>0$ independent of~$k$ sufficiently large.
We are now prepared to pass to the limit in \eqref{eq:adjointaux}.
For this purpose we proceed as in the proof of Lemma \ref{lem:weakclosedstate} and use
\begin{align*}
D \mathbf{f}(y_k)+ D \mathbf{g}(y_k)^\top \F^\eps_{\theta_k}(y_k) \rightarrow D \mathbf{f}(y)+ D \mathbf{g}(y)^\top \F^\eps_{\theta}(y)~\text{in}~\mathcal{B}(L^2(Y;\R^n)),
\end{align*}
as well as
\begin{align*}
\mathbf{Q}_1^\top \mathbf{Q}_1(y_k-{y}_d) \rightarrow \mathbf{Q}_1^\top \mathbf{Q}_1(y-{y}_d)~\text{in}~L^2(I;\R^n),
\end{align*}
and
\begin{align*}
 Q^\top_2 Q_2(y_k(T)-y^T_d) \rightarrow Q^\top_2 Q_2(y(T)-y^T_d)~\text{in}~\R^n
\end{align*}
to show that every weak accumulation point~$\tilde p \in W_T$ of~$p_k$ is in fact a strong accumulation point and  satisfies the differential equation in~\eqref{eq:adjointlim}. Since the solution to this equation is unique we get~$p_k \rightarrow \tilde p $ in~$W_T$ for the whole sequence. Finally utilizing~$\|\mathbf{p}_k\|_{\mathcal{C}} \leq C$ and Lebesgue's dominated convergence theorem we conclude~$\tilde p=\mathbf{p}(y_0)$ for all $y_0 \in Y_0$.
\end{proof}
\begin{proof}[Proof  of Proposition~\ref{prop:closedofNad}]
This is a direct consequence of  Lemma~\ref{lem:weakclosedstate} and Lemma~\ref{lem:weakclosedadjoint}.
\end{proof}
\subsection{Existence of minimizers}
Finally we prove the existence of at least one minimizing triplet to~\eqref{def:approxfeedprop}.
\begin{theorem} \label{thm:existence}
Let Assumption~\ref{ass:feedbacklaw} and~\ref{ass:approxsmoothness} hold. Then for all~$\eps>0$ small enough, Problem~\eqref{def:approxfeedprop} admits at least one minimizing triplet~$(\mathbf{y}^*_\eps, \mathbf{p}^*_\eps, \theta^*_\eps ) \in \mathcal{C}(Y_0;W_T)^2 \times \mathcal{R}_\eps $.
\end{theorem}
\begin{proof}
According to Theorem~\ref{thm:existenceneuralnetwork2}, the admissible set~$\mathcal{N}^\eps_{ad}$ is nonempty for~$\eps>0$ small enough. Fix such a~$\eps>0$ and let
~$(\mathbf{y}_k,\mathbf{p}_k, \theta_k) \in \mathcal{N}^\eps_{ad}$ denote a minimizing sequence for~$\mathcal{J}_\eps$ i.e.
\begin{align*}
\mathcal{J}_\eps(\mathbf{y}_k,\mathbf{p}_k, \theta_k) \rightarrow \inf_{(\mathbf{y},\mathbf{p}, \theta) \in \mathcal{N}^\eps_{ad}} \mathcal{J}_\eps(\mathbf{y},\mathbf{p}, \theta).
\end{align*}
Since~$\mathbf{y}_k \in \mathbf{Y}_{ad}$ and
$
\frac{\gamma_\eps}{2} \|\theta_k\|^2_{\mathcal{R}_\eps} \leq \mathcal{J}_\eps(\mathbf{y}_k,\mathbf{p}_k, \theta_k),
$ for all $k\in\N$,
the sequence~$\{(\mathbf{y}_k,\theta_k)\}\in L^2(Y_0;W_T)\times \mathcal{R}_\eps$ is bounded. Thus it admits at least one subsequence, denoted by the same index, with
\begin{align*}
(\mathbf{y}_k,\theta_k) \rightharpoonup (\mathbf{y}^*_\eps,\theta^*_\eps) ~\text{in}~L^2(Y_0;W_T) \times \mathcal{R}_\eps
\end{align*}
for some~$(\mathbf{y}^*_\eps,\theta^*_\eps)$. As in the proof of Lemma~\ref{lem:weakclosedadjoint} we verify that ~$\|\mathbf{y}_k\|_{\mathcal{C}} \leq C$ and  ~$\|\mathbf{p}_k\|_{\mathcal{C}} \leq C$ for some~$C>0$ independent of~$k\in\N$. Consequently, by possibly taking another subsequence  we arrive at
\begin{align*}
(\mathbf{y}_k,\mathbf{p}_k,\theta_k) \rightharpoonup (\mathbf{y}^*_\eps,\mathbf{p}^*_\eps,\theta^*_\eps) ~\text{in}~L^2(Y_0;W_T)^2 \times \mathcal{R}_\eps
\end{align*}
for some~$(\mathbf{y}^*_\eps,\mathbf{p}^*_\eps,\theta^*_\eps)\in \mathcal{N}^\eps_{ad}$. For the following estimates it will be convenient  to recall the augmented functional~$J_\eps$, see~\eqref{eq:aux2}, which arises in the running cost of ~\eqref{def:approxfeedprop} in compact form:
\begin{align}\label{eq:aux2a}
J_\eps(y,p,\theta)= J(y, \F^\eps_\theta(y))+ \frac{\gamma_1}{2} \|\mathcal{V}(y)-J_\bullet (y, \F^\eps_{\theta^*_\eps}(y))\|^2_{L^2(I;\R)}+\frac{\gamma_2}{2}\|p-\partial_y \mathcal{V}(y)\|^2_{L^2(I;\R^n)},
\end{align}
where $J_t$ was defined below \eqref{def:valuefunc}.
Now fix an arbitrary~$y_0\in Y_0$ and set
\begin{align*}
y_k \coloneqq \mathbf{y}_k(y_0),~p_k \coloneqq \mathbf{p}_k(y_0),~{y^*} \coloneqq \mathbf{y}^*_\eps (y_0),~p \coloneqq \mathbf{p}^*_\eps (y_0).
\end{align*}
From Lemma~\ref{lem:weakclosedstate} and Lemma~\ref{lem:weakclosedadjoint} we get
\begin{align*}
y_k \rightarrow \tilde y,~p_k \rightarrow p~\text{in}~W_T,~\F^\eps_{\theta_k}(y_k) \rightarrow \F^\eps_{\theta^*_\eps}(\tilde y)~\text{in} ~L^2(I;\R^n)
\end{align*}
and, again using the uniform continuity of~$V_{\bullet}^\eps$ and~$\partial_y V_{\bullet}^\eps$, we conclude
\begin{align*}
\mathcal{V}^\eps_{\theta_k}(y_k) \rightarrow \mathcal{V}^\eps_{\theta^*_\eps}(\tilde y)~\text{in}~L^2(I),~ \partial_y \mathcal{V}^\eps_{\theta_k}(y_k) \rightarrow \partial_y \mathcal{V}^\eps_{\theta^*_\eps}(\tilde y)~\text{in}~L^2(I;\R^n),
\end{align*}
as well as the uniform boundedness of~$\mathcal{V}^\eps_{\theta_k}(\mathbf{y}_k)$ and~$\partial_y \mathcal{V}^\eps_{\theta_k}(\mathbf{y}_k)$ in~$\mathcal{C}(Y_0;L^2(I))$ and $\mathcal{C}(Y_0;L^2(I;\R^n))$, respectively.
Moreover we readily verify that
\begin{align*}
|J_t(y_k, \F^\eps_{\theta_k}(y_k))\!-\!J_t(\tilde y, \F^\eps_{\theta^*_\eps}(\tilde y))|\! \leq\! c\!\left(\|y_k-\tilde y\|_{L^2}+\|\F^\eps_{\theta_k}(y_k)\!-\!\F^\eps_{\theta^*_\eps}(\tilde y)\|_{L^2}+ |y_k(T)- \tilde y(T)| \right),
\end{align*}
for some~$c>0$ independent of~$y_0 \in Y_0$, $t\in (0,T)$,   and~$k\in \N$. Thus we arrive at
\begin{align*}
J_{{\bullet}} (y_k, \F^\eps_{\theta_k}(y_k)) \rightarrow J_\bullet (y, \F^\eps_{\theta^*_\eps}(y))~\text{in}~L^\infty(I).
\end{align*}
Summarizing the previous findings there holds
\begin{align*}
\|\mathcal{V}(y_k)\!-\!J_\bullet (y_k, \F^\eps_{\theta_k}(y_k))\|^2_{L^2}\!+\!\|p_k\!-\!\partial_y \mathcal{V}(y_k)\|^2_{L^2}
&\to \|\mathcal{V}(y)\!-\!J_\bcdot (y, \F^\eps_{\theta^*_\eps}(y))\|^2_{L^2}+\|p\!-\!\partial_y \mathcal{V}(y)\|^2_{L^2}
\\  J(y_k,\F^\eps_{\theta_k}(y_k)) & \rightarrow J(y,\F^\eps_{\theta^*_\eps}(y)).
\end{align*}
Using these expressions  in ~$J_\eps$ as given in  ~\eqref{eq:aux2a},
and the boundedness of~$\|y_k\|_{L^2},|y_k(0)|,$
$\|p_k\|_{L^2}~\|\F^\eps_{\theta_k}(y_k)\|_{L^2},~\|\mathcal{V}^\eps_{\theta_k}(y_k)\|_{L^2}$ independent of~$k \in \N$ and~$y_0 \in Y_0$ we finally get by using Lebesgue's dominated convergence theorem
\begin{align*}
\mathcal{J}_{\eps}(\mathbf{y}_k,\mathbf{p}_k, \theta_k) \rightarrow \mathcal{J}_{\eps}(\mathbf{y}^*_\eps,\mathbf{p}^*_\eps, \theta^*_\eps)= \inf_{(\mathbf{y},\mathbf{p}, \theta) \in \mathcal{N}^\eps_{ad}} \mathcal{J}_\eps(\mathbf{y},\mathbf{p}, \theta).
\end{align*}
\end{proof}
\section{Convergence towards optimal controls} \label{sec:convergence}
In Proposition \ref{thm:existenceneuralnetwork} and \ref{prop:solvofadjoint} it was established that the ensemble triple ~$(\mathbf{y}^*, \mathcal{F}(\mathbf{y^*}) , \mathbf{p}^*)$ can be approximated by ensemble triples  ~$(\mathbf{y}_\eps, \mathcal{F}^\eps_{\theta_\eps}(\mathbf{y}_\eps), \mathbf{p}_\eps)$ in the order $O(\eps)$.
In this section, the convergence of solutions to~\eqref{def:approxfeedprop} as~$\eps \rightarrow 0$ is addressed. We first consider the terms in the definition~$\mathcal{J}_\eps$, see \eqref{eq:aux2}. To obtain the desired asymptotic behavior a smallness condition on the regularisation parameter $\gamma_\epsilon$, in relation to the norm of the parameters $\theta_\eps$ describing the approximation quality, is required.

\begin{theorem} \label{thm:convofobj}
Let Assumptions~\ref{ass:feedbacklaw} and~\ref{ass:approxsmoothness} hold the latter with ~$\theta_\eps \in \mathcal{R}_\eps$, and let~$(\mathbf{y}^*_\eps, \mathbf{p}^*_\eps,\theta^*_\eps)$,  denote an optimal triple to~\eqref{def:approxfeedprop} for all $\eps>0$ small enough. If additionally  $\gamma_\eps \|\theta_\eps\|^2_{\mathcal{R}_\eps} = O(\eps)$, then
\begin{align*}
0 \leq  \int_{Y_0} \omega(y_0) \left \lbrack J(\mathbf{y}^*_\eps(y_0),\mathcal{F}^\eps_{\theta^*_\eps}(\mathbf{y}^*_\eps(y_0))-  V^*(0,y_0) \right \rbrack~\mathrm{d} \mathcal{L}(y_0) \leq c\, \eps
\end{align*}
holds and, if~$\gamma_1,\gamma_2 >0$, we also have
\begin{align*}
\int_{Y_0}\!\! \omega(y_0)  (\| V^\eps_{\theta^*_\eps}(t,\mathbf{y}^*_\eps(y_0))\!-\!J_\bullet (\mathbf{y}^*_\eps(y_0), \F^\eps(\mathbf{y}^*_\eps(y_0)))\|^2_{L^2}&+ \|\partial_y V^\eps_{\theta^*_\eps}(t,\mathbf{y}^*_\eps(y_0))\!-\!\mathbf{p}^*_{\eps}(y_0)\|^2_{L^2}) \mathrm{d} \mathcal{L}(y_0)\\ & \leq c \, \eps
\end{align*}
for some~$c>0$ independent of~$\eps$.
\end{theorem}
\begin{proof}
Let~$\mathbf{y}_\eps,\mathbf{p}_\eps$ denote the ensembles of state and adjoint trajectories associated to~$\theta_\eps$, see Theorem~\ref{thm:existenceneuralnetwork2}, for~$\eps>0$ small enough. Then we have
\begin{align*}
\big|&  J(\mathbf{y}_\eps(y_0),\mathcal{F}^\eps_{\theta_\eps}(\mathbf{y}_\eps(y_0))-  V^*(0,y_0) \big|
\\&\leq  C \big ( \|\mathbf{y}_\eps(y_0)-\mathbf{y}^*(y_0)\|_{W_T}+  \|\mathcal{F}^\eps_{\theta_\eps}(\mathbf{y}_\eps(y_0))-\mathcal{F}^*(\mathbf{y}^*(y_0))\|_{L^2(I;\R^m)} \big )  \leq c \eps
\end{align*}
for some~$C>0$ independent of~$\eps$.
Here we have used~$V^*(0,y_0)=J(\mathbf{y}^*(y_0),\mathcal{F}^*(\mathbf{y}^*(y_0)))$ for all~$y_0 \in Y_0$, the embedding~$W_T \hookrightarrow \mathcal{C}(\bar{I};\R^n)$ as well as the a priori estimates of Proposition~\ref{thm:existenceneuralnetwork}. Next we utilize~$\mathbf{p}^*(y_0)=\partial \mathcal{V}^*(\mathbf{y}^*(y_0))$,~$y_0\in Y_0$, to estimate
\begin{align*}
\|\partial_y \mathcal{V}^\eps_{\theta_\eps}&(\mathbf{y}_\eps(y_0))-\mathbf{p}_{\eps}(y_0)\|^2_{L^2(I;\R^n)} \\& \leq 2\left( \|\partial_y \mathcal{V}^\eps_{\theta_\eps}(\mathbf{y}_\eps(y_0))-\partial_y \mathcal{V}^*(\mathbf{y}^*(y_0))\|^2_{L^2(I;\R^n)}+\|\mathbf{p}_{\eps}(y_0)-\mathbf{p}^*(y_0)\|^2_{L^2(I;\R^n)}\right) \\&
\leq c\eps^2,
\end{align*}
where the last inequality is deduced from Proposition~\ref{thm:existenceneuralnetwork} and Proposition~\ref{prop:solvofadjoint}. Proceeding analogously and using~$V^*(t,\mathbf{y}^*(y_0)(t))=J_t(\mathbf{y}^*(y_0),\mathcal{F}^*(\mathbf{y}^*(y_0)))$ for all~$y_0 \in Y_0$,~$t\in I$, we obtain
\begin{align*}
\int_{Y_0} \omega(y_0) \int^T_0 |V^\eps_{\theta_\eps}(t,\mathbf{y}_\eps(y_0)(t))-J_t(\mathbf{y}_\eps(y_0),\mathcal{F}^\eps_{\theta_\eps}
(\mathbf{y}_\eps(y_0))|^2~\mathrm{d}t\,\mathrm{d}\mathcal{L}(y_0)\leq D_1+D_2,
\end{align*}
where, using Assumption~\ref{ass:approxsmoothness} and again  Proposition ~\ref{thm:existenceneuralnetwork}
\begin{align*}
D_1 &\coloneqq \int_{Y_0} \omega(y_0) \|\mathcal{V}^\eps_{\theta_\eps}(\mathbf{y}_\eps(y_0))-\mathcal{V}^*(\mathbf{y}^*(y_0))\|^2_{L^2(I)}~\mathrm{d} \mathcal{L}(y_0) \leq c \eps^2,  \\
D_2 & \coloneqq \int_{Y_0}\omega(y_0)\int^T_0 |J_t(\mathbf{y}^*(y_0),\mathcal{F}^*(\mathbf{y}^*(y_0))-J_t(\mathbf{y}_\eps(y_0),\mathcal{F}^\eps_{\theta_\eps}(\mathbf{y}_\eps(y_0))|^2~\mathrm{d}t\mathrm{d}\mathcal{L}(y_0) \leq c \eps^2.
\end{align*}
Combining the previous estimates with the optimality of~$(\mathbf{y}^*,\mathbf{p}^*,\theta^*_\eps)$, and the assumption on the asymptotic behavior of $\gamma_\epsilon$ we deduce that
\begin{align*}
0 &\leq  \int_{Y_0} \omega(y_0)\left \lbrack J_\eps(\mathbf{y}^*_\eps(y_0),\mathbf{p}^*_\eps(y_0),\theta^*_\eps)-  V^*(0,y_0) \right \rbrack~\mathrm{d} \mathcal{L}(y_0) +\frac{\gamma_\eps}{2} \|\theta^*_\eps\|^2_{\mathcal{R}_\eps}
\\ &\leq \int_{Y_0} \omega(y_0) \left \lbrack J_\eps(\mathbf{y}_\eps(y_0),\mathbf{p}_\eps(y_0),\theta_\eps)-  V^*(0,y_0) \right \rbrack~\mathrm{d} \mathcal{L}(y_0) +\frac{\gamma_\eps}{2} \|\theta_\eps\|^2_{\mathcal{R}_\eps}
\leq c\, \eps.
\end{align*}
Recalling the definition of $J_\eps$, this yields all claimed estimates and finishes the proof.
\end{proof}

Next  the convergence  of the ensemble trajectories~$(\mathbf{y}^*_\eps, \mathbf{p}^*_\eps)$, the feedback controls~$\mathcal{F}^\eps_{\theta^*_\eps}(\mathbf{y}^*_\eps)$ as well as the approximate value function~$\mathcal{V}^\eps_{\theta^*_\eps}$ are analyzed. For this purpose we make use of the additional regularity of ensemble solutions to the closed loop system, see Proposition~\ref{thm:aprioriW12}},  and introduce further constraints to~\eqref{def:approxfeedprop}. Without changing the notation  we henceforth set
\begin{align} \label{def:Yadstrict}
\mathbf{\hat Y}_{ad} = \left\{\, \mathbf{y} \in \mathcal{C}^1(Y_0;W_T)\;|\;\Cbochnorm{\mathbf{y}} \leq 2 M_{Y_0},~\|\mathbf{y}\|_{W^{1,2}}\leq 2 M_{W^{1,2}}\, \right\},
\end{align}
where~$M_{W^{1,2}}>0$ is a constant with~$\|\mathbf{y}^*\|_{W^{1,2}} \leq M_{W^{1,2}} $, the function $y^*$ was introduced in \textbf{A.3}, and  $W^{1,2}= \{\mathbf{y} \in L^2(I;W_T): \partial_i \mathbf{y} \in L^2(Y_0;W_T), i\in \{1,\dots,n\} \}$ endowed with the natural norm. Next we note that
\begin{align*}
\frac{\beta}{2} \|\F^*(\mathbf{y}^*(y_0))\|^2_{L^2} \leq J(\mathbf{y}^*(y_0),\F^*(\mathbf{y}^*(y_0)))=V^*(0,y_0)
\end{align*}
for all~$y_0 \in Y_0$. Thus, due to the continuity of the value function~$V^*$, see Assumption~\ref{ass:feedbacklaw} $\mathbf{A.2}$, there is~$M_U >0$ with~$\|\F^*(\mathbf{y}^*)\|_{L^\infty} \leq M_U$. Correspondingly we set
\begin{align} \label{def:Uadstricter}
\mathbf{\hat U}_{ad} = \left\{\, \mathbf{u} \in L^\infty(Y_0;L^2(I;\R^m))\;|\; \|\mathbf{u}\|_{L^\infty}\leq 2 M_U \,\right\}.
\end{align}
We point out that Theorem~\ref{thm:existence} remains valid despite the additional restriction of the set of admissible states and controls. Problem~\eqref{def:approxfeedprop}  with $\mathbf{ Y}_{ad}, \mathbf{ U}_{ad}$ replaced by $\mathbf{\hat Y}_{ad}, \mathbf{\hat U}_{ad}$ will be denoted by $(\mathcal{\hat P}_\epsilon) $.

\begin{prop} \label{prop:existencestrict}
Let Assumption~\ref{ass:feedbacklaw} and~\ref{ass:approxsmoothness} hold.  Then for all~$\eps>0$ small enough, Problem $(\mathcal{\hat P}_\epsilon) $  admits at least one minimizing triple.
\end{prop}
\begin{proof}
Let~$(\mathbf{y}_\eps,\mathbf{p}_\eps, \theta_\eps)$ be defined as in Theorem~\ref{thm:existenceneuralnetwork2}. Then we have~$\mathbf{y}_\eps \in \mathbf{\hat Y}_{ad}$, see Proposition~\ref{thm:existenceneuralnetwork} and Proposition~\ref{thm:aprioriW12}, as well as~$\mathcal{F}^\eps_{\theta_\eps}(\mathbf{y}_\eps) \in \mathbf{\hat U}_{ad}$, according to Proposition~\ref{thm:existenceneuralnetwork}, for all~$\eps>0$ small enough. Hence the admissible set of~ $(\mathcal{\hat P}_\epsilon) $  is not empty. The existence of a minimizing triple then follows by repeating the arguments of the proof of Theorem~\ref{thm:existence} noting that the admissible set
\begin{align*}
\left\{(\mathbf{y},\mathbf{p},\theta)\!\in\! \mathbf{\hat Y}_{ad}\!\times\!\mathcal{C}(Y_0;W_T) \!\times\! \mathcal{R}_{\eps}|(\mathbf{y},\mathbf{p},\theta)~\text{satisfies}\!~\eqref{eq:statepropapprox}\!
-\!\eqref{eq:constraintprop}, \F^\eps_{\theta}(\mathbf{y})\!\in\! \mathbf{\hat U}_{ad}\right\}
\end{align*}
is closed w.r.t to the weak topology on~$L^2(Y_0;W_T)^2 \times \mathcal{R}_\eps$.
\end{proof}

Let us next address the convergence  of the optimal ensemble states~$\mathbf{y}^*_\eps$, adjoint states~$\mathbf{p}_\eps$ and the associated feedback controls~$\F^\eps_{\theta^*_\eps}(\mathbf{y}^*_\eps)$ as $\eps$ tends to 0.
\begin{theorem} \label{thm:convoftraj}
Let the prerequisites of Theorem~\ref{thm:convofobj} hold, and let $\eps_k >0$ be a strictly decreasing null sequence such that $(\mathcal{\hat P}_{\eps_k})$ admits a minimizing triple $(\mathbf{y}^*_k, \mathbf{p}^*_k, \theta^*_k)$.
 Then $(\mathbf{y}^*_{k}, \mathbf{p}^*_{k}, \mathcal{F}^{\eps_k}_{\theta^*_{\eps_k}}(\mathbf{y}^*_{\eps_k}))$ contains at least one  accumulation point~$(\bar{\mathbf{y}}, \bar{\mathbf{p}}, \bar{\mathbf{u}}) \in L^\infty(Y_0; W_T)^2 \times L^\infty(Y_0; L^2(I;\R^m))$ w.r.t the strong topology on~$L^2(Y_0;W_T)^2 \times L^2(Y_0;L^2(I;\R^m))$. For each  accumulation point and~$\mathcal{L}$-a.e.~$y_0 \in Y_0$ we have that~$(\bar{y}, \bar{p}, \bar{u}) \coloneqq (\bar{\mathbf{y}}(y_0),\bar{\mathbf{p}}(y_0),\bar{\mathbf{u}}(y_0))$ satisfies
\begin{align*}
(\bar{y}, \bar{u}) \in \min \eqref{def:openloopproblem}
\end{align*}
as well as
\begin{align*}
\dot{\bar{y}}&=\mathbf{f}(\bar{y})+\mathbf{g}(\bar{y})\bar{u},~ \bar{y}(0)=y_0, \\
-\dot{\bar{p}}&= D\mathbf{f}(\bar{y})^\top \bar{p}+\lbrack D \mathbf{g}(\bar{y})^\top\bar{u})\rbrack \bar{p}+\mathbf{Q}_1^\top \mathbf{Q}_1(\bar{y}-y_d),~ \bar{p}(T)=Q^\top_2 Q_2(\bar{y}(T)-y^T_d).
\end{align*}
\end{theorem}
\begin{proof}
By choice of the admissible sets~$\mathbf{\hat Y}_{ad}$ and~$\mathbf{\hat U}_{ad}$ we have that~$\{(\mathbf{y}^*_k, \F^\eps_{\theta^*_k}(\mathbf{y}^*_k))\}_{k=1}^{\infty} $ is bounded in~$(W^{1,2}(Y_0;W_T) \cap L^\infty(Y_0;W_T))\times L^\infty(Y_0;L^2(I;\R^m))$. By Gronwall's inequality we can argue that ~$\{\mathbf{p}^*_k\}_{k=1}^\infty$  is also  bounded in ~$L^\infty(Y_0;W_T)$. Thus, due to the Banach-Alaoglu theorem, there is a subsequence, denoted by the same index, and
~$(\bar{\mathbf{y}},\bar{\mathbf{p}},\bar{\mathbf{u}})\in L^\infty(Y_0;W_T)^2 \times L^\infty(Y_0;L^2(I;\R^m))$ such that
\begin{align*}
(\mathbf{y}^*_k, \mathbf{p}^*_k, \F^{\eps_k}_{\theta^*_k}(\mathbf{y}^*_k)) \rightharpoonup^* (\bar{\mathbf{y}},\bar{\mathbf{p}},\bar{\mathbf{u}})~\text{in}~L^\infty(Y_0;W_T)^2 \times L^\infty(Y_0;L^2(I;\R^m)),
\end{align*}
and $\dot{\mathbf{y}}^*_k \rightharpoonup \dot{\bar{\mathbf{y}}}$ in $L^2(Y_0;L^2(I;\R^n))$. By the compact embedding of $W^{1,2}(Y_0;W_T)$ into $L^2(Y_0;\mathcal{C}(I;\R^n))$, see \cite[Theorem 5.3]{AK18}  the subsequence can be chosen such that $\mathbf{y}^*_k \to  \bar{\mathbf{y}}$ strongly in $L^2(Y_0;\mathcal{C}(I;\R^n))$.
These properties imply that~$(\bar{y},\bar{u})\coloneqq (\bar{\mathbf{y}}(y_0),\bar{\mathbf{u}}(y_0))$ satisfies
\begin{equation}\label{eq:aux9}
\dot{\bar{y}}=\mathbf{f}(\bar{y})+\mathbf{g}(\bar{y})\bar{u},~ \bar{y}(0)=y_0,
\end{equation}
for~$\mathcal{L}$-a.e.~$y_0 \in Y_0$. This also implies~$V^*(0,y_0) \leq J(\bar{y},\bar{u}) $ and thus, together with
\begin{align*}
J(\mathbf{y}^*_k, \F^{\eps_k}_{\theta^*_k}(\mathbf{y}^*_k)) \rightarrow V^*(0,\cdot)~\text{in}~L^1(Y_0;W_T), \\
\end{align*}
see Theorem~\ref{thm:convofobj},
we have~$(\bar{y},\bar{u}) \in \argmin \eqref{def:openloopproblem}$ for~$\mathcal{L}$-a.e.~$y_0 \in Y_0$. Moreover, again using the strong convergence of $\mathbf{y}^*_k$ in $L^2(Y_0;\mathcal{C}(I;\R^n))$ and recalling the definition of~$J(\cdot,\cdot)$ as
\begin{align*}
J(y,u)= (1/2) \|\mathbf{Q}_1(y-y_d)\|^2_{L^2(I;\R^n)}+(\beta/2) \|u\|^2_{L^2(I;\R^n)}+(1/2) |{Q}_2(y(T)-y^T_d)|^2,
\end{align*}
for all~$y \in W_T,~u \in L^2(I;\R^m)$,
we also conclude the convergence of the~$L^2(Y_0;L^2(I;\R^m))$ norm of~$\F^{\eps_k}_{\theta^*_k}(\mathbf{y}^*_k)$ towards the norm of~$\bar{\mathbf{u}}$. Thus~$\F^{\eps_k}_{\theta^*_k}(\mathbf{y}^*_k) \rightarrow \bar{\mathbf{u}}$ strongly in~$L^2(Y_0;L^2(I;\R^m))$, and  $\mathbf{y}^*_k \to  \bar{\mathbf{y}}$ strongly in $L^2(Y_0;W_T)$, by Lebegue's bounded convergence theorem.

It remains to address the strong convergence of~$\mathbf{p}_k$. For this purpose we show that the functions~$\lbrack D\mathbf{g}(\mathbf{y}^*_k)^\top \F^{\eps_k}_{\theta^*_k}(\mathbf{y}^*_k(\cdot))\rbrack \mathbf{p}^*_k(\cdot) $ converge weakly to~$\lbrack D\mathbf{g}(\bar{\mathbf{y}}(\cdot))^\top \bar{\mathbf{u}}(\cdot)\rbrack \bar{\mathbf{p}}(\cdot) $ in~$L^2(Y_0;L^2(I;\R^n))$. Fixing a test function~$\varphi \in L^2(Y_0;L^2(I;\R^n))$ we first note that
\begin{align*}
\lim_{k\rightarrow \infty}\big(\varphi&, \lbrack D\mathbf{g}(\bar{\mathbf{y}}(\cdot))^\top \bar{\mathbf{u}}(\cdot))\rbrack (\mathbf{p}^*_k(\cdot)-\bar{\mathbf{p}}) \big) _{L^2(Y_0;L^2(I;\R^n))}=0.
\end{align*}
Second, for~$\mathcal{L}$-a.e.~$y_0 \in Y_0$ we estimate
\begin{align*}
&\big(\varphi(y_0),\lbrack D\mathbf{g}(\mathbf{y}^*_k(y_0))^\top \F^{\eps_k}_{\theta^*_k}(\mathbf{y}^*_k(y_0))-D\mathbf{g}(\bar{\mathbf{y}}(y_0))^\top \bar{\mathbf{u}}(y_0)\rbrack \mathbf{p}^*_k(y_0)\big)_{L^2(I;\R^n)} \\
&\leq \! C \|\varphi(y_0)\|_{L^2} \wnorm{\mathbf{p}^*_k(y_0)}\! \left( \|\F^{\eps_k}_{\theta^*_k}(\mathbf{y}^*_k(y_0))\|_{L^2}\wnorm{\mathbf{y}^*_k(y_0)\!-\!\bar{\mathbf{y}}(y_0)}\!+
\!\|\F^{\eps_k}_{\theta^*_k}(\mathbf{y}^*_k(y_0))\!-\!\bar{\mathbf{u}}(y_0)\|_{L^2}\!\right) \\
& \leq \!
C \|\varphi(y_0)\|_{L^2} \left( \wnorm{\mathbf{y}^*_k(y_0)-\bar{\mathbf{y}}(y_0)}+\|\F^{\eps_k}_{\theta^*_k}(\mathbf{y}^*_k(y_0))-\bar{\mathbf{u}}(y_0)\|_{L^2}\right)
\end{align*}
for some~$C>0$ independent of~$k\in\N$ and~$y_0$. Here we made use of the boundedness of~$\{\mathbf{y}^*_k\}_{k=1}^\infty$ and~$\{\mathbf{p}^*_k\}_{k=1}^\infty$ in~$L^\infty(Y_0;W_T)$, and of~$\{\F^{\eps_k}_{\theta^*_k}(\mathbf{y}^*_k)\}_{k=1}^\infty$ in~$L^\infty(Y_0;L^2(I;\R^m))$. Integrating both sides of the inequality w.r.t to~$\mathcal{L}$ and utilizing the strong convergence of~$\mathbf{y}^*_k$ and~$\F^{\eps_k}_{\theta^*_k}(\mathbf{y}^*_k)$ we finally arrive at
\begin{align*}
\lim_{k \rightarrow \infty}\big(\varphi,\lbrack D\mathbf{g}(\mathbf{y}^*_k(\cdot))^\top \F^{\eps_k}_{\theta^*_k}(\mathbf{y}^*_k(\cdot))-D\mathbf{g}(\bar{\mathbf{y}}(\cdot))^\top \bar{\mathbf{u}}(\cdot)\rbrack \mathbf{p}^*_k(\cdot)\big)_{L^2(Y_0,L^2(I;\R^n))}=0.
\end{align*}
By repeating this argument for the different terms appearing in the adjoint equation we get that~$(\bar{y}, \bar{p}, \bar{u}) \coloneqq (\bar{\mathbf{y}}(y_0),\bar{\mathbf{p}}(y_0),\bar{\mathbf{u}}(y_0))$ satisfies
\begin{align*}
-\dot{\bar{p}}&= D\mathbf{f}(\bar{y})^\top \bar{p}+\lbrack D \mathbf{g}(\bar{y})^\top \bar{u}\rbrack \bar{p}+\mathbf{Q}_1^\top \mathbf{Q}_1(\bar{y}-y_d),~ \bar{p}(T)= Q^\top_2 Q_2(\bar{y}(T)-y^T_d)
\end{align*}
for~$\mathcal{L}$-a.e.~$y_0 \in Y_0$. Applying Gronwall's inequality we deduce
\begin{align*}
\wnorm{\mathbf{p}_k(y_0)-\bar{\mathbf{p}}(y_0)} \leq C \left( \wnorm{\mathbf{y}^*_k(y_0)-\bar{\mathbf{y}}(y_0)}+\|\F^{\eps_k}_{\theta^*_k}(\mathbf{y}^*_k(y_0))-\bar{\mathbf{u}}(y_0)\|_{L^2} \right)
\end{align*}
for~$\mathcal{L}$-a.e.~$y_0 \in Y_0$ and~$C>0$ independent of~$y_0$ and~$k$. This yields~$\mathbf{p}_k \rightarrow \bar{\mathbf{p}}$ strongly in~$L^2(Y_0;W_T)$. Since the weakly convergent subsequence was chosen arbitrarily in the beginning, this  finishes the proof.
\end{proof}
\begin{remark}
If~$\mathbf{g}(y(t))=B \in \R^{m \times n}$ then the statement of the previous theorem also holds~\text{without} constraints on the control (i.e. for~$\mathbf{U}_{ad}=L^2(Y_0;L^2(I;\R^m))$). In this particular case, the  uniform boundedness of~$\F^{\eps_k}_{\theta^*_k}(\mathbf{y}^*_k)$ in~$L^2(Y_0;L^2(I;\R^m))$ follows from
\begin{align*}
\frac{\beta}{2} \|\F^{\eps_k}_{\theta^*_k}(\mathbf{y}^*_k)\|^2_{L^2} \leq c \int_{Y_0} \omega(y_0) J(\mathbf{y}^*_k,\F^{\eps_k}_{\theta^*_k}(\mathbf{y}^*_k(y_0)))~\mathrm{d}\mathcal{L}(y_0) \leq C,
\end{align*}
see Theorem~\ref{thm:convofobj}. Moreover the adjoint equation does no longer depend on the control. Repeating the arguments of the last proof yields the subsequential convergence of~$(\mathbf{y}^*_k,\mathbf{p}^*_k,\F^{\eps_k}_{\theta^*_k}(\mathbf{y}^*_k))$ towards an element~$(\bar{\mathbf{y}},\bar{\mathbf{p}},\bar{\mathbf{u}})\in L^\infty(Y_0;W_T)^2 \times L^2(Y_0;L^2(I;\R^m))$ such that~$(\bar{y},\bar{p},\bar{u})\coloneqq(\bar{\mathbf{y}}(y_0),\bar{\mathbf{p}}(y_0),\bar{\mathbf{u}}(y_0))$ satisfy the system of state and adjoint equations as well as~$(\bar{y},\bar{u})\in \argmin \eqref{def:openloopproblem}$ for~$\mathcal{L}$-a.e.~$y_0 \in Y_0$. Then it only remains to argue the additional regularity~$\mathbf{u}\in L^\infty(Y_0;L^2(Y_0;W_T))$. This is, however, a direct consequence of the first order necessary optimality condition~$\bar{\mathbf{u}}=(-1/\beta)B^\top\bar{p}$ for~\eqref{def:openloopproblem}, see Proposition~\ref{prop:structure}.
\end{remark}
We point out that the statement of Theorem~\ref{thm:convoftraj} holds independently of the values of the penalty parameters~$\gamma_1,\gamma_2$. If~$\gamma_1,\gamma_2 >0$ then we additionally obtain the following convergence results for the approximate value function~$\mathcal{V}^\eps_{\theta^*_k}$ and its derivative~$\partial_y \mathcal{V}^\eps_{\theta^*_k}$ along optimal state trajectories.
\begin{prop} \label{prop18}
Let the prerequisites of Theorem~\ref{thm:convofobj} hold and
let~$(\mathbf{y}^*_k, \mathbf{p}^*_k, \theta^*_k)$ denote a sequence of minimizing triplets as described in Theorem~\ref{thm:convoftraj}. Assume that~$(\mathbf{y}^*_{k}, \mathbf{p}^*_{k}, \F^{\eps_k}_{\theta^*_k}(\mathbf{y}^*_k))$ converges to~$(\bar{\mathbf{y}},\bar{\mathbf{p}}, \bar{\mathbf{u}})$ in~$L^2(Y_0;W_T)^2 \times L^2(Y_0; L^2(I;\R^m))$ and~$\gamma_1,\gamma_2 >0$. Then we also have
\begin{align*}
\mathcal{V}^{\eps_k}_{\theta^*_k}(\mathbf{y}^*_k) \rightarrow \mathcal{V}^*(\bar{\mathbf{y}})~\text{in}~L^2(Y_0;L^2(I)),~\partial_y \mathcal{V}^{\eps_k}_{\theta^*_k}(\mathbf{y}^*_k) \rightarrow \bar{\mathbf{p}}~\text{in}~L^2(Y_0;L^2(I;\R^n)).
\end{align*}
\end{prop}
\begin{proof}
Due to the convergence of~$\mathbf{y}^*_k \to \bar{\mathbf{y}}$ in $L^2(Y_0;L^2(I;\R^n))$ and~$\F^{\eps_k}_{\theta^*_k}(\mathbf{y}^*_k) \to \bar{\mathbf{u}}$ in $L^2(Y_0;L^2(I;\R^m))$, we conclude that
\begin{align*}
J_\bullet (\mathbf{y}^*_k,\F^{\eps_k}_{\theta^*_k}(\mathbf{y}^*_k)) \rightarrow J_\bullet (\bar{\mathbf{y}},\bar{\mathbf{u}})= \mathcal{V}^*(\bar{\mathbf{y}})~\text{in}~L^2(Y_0;L^2(I)).
\end{align*}
Together with
\begin{align*}
\lim_{k\rightarrow \infty} \int_{Y_0} \omega(y_0)  \| V^{\eps_k}_{\theta^*_k}(\mathbf{y}^*_k(y_0))-J_\bullet (\mathbf{y}^*_k(y_0), \F^{\eps_k}(\mathbf{y}^*_k(y_0)))\|^2_{L^2}~\mathrm{d}\mathcal{L}(y_0)=0,
\end{align*}
see Theorem~\ref{thm:convofobj}, we arrive at~$\mathcal{V}^{\eps_k}_{\theta^*_k}(\mathbf{y}^*_k) \rightarrow \mathcal{V}^*(\bar{\mathbf{y}})$ in~$L^2(Y_0;L^2(I))$. The statement on the convergence of~$\partial_y \mathcal{V}^{\eps_k}_{\theta^*_k}(\mathbf{y}^*_k)$ follows similarly from the strong convergence of~$\mathbf{p}_k$.
\end{proof}
\section{Learning from a finite training set}
We turn to analysing a discrete  version of \eqref{def:approxfeedprop}.
In this case we can proceed without the state-space constraint
$\mathbf{y} \in \textbf{Y}_{ad}$ provided certain growth bounds on
$\mathbf{f}$ and $\mathbf{g}$ are satisfied. The numerical realization of \eqref{def:approxfeedprop} will always rely on such a discrete
approximation. Henceforth we fix a finite ensemble of initial conditions
$\{y_0^i: i=1,\dots, N\} \subset Y_0$. For positive weights $\omega_i$, $i=1,\dots,N$, and $\eps>0$ we consider
\begin{align} \label{eq:learningprobfinite}
\inf_{y_i, p_i \in W_T, \theta \in \mathcal{R}_\eps} \left
\lbrack\sum^N_{i=1} \omega_i J_\eps(y_i,p_i, \theta)+
\frac{\gamma_\eps}{2} \|\theta\|^2_{\mathcal{R}_\eps} \right \rbrack
\tag{$\mathcal{P}^N_\eps$}
\end{align}
subject to
\begin{align*}
\dot{y}_i
&=\mathbf{f}(y_i)+\mathbf{g}(y_i)\F^\eps_\theta(y_i),~y_i(0)=y^i_0 \\
-\dot{p}_i&= D\mathbf{f}(y_i)^\top p_i+ \lbrack
D\mathbf{g}(y_i)^\top\F^\eps_\theta(y_i)\rbrack p_i+\mathbf{Q}_1^\top
\mathbf{Q}_1(y_i-y_d),~p_i(T)= Q^\top_2 Q_2(y_i(T)-y_d^T).
\end{align*}

Throughout this section, Assumptions~\ref{ass:feedbacklaw}
and~\ref{ass:approxsmoothness} are supposed to hold. Further $\eps$ is
supposed to be sufficiently small so that the set of admissible
solutions for \eqref{eq:learningprobfinite} is nonempty, compare
Theorem~\ref{thm:existenceneuralnetwork2}. It will be convenient to introduce $\mathbf{y}=
\text{col}(y_1,\dots,y_N)$, and  $\mathbf{p}=
\text{col}(p_1,\dots,p_N)$, which replace the ensemble states and
costates from the previous sections.
\begin{prop}\label{prop:existencefinite}
Let~$\eps>0$ be sufficiently small and let $(\mathbf{y}^k, \mathbf{p}^k,
\theta_k)\in W_T^{2N}\times \mathcal{R}_\eps$
denote an infimizing sequence for \eqref{eq:learningprobfinite}.
If ~$\max_{i} \|y^k_i\|_{L^\infty(I;\R^n)} \leq M_\infty$  for
some~${M_\infty}>0$  independent of~$k \in \N$, then
Problem~\eqref{eq:learningprobfinite}   admits at least one
minimizer~$(\mathbf{y}^*, \mathbf{p}^*, \theta^*)$.
\end{prop}
\begin{proof}
Since by assumption $(\mathbf{y}^k, \mathbf{p}^k, \theta_k)$
is an infimizing sequence for \eqref{eq:learningprobfinite} and since ~$\beta >0$ we have
\begin{align} \label{eq:dimdepend}
\max_{i} \|{\bf{Q_1}} y^k_i\|^2_{L^2}+ \max_{i}
\|\F^\eps_{\theta_k}(y^k_i)\|^2_{L^2} \leq C_N
\end{align}
for some~$C_N >0$ depending on~$N$. Moreover there holds
\begin{align*}
\|\dot{y}^k_i\|_{L^2} \leq  \|\mathbf{f}(y^k_i)\|_{L^2}+
\|\mathbf{g}(y_i)\F^\eps_{\theta_k}(y^k_i)\|_{L^2} \leq
C(\mathbf{f},\mathbf{g}) M_\infty (1+C_N )
\end{align*}
using the uniform~$L^\infty$ and~$L^2$ boundedness of~$y^k_i$
and~$\F^\eps_\theta(y^k_i)$, respectively. Thus we also
have~$\wnorm{y^i_k}\leq \widehat{C}_N $ for all~$k \in \N$, for some
~$\widehat{C}_N>0$ which depends on~$N$ but not on~$k$ and~$i$. The
proof  can now be completed by the same steps as
Theorem~\ref{thm:existence}.
\end{proof}
\begin{remark}\label{rem:linftybound}
The~$L^\infty$-boundedness of the minimizing sequence~$y^k_i$ in
Proposition~\ref{prop:existencefinite} can be be ensured by additional
assumptions on the dynamics of the problem. These include:
\begin{itemize}
\item Add an additional state constraint~$\|y_i\|_{L^\infty} \leq
\widehat{M}$ to~\eqref{eq:learningprobfinite}.
\item Assume that there are~$a_1,a_2,a_3>0$ such that
\begin{align*}
|f(x)| \leq a_1+ a_2 |x|+ a_3 |x|^2,~\|g(x)\| \leq a_1+a_2|x| \quad
\forall x \in \R^n,
\end{align*}
and that $Q_1$ is positive definite. Then  by \eqref{eq:dimdepend} the family $\{y^k_i\}$ is
uniformly w.r.t. $i\in\{1,\dots,n\}$ and $k=1,\dots$ bounded in $L^2(I;\R^n)$ .
Further  we can readily verify that
\begin{align*}
\|\dot{y}^k_i\|_{L^1} &\leq \|\mathbf{f}(y^k_i)\|_{L^1}+
\|\mathbf{g}(y^k_i)\|_{L^1}
  \\ & \leq 2a_1 T+ a_2 \|y^k_i\|_{L^1}+ a_3 \|y^k_i\|^2_{L^2}+ a_2
\|y^k_i\|_{L^2} \|\F^\eps_{\theta_k}(y^k_i)\|^2_{L^2} \leq
  M_N
\end{align*}
for an~$N$-dependent bound~$M_N>0$. Here we made use of the
$L^2$-boundedness of~$y^k_i$ and~$\F^\eps_{\theta_k}(y^k_i)$  which follows from  \eqref{eq:dimdepend} in the
proof of Proposition ~\ref{prop:existencefinite}, and the assumption that $Q_1>0$. Consequently~$y^k_i$ is uniformly bounded
in~$W^{1,1}(I; \R^n)$ and thus also in~$L^\infty(I;\R^n)$.
\item Assume that~$f(x)=Ax-h(x)$ where~$A \in \R^{n \times n}$ and~$h$
is monotone i.e.~$(x,h(x))_{\R^n} \geq 0$ for all~$x\in\R^n$. Moreover
assume that $Q_1$ is positive definite and that
\begin{align*}
\|g(x)\| \leq a_1+a_2|x| \quad \forall x \in \R^n.
\end{align*}
In this case, testing the equation satisfied by $y_i$ with $y_i$, and  a
Gronwall argument yields
\begin{align*}
|y^k_i(t)|^2 \leq C_N \left( |y^i_0|^2+ \|y^k_i\|^2_{L^2} +
\|\F^\eps_{\theta_k}(y^k_i)\|^2_{L^2}  \right)
\end{align*}
for some~$N$-dependent~$C_N>0$ and all~$t\in I$. Thus, the uniform
boundedness of~$y^k_i$ in~$L^\infty(I;\R^n)$ follows again from
the~$L^2$-estimates on~$y^k_i$ and~$\F^\eps_{\theta_k}(y^k_i)$
in~\eqref{eq:dimdepend}.
\end{itemize}
\end{remark}
The convergence result as $\eps\to 0^+$   of
Theorem~\ref{thm:convoftraj} can be transferred to the finite training
set setting as well.

\begin{prop} \label{prop:convergencefinite}
Let  the regularisation parameters satisfy $\gamma_\eps
\|\theta_\eps\|^2_{\mathcal{R}_\eps} = O(\eps)$.
Further let ~$\eps_k>0$ be a positive null sequence such  that
for each~$k \in \N$  there exists a
solution~$(\mathbf{y}^k,\mathbf{p}^k,\theta_k) \in W_T^{2N} \times
\mathcal{R}_\eps$ to~$(\mathcal{P}^N_{\eps_k})$.
If there is~${M_\infty}>0$ with~$\max_{i} \|y^k_i\|_{L^\infty} \leq
M_\infty$ for all~$k \in \N$,
then~$(\mathbf{y}^k,\mathbf{p}^k,
\mathbf{\F}^{\eps_k}_{\theta_k}(\mathbf{y}^k))$ admits at least one
strong accumulation point~$(\mathbf{\bar{y}},  \mathbf{\bar{p}},
\mathbf{\bar{u}})$ in~$W^{2N}_T \times L^2(I;\R^m)^N$. Each such point
satisfies
\begin{align*}
(\bar{y}_i, \bar{u}_i)\in \argmin (P^{y^i_0}_\beta), \quad i=1, \dots,N,
\end{align*}
as well as
\begin{align*}
\dot{\bar{y}}_i
&=\mathbf{f}(\bar{y}_i)+\mathbf{g}(\bar{y}_i)\bar{u}_i,~\bar{y}_i(0)=y^i_0
\\
-\dot{\bar{p}}_i&= D\mathbf{f}(\bar{y}_i)^\top \bar{p}_i+ \lbrack
D\mathbf{g}(\bar{y}_i)^\top\bar{u}_i\rbrack \bar{p}_i+\mathbf{Q}_1^\top
\mathbf{Q}_1(\bar{y}_i-y_d),~\bar{p}_i(T)= Q^\top_2
Q_2(\bar{y}_i(T)-y_d^T).
\end{align*}
\end{prop}
\begin{proof}
For every~$\eps_k$, with $k$ sufficiently large,   denote
by~$\theta_{\eps_k}\in \mathcal{R}_{\eps_k}$ the corresponding
parameters from Assumption~\ref{ass:approxsmoothness},
by~$\mathbf{y}_{\eps_{k}}$ the associated ensemble solution, see
Theorem~\ref{thm:existenceneuralnetwork2}, and
by~$\mathbf{p}_{\eps_{k}}$ the adjoint states. For abbreviation we
set~$y^{\eps_k}_i \coloneqq \mathbf{y}_{\eps_k}(y^i_0)$ and
$p^{\eps_k}_i \coloneqq \mathbf{p}_{\eps_k}(y^i_0)$. Then, by
optimality, we have
\begin{align}\label{eq:aux11}
\sum^N_{i=1} \omega_i J(y^k_i, \F^{\eps_k}_{\theta_k}(y^k_i)) \leq
\sum^N_{i=1} \omega_i J_\eps(y^{\eps_k}_i, p^{\eps_k}_i,
\theta_{\eps_k}) + \frac{\gamma_{\eps_k}}{2} \|\theta_{\eps_k}\|^2_{\mathcal{R}_\eps}.
\end{align}
As in the proof of Theorem~\ref{thm:convofobj} we see that the
righthandside of this inequality converges to~$\sum^N_{i=1} \omega_i
V^*(0,y^i_0)$ as~$k \rightarrow +\infty$. Thus it is bounded
independently of~$k \in \N$. Similarly to
Proposition~\ref{prop:existencefinite} we then conclude the existence
of~$C_N >0$ depending on~$N$, but not on~$k$, such that
\begin{align*}
\max_{i} \|{\mathbf Q_1}y^k_i\|^2_{L^2}+ \max_{i}
\|\F^\eps_{\theta_k}(y^k_i)\|^2_{L^2} \leq C_N.
\end{align*}
Utilizing the state equation this can be improved to a~$k$-independent
bound on the~$W_T$-norm of~$y^k_i$. By a Gronwall-type argument the same
can be shown for the adjoint states~$p^k_i$. Now fix an arbitrary
index~$i \in \{1,\dots,N\}$. Summarizing the previous observations we
get the uniform  boundedness of~$(y^k_i, p^k_i, \F^\eps_{\theta_k}(y^k_i)
)$ in~$W_T^2 \times L^2(I; \R^m) $  w.r.t. $k$, for each $i=1,\dots,N$. Each of its
weak accumulation points~$(\bar{y}_i, \bar{p}_i, \bar{u}_i) \in W_T^2
\times L^2(I; \R^m)$ satisfies
\begin{align*}
\dot{\bar{y}}&=\mathbf{f}(\bar{y})+\mathbf{g}(\bar{y})\bar{u},~
\bar{y}(0)=y_0.
\end{align*}
 From this we conclude that
\begin{align*}
0\le \sum^N_{i=1} \omega_i V^*(0,y^i_0) \leq \sum^N_{i=1} \omega_i
J(\bar y_i,\bar u_i) \le   \lim_{k\to \infty} \sum^N_{i=1} \omega_i
J(y^k_i, \F^{\eps_k}_{\theta_k}(y^k_i))\le \sum^N_{i=1} \omega_i
V^*(0,y^i_0),
\end{align*}
Since the second and third of the above inequalities also hold for each
summand we conclude that~$\lim_{k\to \infty}J(y^k_i,
\F^{\eps_k}_{\theta_k}(y^k_i))\rightarrow J(\bar{y}_i, \bar{u}_i)$ as
well as~$J(\bar{y}_i, \bar{u}_i)=V^*(0,y^i_0)$. Hence
\begin{align*}
(\bar{y}_i, \bar{u}_i)\in \argmin (P^{y^i_0}_\beta).
\end{align*}
The proof can now be concluded with minor adaptations to the proof of
Theorem~\ref{thm:convoftraj}.
\end{proof}

A result analogous to that of Proposition \ref{prop18} can also be
obtained for Problem \eqref{eq:learningprobfinite}. For the sake of
brevity we do not present the details.
\subsection{The reduced objective functional}
In order to compute a solution to~\eqref{eq:learningprobfinite} we will rely on gradient-based optimization methods. For this purpose we introduce a~\textit{reduced objective functional} by  eliminating the state and adjoint equations in~\eqref{eq:learningprobfinite}. Subsequently, we characterize the derivative of the reduced functional by means of adjoint techniques. To simplify the presentation we fix an arbitrary index~$i\in\{1,\dots,N\}$ in the following. Moreover, for abbreviation, we define the mapping
\begin{align*}
\mathbf{A} \colon W_T \times \mathcal{R}_\eps \to \mathcal{B}(W_T;L^2(I;\R^n)),~A(y,\theta)=D\mathbf{f}(y)^\top + \lbrack  D\mathbf{g}(y)^\top\F^\eps_\theta(y)\rbrack.
\end{align*}
Using this notation, the adjoint equation in~\eqref{eq:learningprobfinite} can be expressed compactly as
\begin{align*}
-\dot{p}_i= \mathbf{A}(y_i,\theta)p_i+\mathbf{Q}_1^\top \mathbf{Q}_1(y_i-y_d),~p_i(T)=Q^\top_2 Q_2(y_i(T)-y_d^T).
\end{align*}
First, we argue the existence of~\textit{parameter-to-state operators} for the adjoint and the state equation.
\begin{lemma} \label{lem:existenceparametertostate}
Define~$G_i \colon W_T \times W_T \times \mathcal{R}_\eps \to L^2(I; \R^n) \times L^2(I; \R^n) \times \R^n \times \R^n$ by
\begin{align*}
G_i(y,p, \theta)= \left(
\begin{array}{c}
\dot{y} -\mathbf{f}(y)-\mathbf{g}(y)\F^\eps_\theta(y)  \\
-\dot{p}- A(y,\theta)p-\mathbf{Q}_1^\top \mathbf{Q}_1(y-y_d)\\
y(0)-y^i_0 \\
p(T)- Q^\top_2 Q_2(y(T)-y_d^T)
\end{array}
\right).
\end{align*}
Let~$(\tilde{y},\tilde{p},\tilde{\theta})\in W_T \times W_T \times \mathcal{R}_\eps$ satisfy~$G(\tilde{y},\tilde{p},\tilde{\theta})=0$. Then there exists a neighbourhood~$\mathcal{N}_i(\tilde{y})\times \mathcal{N}_i(\tilde{p}) \times \mathcal{N}_i(\tilde{\theta})$ as well as~$\mathcal{C}^1$-mappings~$Y_i \colon \mathcal{N}_i(\tilde{\theta})\to \mathcal{N}_i(\tilde{y})\subset W_T$,~$P_i \colon \mathcal{N}_i(\tilde{\theta})\to \mathcal{N}_i(\tilde{p}) \subset W_T$ such that
\begin{align*}
G_i(Y_i(\theta),P_i(\theta),\theta)=0 \quad \forall \theta \in \mathcal{N}(\tilde{\theta}).
\end{align*}
Given $y_i:=Y_i(\theta)$ and $~p_i:=P_i(\theta)$, the Fr\'{e}chet derivatives of~$Y_i$ and~$P_i$ at~$\theta \in \mathcal{N}_i(\tilde{\theta})$, in  direction ~$\delta \theta \in \mathcal{R}_\eps$,  denoted  by ~$\delta Y_i \coloneqq Y'_i(\theta)(\delta \theta) $,~$\delta P_i \coloneqq P'_i(\theta)(\delta \theta)$ satisfy
\begin{align*}
&\dot{\delta Y_i}- \mathbf{A}(y_i,\theta)^\top \delta Y_i- \mathbf{g}(y_i) D_y \F^\eps_\theta(y_i)\delta Y= \mathbf{g}(y_i) D_\theta \F^\eps_\theta(y_i)\delta \theta, \\
&-\dot{\delta P_i}-\mathbf{A}(y_i,\theta)\delta P_i =\lbrack D_y\mathbf{A}(y_i,\theta)\delta Y_i \rbrack p_i+ \mathbf{Q}_1 \mathbf{Q}_1 \delta Y_i+\lbrack \partial_\theta \mathbf{A}(y_i,\theta) \delta \theta \rbrack p_i,\\
&~\delta Y_i(0)=0, ~\delta P_i(T)=  Q_2^\top Q_2 \delta Y_i(T).
\end{align*}

\end{lemma}
\begin{proof}
This is a direct consequence of the implicit function theorem applied to~$G$ noting that the directional derivatives satisfy
\begin{align*}
\begin{pmatrix}
\partial_y G_i(y,p,\theta) & \partial_p G_i(y,p,\theta)
\end{pmatrix}
\begin{pmatrix}
\delta Y \\ \delta P
\end{pmatrix}
=- \partial_\theta G_i(y,p,\theta) \delta \theta.
\end{align*}
\end{proof}
Now consider an admissible point~$(\tilde{\mathbf{y}}, \tilde{\mathbf{p}}, \theta) \in W^{2N}_T \times \mathcal{R}_\eps$ for~\eqref{eq:learningprobfinite}. For every~$i=1,\dots,N$, let~$\mathcal{N}_i(\tilde{\theta})$ and~$Y_i,P_i$ denote the corresponding neighbourhoods and operators from Lemma~\ref{lem:existenceparametertostate}. Setting~$\mathcal{N}(\tilde{\theta})= \bigcap^N_{i=1} \mathcal{N}_i(\tilde{\theta})$ define the reduced objective functional
\begin{align} \label{eq:reducedobj}
\mathcal{J}_N \colon \mathcal{N}(\tilde{\theta}) \to [0,+\infty),~\mathcal{J}_N(\theta)= \sum^N_{i=i} \omega_i J_\eps(Y_i(\theta),P_i(\theta), \theta)+ \frac{\gamma_\eps}{2}\|\theta\|^2_{\mathcal{R}_\eps},
\end{align}
and set
\begin{align*}
\Phi_i(t)= \int^t_0 (V^\eps_\theta(s,y_i(s))-J_s(y_i,u))~\mathrm{d}s.
\end{align*}.
\begin{prop} \label{prop:direcderiv}
The functional~$\mathcal{J}_N$ from~\eqref{eq:reducedobj} is at least of class~$\mathcal{C}^1$ on~$\mathcal{N}(\tilde{\theta})$. Given~$ \theta \in \mathcal{N}(\tilde{\theta})$, set~$y_i \coloneqq Y_i(\theta)$,~$p_i \coloneqq P_i(\theta)$ as well as~$\delta Y_i \coloneqq Y'_i(\theta)(\delta \theta)$,~$ \delta P_i \coloneqq P'_i(\theta)(\delta \theta)$. The directional derivative of~$\mathcal{J}_N$ at~$\theta$ in the direction of~$\delta \theta \in \mathcal{R}_\eps$ is given by
\begin{align*}
\mathcal{J}'_N(\theta)( \delta \theta)= \sum^N_{i=1} \omega_i \left( (\widehat{y}_i, \delta Y_i)_{L^2}+(\widehat{y}^T_i, \delta Y_i(T))_{\R^n}+(\widehat{p}_i, \delta P_i)_{L^2}+(\widehat{\theta}_i,\delta \theta)_{\mathcal{R}_\eps}\right)+\gamma_\eps (\theta,\delta \theta)_{\mathcal{R}_\eps}
\end{align*}
with
\begin{align*}
\widehat{y}_i=(1-\gamma_1\Phi_i)& \mathbf{Q}_1 \mathbf{Q}_1 (y_i-y_d)+ \beta(1-\gamma_1\Phi_i) D_y \F^\eps_\theta(y_i)^\top\F^\eps_\theta(y_i)\\&+\gamma_1 (\mathcal{V}^\eps_\theta(t,y_i)-J_\bullet(y_i,\F^\eps_\theta(y_i)))\partial_y \mathcal{V}^\eps_\theta(y_i)+\gamma_2 D_{yy}\mathcal{V}^\eps_\theta(y_i)(\partial_{y}\mathcal{V}^\eps_\theta(y_i)-p_i),
\end{align*}
and
\begin{align*}
\widehat{y}^T_i=\alpha(1-\gamma_1\Phi_i(0)) Q_2 Q_2 (y_i(T)-y^T_d),
\end{align*}
as well as
\begin{align*}
\widehat{p}_i=\gamma_2(p_i-\partial_y \mathcal{V}^\eps_\theta(y_i)),
\end{align*}
and
\begin{align*}
\widehat{\theta}_i&= \gamma_1 \int^T_0 D_\theta V^\eps_\theta(t,y_i(t))^\top(V^\eps_\theta(t,y_i(t))-J_t(y_i,\F^\eps_\theta(y_i)))~\mathrm{d} t \\
&+\int^T_0 \!\lbrack \beta (1\!-\!\gamma_1\Phi_i(t)) D_\theta F^\eps_\theta(t,y_i(t))^\top F^\eps_\theta(t,y_i(t))\!+\!\gamma_2 D_{y \theta} V^\eps_\theta(t,y_i(t))^\top(\partial_{y}{V}^\eps_\theta(t,y_i(t))\!-\!p_i(t)) \rbrack~\mathrm{d} t.
\end{align*}
\end{prop}
\begin{proof}
The regularity of~$\mathcal{J}_N$ follows immediately from Lemma~\ref{lem:existenceparametertostate} and the chain rule. In order to compute the directional derivative we abbreviate
\begin{equation*}
\begin{array}l
F_1(y,u,\theta)=\frac{\gamma_1}{2}\int^T_0 |V^\eps_\theta(t,y(t))-J_t(y,u)|^2~\mathrm{d} t,\\[1.4ex]
F_2(y,p,\theta)=\frac{\gamma_2}{2}\int^T_0 |\partial_y V^\eps_\theta(t,y(t))-p(t)|^2~\mathrm{d} t
\end{array}
\end{equation*}
in the following.
Thus we have
\begin{align*}
J_\eps(Y_i(\theta),P_i(\theta),\theta)&= J(Y_i(\theta),\F^\eps_\theta(Y_i(\theta)))+F_1(Y_i(\theta),\F^\eps_\theta(Y_i(\theta)),\theta)+F_2(Y_i(\theta),P_i(\theta),\theta)\\&=G_1(\theta)+G_2(\theta)+G_3(\theta).
\end{align*}
We readily verify
\begin{align*}
G'_1(\theta)(\delta \theta)=&(\mathbf{Q}_1 \mathbf{Q}_1 (y_i-y_d),\delta Y_i )_{L^2}+ \beta (D_y \F^\eps_\theta(y_i)^\top\F^\eps_\theta(y_i),\delta Y_i)_{L^2}\\&+ \beta (D_\theta \F^\eps_\theta(y_i)^\top\F^\eps_\theta(y_i),\delta \theta)_{\mathcal{R}_\eps}+ (Q_2 Q_2 (y_i(T)-y^T_d),\delta Y_i(T))_{\R^n}.
\end{align*}
Recalling the definition of $\Phi_i$ we get
\begin{align*}
G'_2(\theta)& (\delta \theta)=\gamma_1 (E_1+E_2+E_3+E_4),
\end{align*}
where
\begin{align*}
E_1 &=  ((\mathcal{V}^\eps_\theta(y_i)\!-\!J_\bullet(y_i,\F^\eps_\theta(y_i))) \partial_y \mathcal{V}^\eps_\theta(y_i), \delta Y_i)_{L^2} +(D_\theta \mathcal{V}^\eps_\theta(y_i)^\top(\mathcal{V}^\eps_\theta(y_i)\!-\!J_\bullet(y_i,\F^\eps_\theta(y_i))), \delta \theta)_{\mathcal{R}_\eps} \\
&= ((\mathcal{V}^\eps_\theta(y_i)\!-\!J_\bullet(y_i,\F^\eps_\theta(y_i))) \partial_y \mathcal{V}^\eps_\theta(t,y_i), \delta Y_i)_{L^2} \\
&\qquad +\left( \int^T_0 D_\theta V^\eps_\theta(t,y_i(t))^\top(V^\eps_\theta(t,y_i(t))-J_t(y_i,\F^\eps_\theta(y_i)))~\mathrm{d} t, \delta \theta \right)_{\mathcal{R}_\eps},
\end{align*}
\begin{align*}
E_2 =- \int^T_0 &(V^\eps_\theta(t,y(t))-J_t(y,u))\\
 &\left(\int^T_t (Q_1 Q_1(y(s)-y_d(s)),\delta y(s))~\mathrm{d} s+(Q_2 Q_2(y(T)-y^T_d),\delta y(T))_{\R^n} \right) ~\mathrm{d} t \\
&=-(\Phi_i Q_1 Q_1(y-y_d),\delta y)_{L^2}- \Phi_i(0) (Q_2 Q_2(y(T)-y^T_d),\delta y(T))_{\R^n},
\end{align*}
as well as
\begin{align*}
E_3&=- \int^T_0 (V^\eps_\theta(t,y(t))-J_t(y,u)) \left( \beta \int^T_t (D_y F^\eps_\theta(s,y_i(s))^\top F^\eps_\theta(s,y_i(s)),\delta Y_i(s))_{\R^n}~\mathrm{d} s\right) ~\mathrm{d} t \\ &=- \beta(\Phi_i D_y \F^\eps_\theta(y_i)^\top \F^\eps_\theta(y_i), \delta Y_i )_{L^2},
\end{align*}
and
\begin{align*}
E_4&=- \int^T_0 (V^\eps_\theta(t,y(t))-J_t(y,u)) \left( \beta \int^T_t (D_\theta F^\eps_\theta(s,y_i(s))^\top F^\eps_\theta(s,y_i(s)),\delta \theta)_{\mathcal{R}_\eps}~\mathrm{d} s\right) ~\mathrm{d} t \\ &=- \beta \left(\int^T_0 \Phi_i(t) D_\theta F^\eps_\theta(t,y_i(t))^\top F^\eps_\theta(t,y_i(t))~\mathrm{d} t, \delta \theta \right)_{\mathcal{R}_\eps},
\end{align*}
by means of partial integration. Finally we calculate
\begin{align*}
G'_3(\theta)(\delta \theta)= \gamma_2 (D_{yy}&\mathcal{V}^\eps_\theta(y_i)(\partial_{y}\mathcal{V}^\eps_\theta(y_i)-p_i), \delta Y_i)_{L^2}- \gamma_2(\partial_{y}\mathcal{V}^\eps_\theta(y_i)-p_i, \delta P_i)_{L^2} \\&
+\gamma_2\left(\int^T_0 D_{y \theta} V^\eps_\theta(t,y_i(t))^\top(\partial_{y}{V}^\eps_\theta(t,y_i(t))-p_i(t))~\mathrm{d} t, \delta \theta \right)_{\mathcal{R}_\eps}.
\end{align*}
Summarizing the previous observations, we arrive at the claimed characterization.
\end{proof}
Applying a gradient method to~\eqref{eq:learningprobfinite} requires the computation of the gradient~$\nabla \mathcal{J}_{N}(\theta)\in \mathcal{R}_\eps$ which satisfies
\begin{align*}
\mathcal{J}'_N(\theta)( \delta \theta)=(\nabla \mathcal{J}_N(\theta), \delta \theta)_{\mathcal{R}_\eps} \quad \forall \delta \theta \in \mathcal{R}_\eps.
\end{align*}
This can be done by computing~$\mathcal{J}'_N(\theta)( e_j)$ for the canonical basis~$\{e_j\}^{N_\eps}_{j=1}\subset \mathcal{R}_\eps$. However, such reasoning leads to the necessity to solve ~$2\operatorname{dim}(\mathcal{R}_\eps) N$ additional ODEs in order to compute the sensitivities~$Y'_i(\theta)(e_j)$ and~$P'_i(\theta)(e_j)$, respectively. Introducing suitable costate equations, this can be reduced to~$2N$ additional equation solves.
\begin{lemma} \label{lem:adjoint}
Let~$\widehat{y}_i,\widehat{y}^T_i, \widehat{p}_i$ as well as~$\delta Y_i, \delta P_i$ be defined as in Proposition~\ref{prop:direcderiv}. Then there holds
\begin{align*}
(\widehat{y}_i,\delta Y_i )_{L^2}+ (\widehat{y}^T_i,\delta Y_i (T) )_{\R^n}+(\widehat{p}_i,\delta P)_{L^2}=(D_\theta \F^\eps_\theta(y_i)^\top(\mathbf{g}(y_i)^\top \zeta_i+\lbrack  D\mathbf{g}(y_i) \kappa_i \rbrack^\top p_i), \delta \theta)_{\mathcal{R}_\eps}
\end{align*}
where~$\zeta_i, \kappa_i \in W_T$ satisfy
\begin{align*}
-\dot{\zeta}_i&=\mathbf{A}(y_i,\theta)\zeta+ D_y \F^\eps_\theta(y_i)^\top \mathbf{g}(y_i)^\top\zeta_i + \lbrack D_y \mathbf{A}(y_i,\theta)^\top p\rbrack\kappa_i+\mathbf{Q}_1^\top \mathbf{Q}_1\kappa_i+\widehat{y}_i \\
\dot{\kappa}_i&=\mathbf{A}(y_i,\theta)^\top \kappa_i+\widehat{p}_i,\\
\zeta_i(T)&= Q_2^\top Q_2 \kappa (T)+\widehat{y}^T_i,~\kappa_i(0)=0.
\end{align*}
\end{lemma}
\begin{proof}
For the sake of readability, we drop the subscript~$i$ in the following. By partial integration and Lemma \ref{lem:existenceparametertostate} we obtain
\begin{align*}
(\widehat{p},\delta P)_{L^2}&=(\dot{\kappa}-\mathbf{A}(y,\theta)^\top \kappa,\delta P)=(-\dot{\delta P}-\mathbf{A}(y,\theta) \delta P,\kappa)+( Q_2^\top Q_2 \delta Y(T),\kappa(T))_{\R^n}\\ &=(\lbrack D_y\mathbf{A}(y,\theta)\delta Y \rbrack p+ \mathbf{Q}_1^\top \mathbf{Q}_1 \delta Y+\lbrack \partial_\theta \mathbf{A}(y,\theta) \delta \theta \rbrack p,\kappa) +( \delta Y(T),\zeta(T)- \hat y_i^T)_{\R^n}
\end{align*}
and
\begin{align*}
&(\hat y_i,\delta Y )_{L^2}+ (\hat y_i^T,\delta Y (T) )_{\R^n}\\&=(-\dot{\zeta}\!-\!\mathbf{A}(y,\theta)\zeta\!-\! D_y \F^\eps_\theta(y)^\top \mathbf{g}(y)^\top\zeta- \lbrack D_y \mathbf{A}(y,\theta)^\top p\rbrack\kappa-\mathbf{Q}_1^\top \mathbf{Q}_1\kappa, \delta Y)_{L^2}+ (\hat y^T_i,\delta Y (T) )_{\R^n}\\
&=(\dot{\delta Y}-\mathbf{A}(y,\theta)^\top\delta Y- \mathbf{g}(y) D_y \F^\eps_\theta(y) \delta Y ,\zeta)_{L^2}\\
&\quad -( \lbrack D_y \mathbf{A}(y,\theta)^\top p\rbrack\kappa+\mathbf{Q}_1^\top \mathbf{Q}_1\kappa, \delta Y)_{L^2} -  (\delta Y (T), \zeta(T) )_{\R^n} \\
&=(\mathbf{g}(y) D_\theta \F^\eps_\theta(y)\delta \theta,\zeta)_{L^2} -( \lbrack D_y \mathbf{A}(y,\theta)^\top p\rbrack\kappa+\mathbf{Q}_1^\top \mathbf{Q}_1\kappa, \delta Y)_{L^2} -  (\delta Y (T), \zeta(T)- \hat y_i^T )_{\R^n}.
\end{align*}
Adding both equations finally yields
\begin{align*}
(\hat y_i,\delta Y )_{L^2}+ (\hat y_i^T,\delta Y (T) )_{\R^n}+(p_1,\delta P)_{L^2}&=(\mathbf{g}(y) D_\theta \F^\eps_\theta(y)\delta \theta,\zeta)_{L^2} +(\lbrack \partial_\theta \mathbf{A}(y,\theta) \delta \theta \rbrack p,\kappa)_{L^2} \\
&=(D_\theta \F^\eps_\theta(y)^\top(\mathbf{g}(y)^\top \zeta+\lbrack  D\mathbf{g}(y) \kappa \rbrack^\top p), \delta \theta)_{\mathcal{R}_\eps}
\end{align*}
which ends the proof.
\end{proof}
We arrive at the following characterization of the gradient~$\nabla \mathcal{J}_N(\theta)$.
\begin{theorem} \label{thm:gradient}
Let~$y_i,p_i,\zeta_i, \kappa_i \in W_T,~\widehat{\theta}_i \in \mathcal{R}_\eps$ be defined as in~Proposition~\ref{prop:direcderiv} and Lemma~\ref{lem:adjoint}. The gradient of~$\mathcal{J}_N$ at~$\theta$ is given by
\begin{align*}
\nabla \mathcal{J}_N(\theta)= \sum^N_{i=1} \omega_i \left( D_\theta \F^\eps_\theta(y_i)^\top(\mathbf{g}(y_i)^\top \zeta_i+\lbrack  D\mathbf{g}(y_i) \kappa_i \rbrack^\top p_i) +\widehat{\theta}_i\right)+ \gamma_\eps \theta.
\end{align*}
\end{theorem}
\section{Numerical example}\label{sec:numericalexampl}
We finish this paper by applying the proposed learning approach to one particular instance of Problem~\eqref{def:openloopproblem}. Setting~$I=(0,T)$ and~$\Omega=(0,2\pi)$, we consider the parabolic bilinear optimal control problem
\begin{align*}
\min_{\mathcal{Y} \in L^2(I \times \Omega), u \in L^2(I;\R^3)} \left \lbrack \frac{1}{2} \int_I \|\mathcal{Y}(t)-\mathcal{Y}_d(t)\|^2_{L^2(\Omega)}+ \frac{\beta}{2}|u(t)|^2_{\R^3}~\mathrm{d}t \right \rbrack+\frac{\alpha}{2} \|\mathcal{Y}(T)-\mathcal{Y}_d(T)\|^2_{L^2(\Omega)}
\end{align*}
subject to
\begin{align} \label{eq:bilinearPDE}
\partial_t \mathcal{Y}-\bigtriangleup \mathcal{Y}+ \left( u_1 \chi_1 + u_2 \chi_2+ u_3 \chi_3 \right) \mathcal{Y}= 0,
\end{align}
as well as
\begin{align*}
\mathcal{Y}(t,x)=0 \quad \text{on}~I \times \partial \Omega,~\mathcal{Y}(0,x)= \mathcal{Y}_0(x) \quad \text{on}~\Omega.
\end{align*}
Here~$\alpha>0, \beta>0$, and~$\mathcal{Y}_d$ denotes a given desired state. The dynamics of this infinite-dimensional system can be influenced by choosing a time-dependent three-dimensional control input~$u \in L^2(I; \R^3) $ which acts on the subdomains~$\Omega_1=(0.5,1)$,~$\Omega_2=(2,2.5)$ and $\Omega_3=(4,4.5)$, respectively. The associated characteristic functions are denoted by~$\chi_i$,~$i=1,\dots,3$.

In order to fit this problem into the setting of the current manuscript, let~$\{\lambda_i, \varphi_i\} \in \R_+ \times L^2(\Omega)$ denote the first $n \in \N$ normalized eigenpairs of the Dirichlet Laplacian on~$\Omega$. Approximating the state dynamics~$\mathcal{Y}$ as well as the desired state by
\begin{align*}
\mathcal{Y}(t,x) \approx \sum^n_{i=1} Y_i(t) \varphi_i (t),~\mathcal{Y}_d(t,x) \approx \sum^n_{i=1} Y^i_d(t) \varphi_i (t),~
\end{align*}
we end up with
\begin{align} \label{def:discreteexample}
\min_{{Y} \in L^2(I;\R^10), u \in L^2(I;\R^3)} \left \lbrack \frac{1}{2} \int_I |Y(t)-Y_d(t)|^2+ \frac{\beta}{2}|u(t)|^2_{\R^3}~\mathrm{d}t+ \frac{\alpha}{2}|Y(T)-Y_d(T)|^2_{\R^10} \right \rbrack
\end{align}
subject to
\begin{align*}
\dot{Y}(t)+ AY(t)+ \sum^3_{i=1} u_i M_i Y(t) =0,~Y(0)=Y_0.
\end{align*}
where~$(Y_0)_i= (\mathcal{Y}_0, \varphi_i)_{L^2}$,~$i=1,\dots,n$, and the symmetric matrices~$A,M_i \in \R^{n\times n}$ are given by
\begin{align*}
A_{jk}= \begin{cases}  0 & j \neq k \\
\lambda_j & \text{else}
 \end{cases},~ (M_i)_{jk}=\int_\Omega \phi_j \phi_k \chi_i(x)~\mathrm{d}x, \quad i=1,2,3,~j,k= 1\dots, n.
\end{align*}
\subsection{Learning \& validation setup}
In the following, we determine an approximate optimal feedback law for~\eqref{def:discreteexample} by applying the learning approach detailed in Section~\ref{sec:learnfeedback}. The parametrized model~$V^\eps_\theta$ for the value function is given by realizations of residual networks, as  described in Section~\ref{subsec:residual}, with~$L_\eps=2$ layers,~$\operatorname{arch}(\theta)=(11,60,1)$ and activation function~$\sigma$ given by
\begin{align*}
\sigma(x)= \sin(x)+\cos(x).
\end{align*}
This yields a total of~$1440$ trainable parameters. We emphasize that the architecture as well as the activation function were chosen based on numerical testing. In particular, the present tests should not be mistaken as a \emph{quantitative} survey but as  a \emph{proof of concept} which highlights the potential of learned feedbacks for optimal control and puts a focus on the role played by the penalty parameters~$\gamma_1$ and~$\gamma_2$.

Given a fixed reference vector~$\bar{Y}_0 $, we randomly generate a set~$\mathbf{Y}_0$ of~$130$ initial conditions by sampling uniformly from the closure of~$B_1(\bar{Y}_0)$,
Subsequently, these are split into a training set~$\mathbf{Y}^t_0$ of $N=30$ initial conditions, which is used in the learning problem~\eqref{eq:learningprobfinite} together with uniform weights~$w_j=1/N$, and a validation set~$\mathbf{Y}^v_0= \mathbf{Y}_0 \setminus \mathbf{Y}^T_0$ which we later utilize to assess the performance of the obtained feedback.

In order to obtain a candidate for the optimal network parameters~$\theta^*_\eps$, a Barzilai-Borwein method
~\cite{BaBo19}, is applied to the learning problems \eqref{eq:learningprobfinite}, based on the reduced objective functional introduced in~\eqref{eq:reducedobj} as well as the characterization of its gradient in Theorem~\ref{thm:gradient}. For every~$Y_0 \in \mathbf{Y}^t_0$, this approach entails the computation of the state~$Y \coloneqq Y_\theta (Y_0)$ and the adjoint state~$P \coloneqq P_\theta (Y_0)$ which satisfy
\begin{align} \label{def:exampleneuralequations}
\dot{Y} (t)+ \left(A+ \sum^3_{i=1} {F}^\eps_\theta(t,Y(t))_i  M_i \right) Y(t) &=0,~Y(0)=Y_0 \notag \\
-\dot{P} (t) + \left( A+ \sum^3_{i=1} {F}^\eps_\theta(t,Y(t))_i  M_i  \right) P(t) &= Y(t)-Y_D(t),~P(T)= Y(T)-Y_D(T)
\end{align}
as well as the costates~$K \coloneqq K_\theta(Y_0)$ and~$Z \coloneqq Z_\theta(Y_0)$ with
\begin{align*}
\dot{K}(t)+ \left( A+ \sum^3_{i=1} {F}^\eps_\theta(t,Y(t))_i  M_i  \right) K(t)= \widehat{P}(t)
\end{align*}
and
\begin{multline*}
- \dot{Z} (t)+ \left( A+ \sum^3_{i=1} {F}^\eps_\theta(t,Y(t))_i  M_i  + D_y {F}^\eps_\theta(t,Y(t))^\top \begin{pmatrix}
Y_j(t)^\top M_1 \\
Y_j(t)^\top M_2 \\
Y_j(t)^\top M_3
\end{pmatrix}  \right) Z(t) \\
=- D_y {F}^\eps_\theta(t,Y(t))^\top \begin{pmatrix}
Y(t)^\top M_1 \\
Y(t)^\top M_2 \\
Y(t)^\top M_3
\end{pmatrix}  Z (t) +K (t)+ \widehat{Y} (t)
\end{multline*}
equipped with the boundary conditions
\begin{align*}
K(0)=0, \quad Z(T)= \alpha K(T)+ \widehat{Y}^T_j
\end{align*}
where~$\widehat{Y},\widehat{Y}^T$ and~$\widehat{P}$ are defined in analogy to Proposition~\ref{prop:direcderiv}. Note that this system is not fully coupled, i.e. in practice, we first solve the nonlinear closed-loop equation using a Radau time-stepping scheme and then, successively treat the adjoint and costate equations by an implicit Euler method. This can be done in parallel for various initial conditions to achieve additional speed-up. Moreover, the adjoint state~$P$ and costate~$K$ only need to be computed if~$\gamma_2 >0$. The gradient of the reduced objective functional~$\mathcal{J}_N$ in~\eqref{eq:learningprobfinite} at an admissible~$\theta$ is then obtained as
\begin{align*}
 \frac{1}{30} \sum_{Y \in \mathbf{Y}^t_0} \int_I  \left( D_\theta F^\eps_\theta(t,Y_\theta (Y_0)(t))^\top \left( B^\theta_Y (t) Z_\theta (Y_0) (t)+ B^\theta_K(t) P_\theta (Y_0) (t) \right)~\mathrm{d}t +\widehat{\theta}(Y_0)\right).
\end{align*}
where we set
\begin{align*}
B^\theta_Y (t) \coloneqq
\begin{pmatrix}
Y_\theta (Y_0)(t)^\top M_1 \\
Y_\theta (Y_0)(t)^\top M_2 \\
Y_\theta (Y_0)(t)^\top M_3
\end{pmatrix}
, \quad
B^\theta_K(t) \coloneqq
\begin{pmatrix}
K_\theta (Y_0)(t)^\top M_1 \\
K_\theta (Y_0)(t)^\top M_2 \\
K_\theta (Y_0)(t)^\top M_3
\end{pmatrix},
\end{align*}
integration has to be understood componentwise and~$\widehat{\theta}(Y_0)$ is as in Proposition~\ref{prop:direcderiv}.

Once the network is determined, we compute the state~$Y_{\theta} (Y_0)$ and adjoint~$P_\theta (Y_0)$ for every~$Y_0 \in \mathbf{Y}_0$ from~\eqref{def:exampleneuralequations} and set~$U_\theta(Y_0)\coloneqq\mathcal{F}^\eps_{\bar{\theta}}(Y_\theta(Y_0))$.
Subsequently we determine a stationary point~$(\bar{Y}(Y_0), \bar{U}(Y_0))$ of~\eqref{def:discreteexample},~$Y_0 \in \mathbf{Y}_0$, by applying a Barzilai-Borwein gradient method to its control-reduced formulation. The associated adjoint state is denoted by $\bar{P}(Y_0)$. At this point, it should be stressed that both, the open loop as well as the feedback learning problem, are nonconvex. As a consequence, we cannot ensure global optimality of the computed stationary points and, in particular, both methods might provide different results. For the present example, open loop and learned feedback controls are comparable. Moreover, for every~$Y_0 \in \mathbf{Y}_0$, we have~$J(\bar Y(Y_0), \bar U(Y_0) ) \geq J( Y_\theta (Y_0),  U_\theta (Y_0) )$. In order to assess the performance of open loop and feedback controls, let~$Y^{ad}_{0}\subset \mathbf{Y}_0$ be either~$Y^{ad}_{0}=\mathbf{Y}^t_0$ or~$Y^{ad}_{0}=\mathbf{Y}^v_0$ and consider the relative difference between the averaged objective functional values:
\begin{align*}
\operatorname{Err}_{\mathcal{J}} \coloneqq \frac{\sum_{Y_0 \in Y_{ad}} J(Y_\theta (Y_0) , U_\theta (Y_0))-\sum_{Y_0 \in Y_{ad}} J(\bar Y (Y_0) ,\bar U (Y_0)) }{\sum_{Y_0 \in Y_{ad}} J(\bar Y (Y_0) ,\bar U (Y_0))}
\end{align*}
as well as the associated normalized mean squared error of~$J(Y_\theta (\cdot) , U_\theta (\cdot))$:
\begin{align*}
\operatorname{Err}_{J} \coloneqq \frac{\sum_{Y_0 \in Y_{ad}} (J(Y_\theta (Y_0) , U_\theta (Y_0))- J(\bar Y (Y_0) ,\bar U (Y_0)))^2 }{\sum_{Y_0 \in Y_{ad}} J(\bar Y (Y_0) ,\bar U (Y_0))^2}.
\end{align*}
The normalized mean-squared errors of the state, $\operatorname{Err}_{Y}$, adjoint, $\operatorname{Err}_{P}$, and of the control,~$\operatorname{Err}_{U}$, are defined analogously. Moreover, to quantify the influence of the penalty parameters~$\gamma_1$ and~$\gamma_2$, we define
\begin{align*}
\operatorname{Err}_{V} \coloneqq \frac{\sum_{Y_0 \in Y_{ad}} \int_I |V^\eps_\theta (t,Y_\theta(Y_0)(t))- J_t(Y_\theta(Y_0),U_\theta(Y_0)(t))|^2 ~\mathrm{d}t  }{\sum_{Y_0 \in Y_{ad}} \int_I | J_t(Y_\theta(Y_0),U_\theta(Y_0)(t))|^2 ~\mathrm{d}t}.
\end{align*}
as well as
\begin{align*}
\operatorname{Err}_{\partial V} \coloneqq \frac{\sum_{Y_0 \in Y_{ad}} \int_I |\partial_y V^\eps_\theta (t,Y_\theta(Y_0)(t))- P_\theta (Y_0)(t)|^2 ~\mathrm{d}t  }{\sum_{Y_0 \in Y_{ad}} \int_I | P_\theta (Y_0)(t)|^2 ~\mathrm{d}t}.
\end{align*}
For $Y^{ad}_0= Y^t_0$, these terms correspond to the relative sizes of the additional penalties in~\eqref{eq:learningprobfinite}.
Finally, we also want to compare~$V^\eps_\theta$ with the optimal value function~$V^*$. Of course,~$V^*$ can neither be given analytically nor can it be computed exactly. As a remedy, we recall that if~$V^*$ is sufficiently regular and~$(\bar{Y}(Y_0), \bar{U}(Y_0) )$ is a minimizing pair of~\eqref{def:discreteexample} with adjoint state~$\bar{P}(Y_0)$, we have
\begin{align*}
V^*(t, \bar{Y}(Y_0)(t))= J_t(\bar{Y}(Y_0),\bar{U}(Y_0)) \quad \text{as well as} \quad \partial_y V^*(t, \bar{Y}(Y_0)(t))= \bar{P}(Y_0)(t)
\end{align*}
for all~$t \in I$. As a consequence, setting
\begin{align*}
d (V^*, V^\eps_\theta)=\frac{\sum_{Y_0 \in Y_{ad}} \int^T_0 | V^\eps_\theta(t, \bar{Y}(Y_0)(t))- J_t(\bar{Y}(Y_0),\bar{U}(Y_0)) |^2 ~\mathrm{d}t}{\sum_{Y_0 \in Y_{ad}} \int^T_0 |J_t(\bar{Y}(Y_0),\bar{U}(Y_0)) |^2 ~\mathrm{d}t}.
\end{align*}
as well as
\begin{align*}
d (\partial V^*, \partial V^\eps_\theta)=\frac{\sum_{Y_0 \in Y_{ad}} \int^T_0 | \partial_y V^\eps_\theta(t, \bar{Y}(Y_0)(t))- \bar{P}(Y_0)(t) |^2 ~\mathrm{d}t}{\sum_{Y_0 \in Y_{ad}} \int^T_0 |\bar{P}(Y_0)(t) |^2 ~\mathrm{d}t}.
\end{align*}
provides a suitable ``distance'' for the comparison of~$V^*$ and~$V^\theta_\eps$.
\subsection{Validation results}
As a concrete example, we set~$T=2,~\beta=0.01,~\alpha=0.25 $ and~$\mathcal{Y}_d(t,x)= x^2/10$, i.e., we try to steer the system towards a parabola. Note that there is no control input~$u \in L^2(I;\R^3)$ such that the corresponding solution~$\mathcal{Y}$ of the PDE~\eqref{eq:bilinearPDE} satisfies~$\mathcal{Y}(t)= \mathcal{Y}_d$. The parabolic binlinear control problem is approximated using~$n=10$ eigenfunctions. All computations were carried out in Matlab 2019 on a notebook with~$32$ GB RAM and an Intel\textregistered Core\texttrademark ~i7-10870H CPU@2.20 GHz.

In order to compute an approximately optimal feedback law for this problem, we solve~\eqref{eq:learningprobfinite} for various penalty parameter configurations~$\gamma_1, \gamma_2 \in \{0,0.1,1\}$. The resulting normalized errors can be found in Table~\ref{tab:val1}, for~$Y^{ad}_0=\mathbf{Y}^t_0$, and Table~\ref{tab:val2}, for~$Y^{ad}_0=\mathbf{Y}^v_0$. Comparing their individual entries, we observe that there is (almost) no difference in performance between the training and the validation sets. This means that, while the utilized  networks are rather simple and only comprise a small number of trainable parameters, the corresponding learned feedback controls generalize well to initial conditions which are not contained in the training set.

Indeed, on the one hand \emph{all} computed networks provide feedback controls which perform similarly to their open loop counterparts. This is manifested in very small averaged errors  for the objective functional, i.e.~$\operatorname{Err}_{\mathcal{J}}$ and $\operatorname{Err}_{{J}}$, the states and adjoint states,~$\operatorname{Err}_{{Y}}$ and $\operatorname{Err}_{{P}}$, as well as the controls,~$\operatorname{Err}_{{U}}$. These start to (slowly) deteriorate as~$\gamma_1$ and/or~$\gamma_2$ grow. However, cf. the explanation in Section~\ref{sec:learnfeedback}, this is expected: For~$\gamma_1>0$ and/or~$\gamma_2 >0$, the learned feedback has to strike a balance between minimizing~$J(Y_\theta (\cdot), U_\theta(\cdot))$ and keeping the penalty terms small, hence the slightly larger error.

On the other hand, the picture looks different once we consider the errors associated to the approximation of the value function, i.e.,~$\operatorname{Err}_{V}$,~$\operatorname{Err}_{\partial V}$ as well as $d(\partial V^*, \partial V^\eps_\theta)$ and $d (\partial V^*, \partial V^\eps_\theta)$.  Here $\gamma_1>0$ and/or~$\gamma_2 >0$ have a significant influence on $d( V^*,  V^\eps_\theta)$ and~$d(\partial V^*, \partial V^\eps_\theta)$ while
 the other normalized mean squared errors remain relatively small. Moreover, we have~$\operatorname{Err}_{V} \approx d ( V^*,  V^\eps_\theta) $ and~$\operatorname{Err}_{\partial V} \approx d (\partial V^*, \partial V^\eps_\theta) $ on the test as well as on the validation set. Hence, large values for these terms are a reliable indicator for structural differences between~$V^\eps_\theta$ and~$V^*$ and/or~$\partial_y V^\eps_\theta$ and~$ \partial_y V^*$, respectively.

Now, while~$\gamma_1=\gamma_2=0$ provides a very good approximation to the open loop optimal control, it performs the worst in terms of approximating the optimal value function and its derivative. This is related to two observations. First, in this case, the learning problem~\eqref{eq:learningprobfinite} only depends on the derivative~$\partial_y V^\theta_\eps$ but \emph{not} on the value function~$V^\eps_\theta$. Since primitives are not unique, approximating~$V^*$ by~$V^\eps_\theta$ is unlikely.
Second, due to the absence of~$V^\eps_\theta$ in the problem, some of the parameters in the model are not trainable. In fact, for $\gamma_1=\gamma_2=0$, there holds~$\partial_{W_{12}} \mathcal{J}_N(\theta)=0$ for every admissible~$\theta$.

Once we increase~$\gamma_1$ and~$\gamma_2$, this is no longer the case. Hence, we observe rapid decrease for $d( V^*, V^\eps_\theta)$ and $d(\partial V^*,\partial V^\eps_\theta)$. Most remarkably, the improvement for both is, to some extend, already visible for~$\gamma_1>0$ and~$\gamma_2=0$. In this setting, applying the gradient method neither requires computing the adjoint state~$P$ nor the costate~$K$ which limits the cost of every gradient step to~$2N=60$ ODE solve. Quite the contrary, increasing~$\gamma_2 >0$ but keeping~$\gamma_1=0$ fixed, there is \emph{no} improvement for $d(\ V^*, V^\eps_\theta)$. This further backs up our reasoning given for the case of~$\gamma_1=\gamma_2=0$.

Consequently, the computed results indicate that the best balance between finding an optimal control and approximating the value function is achieved by a careful choice of~$\gamma_1,\gamma_2 >0$. Moreover, they highlight two important points: First, the presented learning approach indeed allows to compute semiglobal optimal feedback laws~$F^\eps_\theta$ for higher dimensional problems and, thus, to some extent, alleviates the curse of dimensionality. Second, incorporating additional terms into the learning problem penalizing the violation of the dynamic programming principles~\eqref{eq:dynamicalprog}, allows to compute a good approximation~$V^\eps_\theta$ of the optimal value function on the fly. As stated initially, the present example should be understood as a proof of concept and, following these first promising results, we believe that this approach to feedback learning deserves further investigations, both, from the theoretical and the numerical side. For example, it would be interesting to explore systematic ways of choosing the penalty parameters~$\gamma_1,\gamma_2$.
However, this goes beyond the scope of the current paper and is left for future work.
\begin{table}[!htb]
\centering
    \medskip
\begin{tabular}{  l c c c c c  }

\toprule
$\qquad \text{Penalty}$  & $\text{Err}_{\mathcal{J}}$ & $\text{Err}_{Y}$ &  $\text{Err}_{P}$ & $\text{Err}_{U}$\\
\midrule
$\gamma_1= 0, \gamma_2=0$ & $ 0.15 \%$ & $0.04  \%$& $ 0.12 \%$ & $ 2.4 \%$  \\
\midrule
$\gamma_1= 0.1, \gamma_2=0.1$ & $0.36 \%$ & $0.1 \%$ & $0.24 \%$ & $5.5 \%$
\\
\midrule
$\gamma_1= 0.1, \gamma_2=0$ & $0.29 \%$ & $0.1 \%$ & $0.85 \%$ & $ 4.4 \%$
\\
\midrule
$\gamma_1= 1, \gamma_2=1$ & $ 0.64 \%$ & $ 0.25 \%$& $1 \%$ & $8.65 \%$ \\
\midrule
$\gamma_1= 0, \gamma_2=1$ & $ 0.1 \%$ & $ 0.05 \%$& $0.26 \%$ & $2.1 \%$ \\
\bottomrule
\end{tabular}
\centering
    \medskip
\begin{tabular}{  l c c c c c  }

\toprule
$\qquad \text{Penalty}$  & $\text{Err}_{J}$ & $\text{Err}_{V}$ &  $\text{Err}_{\partial V}$ & $d(V^\eps_\theta;V^*)$ & $d(\partial_y V^\eps_\theta;\partial_y V^*)$ \\
\midrule
$\gamma_1= 0, \gamma_2=0$ & $ 0.0003 \%$ & $ 79 \%$& $33 \%$ & $78.8 \%$ & $33.5 \%$ \\
\midrule
$\gamma_1= 0.1, \gamma_2=0.1$ & $ 0.001 \%$ & $0.03 \%$ & $7.4 \%$ & $0.03 \%$ & $7 \%$
\\
\midrule
$\gamma_1= 0.1, \gamma_2=0$ & $ 0.001 \%$ & $0.02 \%$ & $12.5 \%$ & $0.02 \%$ & $12.1 \%$
\\ \midrule
$\gamma_1= 1, \gamma_2=1$ & $ 0.005 \%$ & $ 0.007 \%$& $4.5 \%$ & $0.01 \%$ & $3.5 \%$
\\ \midrule
$\gamma_1= 0, \gamma_2=1$ & $ 0.003 \%$ & $ 88.8 \%$& $6.4 \%$ & $88.5 \%$ & $6.4 \%$\\
\bottomrule
\end{tabular}

\caption{Results on training set i.e.~$Y^{ad}_0=\mathbf{Y}^t_0$.}
\label{tab:val1}
\end{table}

\begin{table}[!htb]
\centering
    \medskip
\begin{tabular}{  l c c c c c  }

\toprule
$\qquad \text{Penalty}$  & $\text{Err}_{\mathcal{J}}$ & $\text{Err}_{Y}$ &  $\text{Err}_{P}$ & $\text{Err}_{U}$ \\
\midrule
$\gamma_1= 0, \gamma_2=0$ & $ 0.23 \%$ & $0.06  \%$& $ 0.57 \%$ & $ 4.42 \%$ \\
\midrule
$\gamma_1= 0.1, \gamma_2=0.1$ & $0.51 \%$ & $0.15 \%$ & $1.1 \%$ & $9.2 \%$
 \\ \midrule
$\gamma_1= 0.1, \gamma_2=0$ & $0.47 \%$ & $0.16 \%$ & $2.8 \%$ & $8.7 \%$
 \\ \midrule
$\gamma_1= 1, \gamma_2=1$ & $ 0.85 \%$ & $ 0.35 \%$ & $ 6\%$ & $13.5 \%$
 \\ \midrule
$\gamma_1= 0, \gamma_2=1$ & $ 0.25 \%$ & $ 0.1  \%$ & $ 1.3\%$ & $4.8 \%$\\
\bottomrule
\end{tabular}
\centering
    \medskip
\begin{tabular}{  l c c c c c  }

\toprule
$\qquad \text{Penalty}$  & $\text{Err}_{J}$ & $\text{Err}_{V}$ &  $\text{Err}_{\partial V}$ & $d(V^\eps_\theta;V^*)$ & $d( \partial_y V^\eps_\theta;\partial_y V^*)$ \\
\midrule
$\gamma_1= 0, \gamma_2=0$ & $ 0.002 \%$ & $ 78.7 \%$& $33.6 \%$ & $78.6 \%$ & $33.9 \%$ \\
\midrule
$\gamma_1= 0.1, \gamma_2=0.1$ & $0.007 \%$ &  $0.03 \%$ & $8.9 \%$ & $0.03 \%$  & $7.9 \%$ \\
\midrule
$\gamma_1= 0.1, \gamma_2=0$ & $0.008 \%$ &  $0.02 \%$ & $15.1 \%$ & $0.02  \%$  & $13.1 \%$ \\ \midrule
$\gamma_1= 1, \gamma_2=1$ & $ 0.02 \%$ & $0.009 \%$ & $ 9.8 \%$ & $ 0.01 \%$ & $ 4.2 \%$ \\
\midrule
$\gamma_1= 0, \gamma_2=1$ & $ 0.002 \%$ & $ 88 \%$ & $ 8.6 \%$ & $ 88 \%$ & $ 7.1 \%$ \\
\bottomrule
\end{tabular}
\caption{Results on validation set i.e.~$Y^{ad}_0=\mathbf{Y}^v_0$.}
\label{tab:val2}
\end{table}
\appendix
\section{Condition \eqref{eq:kk10}} \label{app1}

Here we address condition \eqref{eq:kk10}.
Define  $N_\epsilon=\lceil \frac{2\widehat M}{\eps} \rceil,\,\tilde M:=\eps N_\eps$  and introduce the equidistant grid  $G=\{-\tilde M, (1-N_\eps)\eps, \dots,-\eps,0,\eps, \dots,(N_\eps-1)\eps, \tilde M\}$. Next endow the hypercube $[-\tilde M, \tilde M]^{n+1}$ with the $(n+1)-$ dimensional product of the grid $G$. These grid points define
$\{Q_i\}_{i=1}^{(2N_\eps)^{n+1}}$ closed subhypercubes of dimension $\eps^n$ whose union covers $\bar K=[0,T] \times \bar B_{2\tilde M}(0)$.

We extend this  $n+1$-dimensional grid by adding $k\ge \lceil\frac{1}{2}\sqrt{n}\rceil+1$ layers (again all of dimension $\eps^n)$, to the surfaces of the preexisting grid, resulting in $\tilde N_\eps=(2N_\eps +2k)^{n+1}$ hypercubes whose union covers $ [- \tilde M-k\eps,\tilde M+k\eps]^{n+1}$. The subhypercubes are ordered in such a manner that the interiors ones $\{Q_i\}_{i=1}^{(2N_\eps +2(k-1))^{n+1}}$ are assembled first and the ones with a boundary face $\{Q_i\}_{i=(2N_\eps +2(k-1)+1)^{n+1}}^{(2N_\eps +2k)^{n+1}}$ come last. The set of indices corresponding to interior hypercubes are denoted by $\mathcal{I}$, those to boundary hypercubes by $\mathcal{F}$.

Next we introduce a staggered grid and place a node $x_i=(t_i,y_i)$ at the barycenter of each of the $Q_i,\, i=1, \dots,(2N_\eps +2k)^{n+1}$. We shall use the standard mollifier of radius $r_\eps$ defined by
\begin{equation*}
\psi(x)= \left\{
              \begin{array}{ll}
                \exp(\frac{1}{|\frac{x}{r_\eps}|^2-1}), & \text{ for} |x|\le r_\eps \\[1.4ex]
                0, & \text{ for} |x|\le r_\eps,
              \end{array}
            \right.
\end{equation*}
where $r_\eps=\eps(\frac{1}{2}\sqrt{n}+.1)$. Note that by adding $.1$ in the previous expression the cube $[-\frac{\eps}{2},\frac{\eps}{2}]^n$ is contained in the interior of the support of $psi$.   Finally we introduce
$\psi_j(x)= \psi(x-x_j)$, for ${j\in \mathcal{I}\cup \mathcal{F}}$ and
\begin{equation*}
\varphi_j= \frac{\psi_j}{\sum_{i\in \mathcal{I}\cup \mathcal{F}} \psi_i}, \text{ for } j \in \mathcal{I}.
\end{equation*}
Let us deduce the following properties:
\begin{itemize}
\item[(i)] For each ${j\in \mathcal{I}\cup \mathcal{F}}$ we have $supp\, \psi_j=\bar K_j$ where $K_j = \{x: |x-x_j|< r_\eps\}$.
\item[(ii)] By construction there exists $\mathfrak{m} $ such that    $\text{card} \{j: \psi_j(x) \neq 0 \} \le \mathfrak{m}, \quad  \forall x \in \R^{n+1},  \text{ and } \forall {j\in \mathcal{I}\cup \mathcal{F}}$, for each $\eps\in (0,\eps_0]$.
\item[(iii)] For each $j\in \mathcal{I}$ the denominator in the definition of $\varphi_j$ is different from zero. Hence $\varphi_j$ is well-defined
with $supp\, \varphi_j=\bar K_j$ for  $j\in \mathcal{I}$ and $\varphi_j \colon\R^{n+1} \to [0,1]$, and it is  $\mathcal{C}^\infty$ smooth.
\item[(iv)] $\bar{K} \subset \bigcup_{j\in \mathcal{I}} Q_j\subset \bigcup_{j\in \mathcal{I}} K_j$.
\item[(v)]     $\text{card} \{j: \varphi_j(x) \neq 0 \} \le \mathfrak{m}, \quad  \forall x \in \R^{n+1},  \text{ and } \forall {j\in \mathcal{I}}$, for each $\eps\in (0,\eps_0]$.
\item[(vi)] Due to the choice of $r_\eps$ the functions $\psi$ are uniformly bounded from below on $Q_j$ for each $j$, independent of $\eps\in (0,\eps_0]$. Moreover due to the boundedness of $\psi$ and by the definition of $\mathfrak{m}$, there exists $\nu>0$ such that
\begin{equation*}
\sum_{i\in \mathcal{I}\cup \mathcal{F}} \psi_i(x) \ge \nu , \forall x \in Q_j, \text{ with } j \in \mathcal{I}\cup \mathcal{F},
\end{equation*}
and thus in particular $\sum_{i\in \mathcal{I}\cup \mathcal{F}} \psi_i(x) \ge \nu , \forall x \in \bar K$.
\item[(vii)] $$ supp\; \varphi_j \cap \bar K = \emptyset\; \forall j \in \mathcal{F}. $$
This is a consequence of the fact that for $j \in \mathcal{F}$ we have
$dist(x_j, \partial ([-\tilde M, \tilde M]^n))= \eps[(k-1)+\frac{1}{2}]$ and thus
$dist(\partial K_j, \partial ([-\tilde M, \tilde M]^n)) \le \eps[(k-1)+\frac{1}{2}-r_\eps] = k-\frac{1}{2}(1+\sqrt{n})-.1>\eps (k-1 - \frac{1}{2}\sqrt{n})>0$.
\item[(viii)] $\sum_{i\in \mathcal{I}}\varphi_i =1, \; \forall x \in \bar K$. This is a consequence of $(vii)$ and the definition of $\varphi_j$.
\item[(ix)]
$\|D^{j}\varphi_i\|_{C(\bar K_i\cap \bar K)} \le\bar \mu \eps^{-j}, \text{ for some } \bar \mu \text{ independent of }  i\in \mathcal{I}, \; \text{ and } j\in\{1,2\}.$
\end{itemize}
Once we have verified $(ix)$, all the properties demanded in
\eqref{eq:kk10} on the partition of unity $\{\varphi_i\}_{i\in \mathcal{I}}$ subordinate to~$K_i$ will be satisfied.

In the following calculations we repeatedly use that $ \nabla \sum _{i\in \mathcal{I}}\varphi_i(x)=0$ for $x\in \bar K$. This follows from $(viii)$. As short calculation shows that for each $j\in \mathcal{I}$ , each $x\in \bar K$, and $k,\ell \in \{1,\dots,n\}$
\begin{equation*}
\partial_{x_k} \varphi_j(x) = \frac{\partial_{x_k} \psi_j(x)}{\sum_{i\in \mathcal{I}} \psi_i(x)}, \quad \partial_{x_\ell} \partial_{x_k} \varphi_j(x)= \frac{\partial_{x_\ell}  \partial_{x_k} \varphi_j(x)  \sum_{i\in \mathcal{I}} \psi_i(x) - \psi_j \sum_{i\in \mathcal{I}} \partial_{x_\ell}  \partial_{x_k} \psi_i(x)     }{(\sum_{i\in \mathcal{I}} \psi_i(x))^2  },
\end{equation*}
where we use that $\partial_{x_k}\sum_{i\in \mathcal{I}} \psi_i(x)=0$ for $x\in \bar K$.

To obtain the required estimates we introduce for $\eta>0$
\begin{equation*}
\psi_\eta(x)= \left\{
              \begin{array}{ll}
                \exp(\frac{1}{|\frac{x}{\eta}|^2-1}), & \text{ for} |x|\le \eta \\[1.4ex]
                0, & \text{ for} |x|\le \eta.
              \end{array}
            \right.
\end{equation*}
Then we have
\begin{equation*}
\partial_{x_k} \psi_\eta=- \psi_\eta \, \frac{2x_k}{\eta^2 (|\frac{x}{\eta}|^2-1)^2},
\end{equation*}
\begin{equation*}
\begin{array}l
(\partial_{x_k})^2 \psi_\eta=
\frac{2\psi_\eta}{\eta^2(|\frac{x}{\eta}|^2-1)^4}\big[\frac{2x_k^2}{\eta^2}-(|\frac{x}{\eta}|^2-1)^2 +\frac{4x_k^2}{\eta^2}(|\frac{x}{\eta}|^2-1)   \big],\\[1.5ex]
\end{array}
\end{equation*}
and for $k\ne \ell$
\begin{equation*}
\partial_{x_\ell} \partial_{x_k}\psi_\eta=
\frac{2\psi_\eta\,x_\ell x_k}{\eta^2(|\frac{x}{\eta}|^2-1)^4} \big[ \frac{2}{\eta^2} + \frac{4}{\eta^2}(|\frac{x}{\eta}|^2 -1))\big]=
\frac{2\psi_\eta\, x_\ell x_k}{\eta^2(|\frac{x}{\eta}|^2-1)^4} \big[ \frac{-2}{\eta^2} + \frac{4}{\eta^4}|x|^2 \big].
\end{equation*}
Considering the behavior of $\partial_{x_k} \psi_\eta$ and $\partial_{x_\ell} \partial_{x_k}\psi_\eta$ separately on the ball $B_{\frac{\eta}{2}}(0)$ and its complement in
$B_{{\eta}}(0)$, it  follows that these functions behave like $O(\frac{1}{\eta})$ and  $O(\frac{1}{\eta^2})$. Applying these estimates in the
expressions for  the first and second derivatives for $\varphi_j$ and using the lower bound established in $(vi)$ we obtain $(ix)$.

\section{Perturbation results} \label{app:perturbation}
Here we collect pertinent existence and stability results for  dynamical systems. The constant  $\widehat{M}(0)$ appearing  below relates to Assumption $\textbf{A.2}$.
\begin{prop} \label{prop:s}
Let~$\mathbf{y} \in \mathcal{C}(Y_0;W_T)$,~$\mathbf{y}(y_0) \in \mathcal{Y}_{ad}$ for all~$y_0 \in Y_0$,~$\delta \mathbf{v} \in \mathcal{C}(Y_0;L^2(I;\R^n))$, as well as~$\delta \mathbf{y}_0 \in \mathcal{C}(Y_0;\R^n)$ be given. Moreover let~$A \colon I \times \R^n \to \R^{n \times n} $ be  continuous, and denote by~$\mathbf{A} \colon L^\infty(I;\R^n) \to \mathcal{B}(L^2(I;\R^n))$ the induced Nemitsky operator i.e.
\begin{align*}
\mathbf{A}(y) \delta y = A(t,y(t)) \delta y(t) \quad \forall \delta y \in L^2(I;\R^n),~y \in L^\infty(I;\R^n)
\end{align*}
and a.e.~$t \in I$. Then there is~$\delta \mathbf{y} \in \mathcal{C}(Y_0;W_T)$ such that~
\begin{align} \label{eq:linearensemble}
\dot{\delta y}= \mathbf{A}(y)\delta y+v,~\delta y (0)= \delta y_0
\end{align}
for~$y\coloneqq \mathbf{y}(y_0)$,~$\delta v\coloneqq \delta \mathbf{v}(y_0) $,~$\delta y_0 \coloneqq \delta \mathbf{y}(y_0)$ and all~$y_0 \in Y_0$. It satisfies
\begin{align} \label{eq:aprioriindependent}
\|\delta \mathbf{y}(y_0)\|_{W_T} \leq C \left( \|\delta \mathbf{v}(y_0)\|_{L^2}+|\delta \mathbf{y}_0(y_0)| \right)
\end{align}
for some~$C>0$ depending continuously on $\max_{(\tau,y)\in I \times \bar{B}_{2 \widehat{M}(0)} }\|A(\tau,y)\|_{\R^{n \times n}}$, and  independent of $y_0\in Y_0$.
\end{prop}
\begin{proof}
Let~$y_0 \in Y_0$ be arbitrary but fixed. Then there is a unique solution~$\delta y \in W_T$ to~\eqref{eq:linearensemble} which satisfies
\begin{align*}
\frac{1}{2} |\delta y (t)|^2 &= \frac{1}{2} |\delta y (0)|^2+ \int^t_0 (\dot{\delta y}(s),\delta (y)(s))  ~\mathrm{d}s \\
&= \frac{1}{2} |\delta y_0|^2+ \int^t_0 (\delta y(s), A(t,y(t))\delta (y)(s))+ (\delta v(s), \delta y(s))  ~\mathrm{d}s \\
& \leq \frac{1}{2} |\delta y_0|^2+ \frac{1}{2} |\delta v|^2_{L^2}+ \frac{1}{2}\int^t_0 \left(2 \max_{(\tau,y)\in I \times \bar{B}_{2 \widehat{M}(0)} }\|A(\tau,y)\|_{\R^{n \times n}} +1\right) |\delta y(s)|^2~\mathrm{d}s
\end{align*}
for all~$t \in I$. Setting
\begin{align*}
L \coloneqq \left(2 \max_{(z,y)\in I \times \bar{B}_{2 \widehat{M}(0)} }\|A(z,y)\|_{\R^{n \times n}} +1\right),
\end{align*}
Gronwall's inequality implies that
\begin{align*}
\|\delta y\|_{L^\infty} \leq e^{TL}(|\delta y_0|+ |\delta v|_{L^2}).
\end{align*}
By~\eqref{eq:linearensemble} we further get
$
\|\dot{\delta y}\|_{L^2} \leq L (\|\delta y\|_{L^2} + \|v\|_{L^2}),
$
which implies ~\eqref{eq:aprioriindependent}.
Next, let~$y^k_0 \in Y_0$ denote a convergent sequence with limit~$y_0$. For abbreviation set
\begin{align*}
\delta y_k \coloneqq \delta \mathbf{y}(y^k_0),~y_k\coloneqq \mathbf{y}(y^k_0),~\delta v_k \coloneqq \delta \mathbf{v}(y^k_0),~\delta y_0 \coloneqq \delta \mathbf{y}_0(y^k_0)
\end{align*}
as well as
\begin{align*}
y\coloneqq \mathbf{y}(y_0),~\delta v \coloneqq \delta \mathbf{v}(y_0),~\delta y_0 \coloneqq \delta \mathbf{y}_0(y_0).
\end{align*}
Note that~$\delta y_k$ is uniformly bounded in~$W_T$ by \eqref{eq:aprioriindependent}. Thus it admits a subsequence, denoted by the same index, with~$\delta y_k \rightharpoonup \delta y$ in~$W_T$ for some~$\delta y \in W_T$. This implies
\begin{align*}
\delta y_k (0) \rightarrow \delta y(0)~\text{in}~\R^n,~ \delta y_k  \rightarrow \delta y~\text{in}~L^\infty(I;\R^n),~\dot{\delta y}_k  \rightharpoonup \dot{\delta y}~\text{in}~L^2(I;\R^n).
\end{align*}
Moreover, due to the continuity of~$\mathbf{y},\delta \mathbf{v}$ and~$\delta \mathbf{y}_0$ , we get
\begin{align*}
\dot{\delta y}_k=\mathbf{A}(y_k)\delta y_k+ \delta v_k \rightarrow \mathbf{A}(y)\delta y+ \delta v~\text{in}~L^2(I;\R^n),~\delta y_k (0) \rightarrow \delta y_0~\text{in}~\R^n.
\end{align*}
Summarizing the previous observations we conclude that
\begin{align*}
\dot{\delta y}= \mathbf{A}(y)\delta y+\delta v,~\delta y(0)= \delta y_0
\end{align*}
as well as~$\delta y_k \rightarrow \delta y$ in~$W_T$,
and thus~$\delta y= \delta \mathbf{y}(y_0) $. By uniqueness of solutions to the above equation ~$\delta \mathbf{y}(y^k_0)\rightarrow \delta \mathbf{y}(y_0) $ for the whole sequence in~$W_T$ follows,  and therefore~$\delta \mathbf{y} \in \mathcal{C}(Y_0;W_T)$.
\end{proof}

Next we address nonlinear systems of the form:
\begin{align}
\label{eq:pertstate}
\dot{y_v}= \mathbf{f}(y_v)+ \mathbf{g}(y_v) \mathcal{F}^*(y_v)+v, \quad
y_v(0)=y_0
\end{align}
where~$v \in L^2(I ; \R^n)$ is a  perturbation.

\begin{prop} \label{thm:existspert}
Let Assumption~\ref{ass:feedbacklaw} hold.
Then there exist an open neighbourhood~$V_1 \subset L^2(I; \R^n)$ of~$0$
and an open neighbourhood~$\mathbf{Y}_0$ of $Y_0$ such
that~\eqref{eq:pertstate} admits a unique solution~$y_v=\mathbf{y}^v
(y_0)\in \mathcal{Y}_{ad} $ for every pair~$(v,y_0)\in V_1 \times
\mathbf{Y}_0$. Moreover the mapping
\begin{align}
\mathbf{y}^\bullet(\bcdot) \colon V_1  \times \mathbf{Y}_0 \to
\mathcal{Y}_{ad}   , \quad (v, y_0) \mapsto \mathbf{y}^v(y_0)
\end{align}
is continuously Fr\'echet differentiable.
\end{prop}
\begin{proof}
Define the mapping
\begin{align*}
G \colon \mathcal{Y}_{ad} \times \R^n \times L^2(I; \R^n)
\to L^2(I; \R^n) \times \R^n
\end{align*}
with
\begin{align*}
G(y,y_0,v)= \left(
\begin{array}{c}
\dot{y}-\mathbf{f}(y)-\mathbf{g}(y)\mathcal{F}^*(y)-v\\
y(0)-y_0\\
\end{array}
\right).
\end{align*}
Now fix an arbitrary~$\bar{y}_0 \in Y_0$ and, utilizing $(\mathbf{A.3})$   denote by $\bar{y}=\mathbf{y}^*(y_0)\in \operatorname{int}
\mathcal{Y}_{ad}$ the unique solution in~$\mathcal{Y}_{ad}$ to the unperturbed closed loop system~$G(\bar{y},\bar{y}_0,0)=0$. Since~$G$ is  of
class~$\mathcal{C}^1$ in a neighborhood of~$(\bar{y},\bar{y}_0,0)$ we have
\begin{align*}
D_y G(y,y_0,v)\delta y = \left(
\begin{array}{c}
\dot{\delta y}-D\mathbf{f}(y)\delta y-\lbrack D\mathbf{g}(y) \delta y \rbrack
\mathcal{F}^*(y) - \mathbf{g}(y) \partial_y \mathcal{F}^*(y)\delta y\\
\delta y(0)\\
\end{array}
\right).
\end{align*}
It is straightforward that the linearized equation
\begin{align*}
D_y G(\bar{y},\bar{y}_0,v)\delta y=\left(
\begin{array}{c}
\delta v \\ \delta y_0\\
\end{array}
\right)
\end{align*}
admits a unique solution~$\delta \bar{y}\in W_T$ for every~$\delta v \in L^2(I;\R^n),~\delta y_0 \in \R^n$. Moreover, applying Gronwall's lemma yields~$c>0$ independent of~$\bar{y},~\bar{y}_0$ with
\begin{align*}
\wnorm{\delta \bar{y}} \leq c(\|\delta v\|_{L^2(I;\R^n)}+|\delta y_0|), \quad \forall \delta v \in L^2(I;\R^n),~\delta y_0 \in \R^n.
\end{align*}
Thus from the implicit function theorem we get
constants~$\kappa_1=\kappa_1(\bar{y}_0)$
and~$\kappa_2=\kappa_2(\bar{y}_0)$, such that for every~$y_0\in \R^n$
with~$|y_0-\bar{y}_0|<\kappa_1$ and~$\|v\|_{L^2(I;\R^n)}< \kappa_2$ there
exists~$\mathbf{y}^v(y_0) \in \mathcal{Y}_{ad}$
with~$G(\mathbf{y}^v(y_0),y_0,v)=0$.  By $(\textbf{A.1})$  it is the
unique solution to~\eqref{eq:pertstate} in~$\mathcal{Y}_{ad}$. Moreover, the mapping
\begin{align*}
\mathbf{y}^\bullet(\cdot) \colon  B_{\kappa_2}(0) \times
B_{\kappa_1}(\bar{y}_0) \to \mathcal{Y}, \quad (v,y_0) \mapsto
\mathbf{y}^v(y_0)
\end{align*}
is of class~$\mathcal{C}^1$.
Observe that repeating this argument for every~$y_0 \in Y_0$ yields an open covering of~$Y_0$ i.e.
\begin{align*}
Y_0 \subset \bigcup_{\bar{y}_0 \in Y_0}
B_{\kappa_1(\bar{y}_0)}(\bar{y}_0).
\end{align*}
Since~${Y}_0$ is compact there exists a finite set of initial
conditions~$\{\bar{y}^i_0\}^N_{i=1}\subset Y_0$, including~$0$, such that
\begin{align*}
Y_0 \subset \mathbf{Y}_0 :=\bigcup^N_{i=1}
B_{\kappa_1(\bar{y}^i_0)}(\bar{y}^i_0).
\end{align*}
Set~$V= \bigcap^N_{i=1} B_{\kappa_2(\bar{y}^i_0)}(0) \subset
L^2(I;\R^n) $. Summarizing these arguments  yields the
existence of a~$\mathcal{C}^1$-mapping
\begin{align*}
\mathbf{y}^\cdot(\cdot) \colon V \times \mathbf{Y}_0 \to
\mathcal{Y}_{ad}, \quad \mathbf{y}^v(y_0)~\text{ uniquely
solves}~\eqref{eq:pertstate}~\text{in}~\mathcal{Y}_{ad}.
\end{align*}
\end{proof}
We use the following consequences of the previous proposition.
\begin{coroll} \label{coroll:locallipschitzofstate}
There exists an open neighborhood~$V_2 \subset V_1 \subset L^2(I;\R^n)$
of~$0$ as well as~$c>0$ such that
\begin{align*}
\wnorm{\mathbf{y}^{v_1}(y_0)-\mathbf{y}^{v_2}(y_0)} \leq c
\|v_1-v_2\|_{L^2(I;\R^n)} \quad \forall y_0 \in Y_0,~v_1 \in V_2, v_2 \in V_2
\end{align*}
and
\begin{align*}
\wnorm{\mathbf{y}^v(y_0)} \leq M_{Y_0} +c \|v\|_{L^2(I;\R^n)} \quad
\forall y_0 \in Y_0,~v \in V_2,
\end{align*}
hold. Here~$M_{Y_0}$ denotes the constant from Assumption $(\mathbf{A.3})$.
\end{coroll}

\begin{proof}
The first assertion follows from the continuous differentiability of $v\to \mathbf{y}^v(y_0)$ and compactness of $Y_0$.
To verify the  second we use that $\mathbf{y}^*(y_0) = \mathbf{y}^0(y_0)$ and estimate
\begin{align*}
\wnorm{\mathbf{y}^v(y_0)} &\leq \wnorm{\mathbf{y}^*(y_0)} +
\wnorm{\mathbf{y}^v(y_0)-\mathbf{y}^0(y_0)}.
\end{align*}
The claim now follows from the first inequality  and  $(\mathbf{A.3})$.
\end{proof}

\bibliographystyle{siam}
\bibliography{referencesnew}
\end{document}